\documentclass[reqno]{amsart}
\usepackage[left=1in,right=1in,top=1in,bottom=1in]{geometry}
\usepackage{tikz}
\usetikzlibrary{shapes,snakes,calc,arrows}
\usetikzlibrary{decorations.pathreplacing,shapes.geometric}
\usetikzlibrary{calc,positioning}
\usepackage{amsthm}
\usepackage{amscd}
\usepackage{amsfonts}
\usepackage{amsmath}
\usepackage{amssymb}
\usepackage{amsrefs}
\usepackage{mathrsfs}
\usepackage{enumitem}
\usepackage{multirow}
\usepackage{verbatim}
\usepackage{url}
\usepackage{graphicx}
\usepackage{yfonts}
\usepackage{color}

\usepackage{microtype}
\usepackage{hyperref}

\tikzset{Box/.style={very thick, rounded corners}}
\tikzset{marked/.style={star, star point height = .75mm, star points =5, fill=black,minimum size=2mm, inner sep=0mm} }
\tikzset{verythickline/.style = {line width=7pt}}
\tikzset{thickline/.style = {line width=5pt}}
\tikzset{medthick/.style = {line width=3pt}}
\tikzset{med/.style = {line width=2pt}}
\tikzset{count/.style = {fill=white,circle,draw,thin, inner sep=2pt}}
\tikzset{rcount/.style = {fill=white,rectangle,draw,thin,inner sep=2pt, rounded corners}}
\tikzset{cpr/.style = {draw,fill=white,rectangle,thin, rounded corners}}

\definecolor{ggreen}{HTML}{00BB33}

\begin{document}


\newcommand{\supp}{\text{supp}}
\newcommand{\Aut}{\text{Aut}}
\newcommand{\Gal}{\text{Gal}}
\newcommand{\Inn}{\text{Inn}}
\newcommand{\Irr}{\text{Irr}}
\newcommand{\Ker}{\text{Ker}}
\newcommand{\N}{\mathbb{N}}
\newcommand{\Z}{\mathbb{Z}}
\newcommand{\Q}{\mathbb{Q}}
\newcommand{\R}{\mathbb{R}}
\newcommand{\C}{\mathbb{C}}
\renewcommand{\H}{\mathcal{H}}
\newcommand{\B}{\mathcal{B}}
\newcommand{\A}{\mathcal{A}}
\newcommand{\Y}{\mathcal{Y}}
\newcommand{\X}{\mathcal{X}}
\newcommand{\K}{\mathcal{K}}
\newcommand{\M}{\mathcal{M}}
\newcommand{\e}{\epsilon}

\newcommand{\J}{\mathscr{J}}
\newcommand{\D}{\mathscr{D}}

\newcommand{\I}{\text{I}}
\newcommand{\II}{\text{II}}
\newcommand{\III}{\text{III}}

\newcommand{\<}{\left\langle}
\renewcommand{\>}{\right\rangle}
\renewcommand{\Re}[1]{\text{Re}\ #1}
\renewcommand{\Im}[1]{\text{Im}\ #1}
\newcommand{\dom}[1]{\text{dom}\,#1}
\renewcommand{\i}{\text{i}}
\renewcommand{\mod}[1]{(\operatorname{mod}#1)} 
\newcommand{\alg}{\operatorname{alg}}
\newcommand{\mb}[1]{\mathbb{#1}}
\newcommand{\mc}[1]{\mathcal{#1}}
\newcommand{\mf}[1]{\mathfrak{#1}}
\newcommand{\mr}{\mathrm}
\newcommand{\im}{\operatorname{im}}

\newcommand{\lat}{\mathrm{lat}}
\newcommand{\lrleq}{\leq_{\mathrm{lat}}}
\newcommand{\lrgeq}{\geq_{\mathrm{lat}}}

\newcommand{\paren}[1]{\left(#1\right)}
\newcommand{\set}[1]{\left\{#1\right\}}
\newcommand{\sq}[1]{\left[#1\right]}
\newcommand{\abs}[1]{\left|#1\right|}

\definecolor{ianeditcolour}{HTML}{2266FF}
\newenvironment{ian}{\color{ianeditcolour}}{}
\newenvironment{brent}{\color{red}}{}


\newtheorem{thm}{Theorem}[subsection]
\newtheorem{prop}[thm]{Proposition}
\newtheorem{lem}[thm]{Lemma}
\newtheorem{cor}[thm]{Corollary}
\newtheorem{innercthm}{Theorem}
\newenvironment{cthm}[1]
  {\renewcommand\theinnercthm{#1}\innercthm}
  {\endinnercthm}
\newtheorem{innerclem}{Lemma}
\newenvironment{clem}[1]
  {\renewcommand\theinnerclem{#1}\innerclem}
  {\endinnerclem}

\theoremstyle{definition}
\newtheorem{defn}[thm]{Definition}
\newtheorem{ex}[thm]{Example}
\newtheorem{nota}[thm]{Notation}
\newtheorem{exam}[thm]{Example}
\newtheorem{rem}[thm]{Remark}
\newtheorem{innercdefi}{Definition}
\newenvironment{cdefi}[1]
  {\renewcommand\theinnercdefi{#1}\innercdefi}
  {\endinnercdefi}
\newtheorem{cons}[thm]{Construction}


\title{Combinatorics of Bi-Freeness with Amalgamation}

\author{Ian Charlesworth}
\author{Brent Nelson}
\author{Paul Skoufranis}
\address{Department of Mathematics, UCLA, Los Angeles, California, 90095, USA}
\email{ilc@math.ucla.edu\\bnelson6@math.ucla.edu\\pskoufra@math.ucla.edu}
\thanks{This research was supported in part by NSF grants DMS-1161411, DMS-0838680 and by NSERC
PGS-6799-438440-2013.}

\begin{abstract}
In this paper, we develop the theory of bi-freeness in an amalgamated setting.
We construct the operator-valued bi-free cumulant functions, and show that the vanishing of mixed cumulants is necessary and sufficient for bi-free independence.
Further, we develop a multiplicative convolution for operator-valued random variables and explore ways to construct bi-free pairs of $B$-faces.
\end{abstract}

\maketitle


\section{Introduction}
\label{sec:intro}

In \cite{V2013-1}, Voiculescu introduced the notion of bi-free probability as a generalization of free probability to study of left and right actions of algebras on a reduced free product space simultaneously.
Voiculescu demonstrated that many results from free probability have direct analogues in the bi-free setting.
In particular, \cite{V2013-1} demonstrated the existence of bi-free cumulant polynomials an analogue to free cumulants, although it did not produce an explicit formula.

Mastnak and Nica defined the $(\ell, r)$-cumulant functions in \cite{MN2013}, via permutations applied to non-crossing partitions.  In addition, they hypothesized that the $(\ell, r)$-cumulants were the correct cumulant functions for bi-free probability and gave rise to the bi-free cumulant polynomials of Voiculescu.
In \cite{CNS2014}, the authors of this paper proved this was the case via a diagrammatic argument.
In addition, we constructed an intuitive multiplicative convolution and a complicated operator model on a Fock space for a pair of faces.

In \cite{V2013-1}*{Section 8}, Voiculescu laid the framework for generalizing bi-free probability to an amalgamated setting.  As the combinatorics of free probability with amalgamation are well-understood (c.f. \cites{NSS2002, NSBook, S1998}), the goal of this paper is to demonstrate that the results of \cite{CNS2014} generalize to an amalgamated setting and that the combinatorics of bi-freeness with amalgamation are natural extensions of the combinatorics of free probability with amalgamation.

This paper contains ten sections beyond this introduction, structured as follows.
Section \ref{sec:background} recalls the necessary background from \cite{CNS2014}.
In particular, the notion of bi-non-crossing partitions, their lateral refinements, the incident algebra for bi-non-crossing partitions, and the structure of the universal polynomials for moments of bi-free pairs of faces will be reviewed.

Section \ref{sec:bifreefamilieswithamalgamation} introduces the setting for bi-free probability with amalgamation from \cite{V2013-1}.
We define the notion of a $B$-$B$-non-commutative probability space $(\mathcal{A}, E, \varepsilon)$ (Definition \ref{defn:BBncps}) and demonstrate a representation of $\mathcal{A}$ as linear operators on a $B$-$B$-bimodule (see Theorem \ref{thm:representingbbncps}).
In addition, the notion of bi-free pairs of $B$-faces is reintroduced (see Definition \ref{defn:pairofBfaces}).  

Section \ref{sec:OperatorValuedBiMultiplicativeFunctions} introduces the notion of an operator-valued bi-multiplicative function (see Definition \ref{defnbimultiplicative}).
Such functions are extensions of multiplicative functions (see \cite{NSS2002}*{Section 2} or \cite{S1998}*{Section 2}) to the bi-free setting but have natural descriptions via multiplicative functions (see Remark \ref{remonbimultiplicative}).

Section \ref{sec:verifyingrecursivedefinitionfromuniversalpolynomialshasdesiredproperties} defines certain terms $E_\pi(T_1, \ldots, T_n)$ (see Definition \ref{defn:recursivedefinitionofEpi}) which appear when actions of pairs of $B$-faces are examined and give rise to an operator-valued bi-multiplicative moment function.
Unfortunately, Section \ref{sec:verifyingrecursivedefinitionfromuniversalpolynomialshasdesiredproperties} is fairly technical due to the lack of an analogue of centring techniques from free-probability.

Section \ref{sec:operatorvaluedbifreecumulants} defines the operator-valued bi-free cumulants (see Definition \ref{def:kappa}) as a convolution of the M\"{o}bius function for bi-non-crossing partitions with $E_\pi(T_1,\ldots, T_n)$ and demonstrates that they are bi-multiplicative (see Corollary \ref{cor:cumulantsarebimultiplicative}).
In addition, we show that the operator-valued bi-free cumulants posess a certain property analogous to the property of operator-valued free cumulants demonstrated in \cite{S1998}*{Section 3.2} (see Section \ref{subsec:bimomentandbicumulantfunctions}) and vanish when a left or right $B$-operator is input (see Proposition \ref{prop:inputaL_borR_bandyougetzerocumulants}).

Section \ref{sec:universalmomentpolysforbifreewithamalgamation} demonstrates through Theorem \ref{thm:bifreeequivalenttouniversalpolys} that a family of pairs of $B$-faces are bi-free with amalgamation over $B$ if and only if certain universal moment polynomials involving $E_\pi(T_1, \ldots, T_n)$ are satisfied.
The main result of this paper, Theorem \ref{thmequivalenceofbifreeandcombintoriallybifree}, then follows immediately in Section \ref{sec:additivity}: a family of pairs of $B$-faces are bi-free over $B$ if and only if their mixed operator-valued bi-free cumulants vanish.

Section \ref{sec:MultiplicativeConvolution} demonstrates how operator-valued bi-free cumulants involving products of operators may be computed.
Also, Proposition \ref{prop:multiplcative-convolution} (which holds only in the scalar setting) generalizes \cite{CNS2014}*{Theorem 5.2.1} from the multiplicative convolution for bi-free two-faced families with singletons in each left and each right face to an arbitrary number of operators in each face.

Finally, Section \ref{sec:howtogetexamplesofbifreefamilies} provides additional methods for constructing bi-free pairs of $B$-faces by showing that conjugating a pair of $B$-faces by a $B$-valued Haar bi-unitary produces a pair of $B$-faces bi-free from the original pair, and that the bi-freeness of pairs of $B$-faces where all left faces commute with all right faces is intrinsically related to the freeness of just the left or just the right faces.


\section{Background on Bi-Non-Crossing Partitions}
\label{sec:background}
We begin by reviewing many of the results from \cite{CNS2014}, which we will extend to the amalgamated setting.
Throughout this section, let $n \in \N$, $\chi : \set{1, \ldots, n} \to \set{\ell, r}$, and $\e : \{1,\ldots, n\} \to K$ for some fixed set $K$.

\subsection{Bi-Non-Crossing Partitions and Diagrams}
\label{sec:BiNonCrossingPartitionsAndDiagrams}
Recall that $\chi$ induces a permutation $s_\chi \in S_n$ corresponding to reading the left nodes in increasing order, followed by the right nodes in decreasing order: if $\chi^{-1}(\{\ell\}) = \set{i_1<\cdots<i_p}$ and $\chi^{-1}(\{r\}) = \set{i_{p+1} > \cdots > i_n}$, we set $s_\chi(k) = i_k$.
Recall also that the set of partitions on $n$ elements, $\mc P(n)$, is partially ordered by refinement: for $\pi, \sigma \in \mc P(n)$, we have $\pi \leq \sigma$ if and only if every block of $\pi$ is contained in a single block of $\sigma$.
Finally, let $\abs{\sigma}$ denote the number of blocks in $\sigma$.

\begin{defn}
A partition $\pi \in \mc{P}(n)$ is said to be \emph{bi-non-crossing} with respect to $\chi$ if the partition  $s_\chi^{-1}\cdot \pi \in NC(n)$ -- that is, the partition formed by applying $s_\chi^{-1}$ to the blocks of $\pi$ -- is non-crossing.
Equivalently, $\pi$ is bi-non-crossing if whenever there are blocks $U, V \in \pi$ with $u_1, u_2 \in U$ and $v_1, v_2 \in V$ such that $s_\chi^{-1}(u_1) < s_\chi^{-1}(v_1) < s_\chi^{-1}(u_2) < s_\chi^{-1}(v_2)$, we have $U = V$.
The set of bi-non-crossing partitions with respect to $\chi$ is denoted $BNC(\chi)$, and inherits a lattice structure from $\mc P(n)$.
We will use $0_{\chi}$ and $1_{\chi}$ to denote the minimal and maximal elements of $BNC(\chi)$, respectively.
\end{defn}

To each partition $\pi \in BNC(\chi)$ we can associate a ``bi-non-crossing diagram'' by placing nodes along two vertical lines, labelled $1$ to $n$ from top to bottom, such that the nodes on the left line correspond to those values for which $\chi(k) = \ell$ (similarly for the right), and connecting those nodes which are in the same block of $\pi$ in a non-crossing manner.

\begin{exam}
If $\chi^{-1}(\set l)=\{1,2,4\}$, $\chi^{-1}(\set r)=\{3,5\}$, and
	\begin{align*}
		\pi=\left\{\vphantom{\sum}\{1,3\}, \{2,4,5\} \right\}= s_\chi\cdot \left\{\vphantom{\sum} \{1,5\}, \{2,3,4\} \right\},
	\end{align*}
then the bi-non-crossing diagram associated to $\pi$ is
	
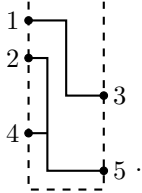
\begin{align*}
	\begin{tikzpicture}[baseline]
	\node[ left] at (-.5, 2) {1};
	\draw[fill=black] (-.5,2) circle (0.05);
	\node[left] at (-.5, 1.5) {2};
	\draw[fill=black] (-.5,1.5) circle (0.05);
	\node[right] at (.5, 1) {3};
	\draw[fill=black] (.5,1) circle (0.05);
	\node[left] at (-.5, .5) {4};
	\draw[fill=black] (-.5,.5) circle (0.05);
	\node[right] at (.5,0) {5};
	\draw[fill=black] (.5,0) circle (0.05);
	\draw[thick] (-.5,2) -- (0,2) -- (0, 1) -- (.5,1);
	\draw[thick] (-.5,1.5) -- (-.25,1.5) -- (-.25, 0) -- (.5,0);
	\draw[thick] (-.5,.5) -- (-.25,.5);
	\draw[thick, dashed] (-.5,2.25) -- (-.5,-.25) -- (.5,-.25) -- (.5,2.25);
	\end{tikzpicture}.
	\end{align*}
\end{exam}

\begin{defn}
The set $LR(\chi, \e)$ of \emph{shaded LR diagrams} is defined recursively.
If $n = 0$, $LR(\chi, \e)$ consists of an empty diagram.
If $n > 0$, let $\bar\chi(k) = \chi(k-1)$ for $k \in \set{2, \ldots, n}$ and $\bar\e(k) = \e(k-1)$ for $k \in \set{2, \ldots, n}$.
Each $D \in LR(\bar\chi, \bar\e)$ then corresponds to two (unique) elements of $LR(\chi, \e)$ via the following process:
\begin{itemize}
\item First, add to the top of $D$ a node on the side corresponding to $\chi(1)$, shaded by $\e(1)$.
\item If a string of shade $\e(1)$ extends from the top of $D$, connect it to the added node.
\item Then, choose to either extend a string from the added node to the top of the new diagram or not, and extend any other strings from $D$ to the top of the new diagram.
\end{itemize}
We will denote the subset of $LR(\chi, \epsilon)$ consisting of diagrams with exactly $k$ strings that reach the top by $LR_k(\chi, \e)$.

We refer to the vertical portions of strings in $LR$ and bi-non-crossing diagrams as \emph{spines}, and horizontal segments connecting nodes to spines as \emph{ribs}.
\end{defn}

Note that any element of $LR_0(\chi, \e)$ corresponds to a partition $\pi \in BNC(\chi)$ by taking as blocks in $\pi$ each set of nodes joined by strings in the diagram.
The non-crossing diagram corresponding to $\pi$ will then match the original $LR$-diagram with the shades removed.

\begin{ex}\label{2_node_ex}
Consider $\chi=(\ell,r)$ and $\epsilon=(','')$.
Then $LR(\chi,\epsilon)$ consists of the following diagrams:
	
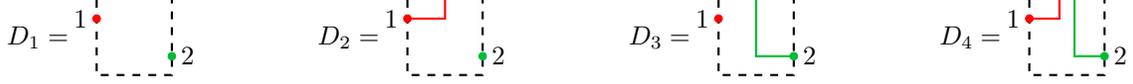
\begin{align*}
		\begin{tikzpicture}[baseline]
			\node[left] at (-.75,.25) {$D_1=$};
			\draw[thick,dashed] (-.5,.75) -- (-.5,-.25) -- (.5,-.25) -- (.5,.75);
			\node[left] at (-.5,.5) {1};
			\draw[red,fill=red] (-.5,.5) circle (0.05);
			\node[right] at (.5,0) {2};
			\draw[ggreen,fill=ggreen] (.5,0) circle (0.05);
		\end{tikzpicture}
		\qquad\qquad
		\begin{tikzpicture}[baseline]
			\node[left] at (-.75,.25) {$D_2=$};
			\draw[thick,dashed] (-.5,.75) -- (-.5,-.25) -- (.5,-.25) -- (.5,.75);
			\node[left] at (-.5,.5) {1};
			\draw[red,fill=red] (-.5,.5) circle (0.05);
			\draw[thick,red] (-.5,.5) -- (0,.5) -- (0, .75);
			\node[right] at (.5,0) {2};
			\draw[ggreen,fill=ggreen] (.5,0) circle (0.05);
		\end{tikzpicture}
		\qquad\qquad
		\begin{tikzpicture}[baseline]
			\node[left] at (-.75,.25) {$D_3=$};
			\draw[thick,dashed] (-.5,.75) -- (-.5,-.25) -- (.5,-.25) -- (.5,.75);
			\node[left] at (-.5,.5) {1};
			\draw[red,fill=red] (-.5,.5) circle (0.05);
			\node[right] at (.5,0) {2};
			\draw[ggreen,fill=ggreen] (.5,0) circle (0.05);
			\draw[thick,ggreen] (.5,0) -- (0,0) -- (0,.75);
		\end{tikzpicture}
		\qquad\qquad
		\begin{tikzpicture}[baseline]
			\node[left] at (-.75,.25) {$D_4=$};
			\draw[thick,dashed] (-.5,.75) -- (-.5,-.25) -- (.5,-.25) -- (.5,.75);
			\node[left] at (-.5,.5) {1};
			\draw[red,fill=red] (-.5,.5) circle (0.05);
			\draw[thick,red] (-.5,.5) -- (-.1, .5)-- (-.1,.75);
			\node[right] at (.5,0) {2};
			\draw[ggreen,fill=ggreen] (.5,0) circle (0.05);
			\draw[thick,ggreen] (.5,0) -- (.1,0) -- (.1,.75);
		\end{tikzpicture}
	\end{align*}
Here $LR_0(\chi,\epsilon) = \{D_1\}$.
\end{ex}

\begin{ex}\label{3_node_ex}
For a slightly more robust example we consider $\chi=(r,\ell,r)$ and $\epsilon=(',','')$. Then $LR(\chi,\epsilon)$ consists of the following diagrams:
	
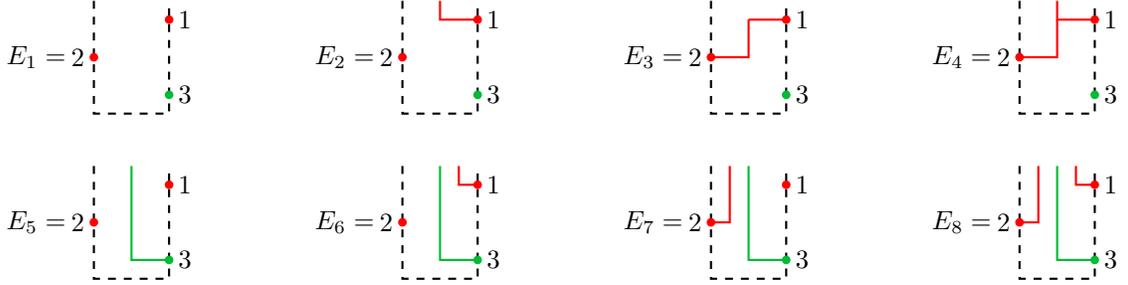
\begin{align*}
		\begin{tikzpicture}[baseline]
			\node[left] at (-.75,.5) {$E_1=$};
			\draw[thick,dashed] (-.5,1.25) -- (-.5,-.25) -- (.5,-.25) -- (.5,1.25);
			\draw[red,fill=red] (.5, 1) circle (0.05);
			\draw[red,fill=red] (-.5,.5) circle (0.05);
			\draw[ggreen,fill=ggreen] (.5,0) circle (0.05);
			\node[right] at (.5,1) {1};
			\node[left] at (-.5,.5) {2};
			\node[right] at (.5,0) {3};
		\end{tikzpicture}
		\qquad\qquad
		\begin{tikzpicture}[baseline]
			\node[left] at (-.75,.5) {$E_2=$};
			\draw[thick,dashed] (-.5,1.25) -- (-.5,-.25) -- (.5,-.25) -- (.5,1.25);
			\draw[red,fill=red] (.5, 1) circle (0.05);
			\draw[thick,red] (.5,1) -- (0,1) -- (0,1.25);
			\draw[red,fill=red] (-.5,.5) circle (0.05);
			\draw[ggreen,fill=ggreen] (.5,0) circle (0.05);
			\node[right] at (.5,1) {1};
			\node[left] at (-.5,.5) {2};
			\node[right] at (.5,0) {3};
		\end{tikzpicture}
		\qquad\qquad		
		\begin{tikzpicture}[baseline]
			\node[left] at (-.75,.5) {$E_3=$};
			\draw[thick,dashed] (-.5,1.25) -- (-.5,-.25) -- (.5,-.25) -- (.5,1.25);
			\draw[red,fill=red] (.5, 1) circle (0.05);
			\draw[thick,red] (.5,1) -- (0,1);
			\draw[red,fill=red] (-.5,.5) circle (0.05);
			\draw[thick,red] (-.5,.5)--(0,.5)--(0,1);
			\draw[ggreen,fill=ggreen] (.5,0) circle (0.05);
			\node[right] at (.5,1) {1};
			\node[left] at (-.5,.5) {2};
			\node[right] at (.5,0) {3};
		\end{tikzpicture}
		\qquad\qquad
		\begin{tikzpicture}[baseline]
			\node[left] at (-.75,.5) {$E_4=$};
			\draw[thick,dashed] (-.5,1.25) -- (-.5,-.25) -- (.5,-.25) -- (.5,1.25);
			\draw[red,fill=red] (.5, 1) circle (0.05);
			\draw[thick,red] (.5,1) -- (0,1) -- (0,1.25);
			\draw[red,fill=red] (-.5,.5) circle (0.05);
			\draw[thick,red] (-.5,.5)--(0,.5)--(0,1);
			\draw[ggreen,fill=ggreen] (.5,0) circle (0.05);
			\node[right] at (.5,1) {1};
			\node[left] at (-.5,.5) {2};
			\node[right] at (.5,0) {3};
		\end{tikzpicture}
		\\
		\\
		\begin{tikzpicture}[baseline]
			\node[left] at (-.75,.5) {$E_5=$};
			\draw[thick,dashed] (-.5,1.25) -- (-.5,-.25) -- (.5,-.25) -- (.5,1.25);
			\draw[red,fill=red] (.5, 1) circle (0.05);
			\draw[red,fill=red] (-.5,.5) circle (0.05);
			\draw[ggreen,fill=ggreen] (.5,0) circle (0.05);
			\draw[thick, ggreen] (.5,0) -- (0,0) -- (0, 1.25);
			\node[right] at (.5,1) {1};
			\node[left] at (-.5,.5) {2};
			\node[right] at (.5,0) {3};
		\end{tikzpicture}
		\qquad\qquad
		\begin{tikzpicture}[baseline]
			\node[left] at (-.75,.5) {$E_6=$};
			\draw[thick,dashed] (-.5,1.25) -- (-.5,-.25) -- (.5,-.25) -- (.5,1.25);
			\draw[red,fill=red] (.5, 1) circle (0.05);
			\draw[thick,red] (.5,1) -- (.25, 1) -- (.25,1.25);
			\draw[red,fill=red] (-.5,.5) circle (0.05);
			\draw[ggreen,fill=ggreen] (.5,0) circle (0.05);
			\draw[thick, ggreen] (.5,0) -- (0,0) -- (0, 1.25);
			\node[right] at (.5,1) {1};
			\node[left] at (-.5,.5) {2};
			\node[right] at (.5,0) {3};
		\end{tikzpicture}
		\qquad\qquad
		\begin{tikzpicture}[baseline]
			\node[left] at (-.75,.5) {$E_7=$};
			\draw[thick,dashed] (-.5,1.25) -- (-.5,-.25) -- (.5,-.25) -- (.5,1.25);
			\draw[red,fill=red] (.5, 1) circle (0.05);
			\draw[red,fill=red] (-.5,.5) circle (0.05);
			\draw[thick,red] (-.5,.5) -- (-.25,.5) -- (-.25,1.25);
			\draw[ggreen,fill=ggreen] (.5,0) circle (0.05);
			\draw[thick, ggreen] (.5,0) -- (0,0) -- (0, 1.25);
			\node[right] at (.5,1) {1};
			\node[left] at (-.5,.5) {2};
			\node[right] at (.5,0) {3};
		\end{tikzpicture}
		\qquad\qquad
		\begin{tikzpicture}[baseline]
			\node[left] at (-.75,.5) {$E_8=$};
			\draw[thick,dashed] (-.5,1.25) -- (-.5,-.25) -- (.5,-.25) -- (.5,1.25);
			\draw[red,fill=red] (.5, 1) circle (0.05);
			\draw[thick,red] (.5,1) -- (.25,1) -- (.25,1.25);
			\draw[red,fill=red] (-.5,.5) circle (0.05);
			\draw[thick,red] (-.5,.5) -- (-.25,.5) -- (-.25,1.25);
			\draw[ggreen,fill=ggreen] (.5,0) circle (0.05);
			\draw[thick, ggreen] (.5,0) -- (0,0) -- (0, 1.25);
			\node[right] at (.5,1) {1};
			\node[left] at (-.5,.5) {2};
			\node[right] at (.5,0) {3};
		\end{tikzpicture}
	\end{align*}
In this case, $LR_0(\chi, \e) = \set{E_1, E_3}$.
We also note that $D_k$ from Example \ref{2_node_ex} gives rise to diagrams $E_{2k-1}$ and $E_{2k}$.
\end{ex}

These shaded LR diagrams are useful because the choices made when constructing them correspond to the choices made in expanding a mixed moment of bi-free random variables in terms of pure moments.
In doing so, we will often consider a map $\alpha : \set{1, \ldots, n} \to I \sqcup J$ where $I$ and $J$ are index sets of left and right random variables, respectively.
We will want to consider the bi-non-crossing diagrams corresponding to such a situation, and so, as useful shorthand, we will define $\chi_\alpha : \set{1, \ldots, n} \to \set{\ell, r}$ by $\chi_\alpha^{-1}\paren{\set\ell} = \alpha^{-1}\paren{I}$, and take $BNC(\alpha) = BNC(\chi_\alpha)$.

\subsection{Combinatorics on Bi-Non-Crossing Partitions}
Recall that the incident algebra $IA(BNC)$ on the set of bi-non-crossing partitions is the algebra of functions of the form
$$f : \bigcup_{n\geq1}\bigcup_{\chi : \set{1, \ldots, n}\to\set{\ell, r}}BNC(\chi)\times BNC(\chi) \longrightarrow \C$$
such that $f(\pi, \sigma) = 0$ unless $\pi$ is a refinement of $\sigma$, with point-wise addition, and multiplication given by convolution: if $f, g \in IA(BNC)$ and $\pi, \sigma \in BNC(\chi)$,
$$(f* g)(\pi, \sigma) = \sum_{\pi\leq\rho\leq\sigma}f(\pi,\rho)g(\rho,\sigma).$$
As was shown in \cite{CNS2014}, we can associate to each interval in $BNC(\chi)$ a product of full intervals.
We say a function $f \in IA(BNC)$ is \emph{multiplicative} if whenever $\sq{\pi, \sigma}$ corresponds to $\prod_{j=1}^k BNC(\beta_k)$, we have
$$f(\pi, \sigma) = \prod_{j=1}^k f(0_{\beta_k}, 1_{\beta_k}).$$
We note that a multiplicative function $f$ is completely determined by the values it takes on each $\paren{0_\chi,1_\chi}$.

The multiplicative identity in $IA(BNC)$ is given by the delta function
$$\delta_{BNC}(\pi, \sigma) = \left\{\begin{array}{ll}1&\text{ if } \pi = \sigma\\0&\text{ otherwise}\end{array}\right..$$
We then define the zeta function on $IA(BNC)$ by
$$\zeta_{BNC}(\pi, \sigma) = \left\{\begin{array}{ll}1&\text{ if } \pi \leq \sigma\\0&\text{ otherwise}\end{array}\right.,$$
and the M\"obius function $\mu_{BNC}$ as its multiplicative inverse:
$$\mu_{BNC}*\zeta_{BNC} = \delta_{BNC} = \zeta_{BNC}*\mu_{BNC}.$$
It can be shown that these three functions are multiplicative and $\mu_{BNC}(\pi, \sigma) = \mu_{NC}(s^{-1}_\chi \cdot \pi, s^{-1}_\chi \cdot \sigma)$.

Suppose that $T_1, \ldots, T_n$ are elements of a non-commutative probability space $\paren{\A, \varphi}$ and $\pi \in BNC(\chi)$ with blocks $V_t = \set{k_{t, 1}<\cdots<k_{t,m_t}}$, $t \in \set{1, \ldots, j}$.
We set the notation
$$\varphi_\pi(T_1, \ldots, T_n) := \prod_{t=1}^j \varphi\paren{T_{k_{t, 1}}\cdots T_{k_{t, m_t}}}$$
and
$$\kappa_\pi(T_1, \ldots, T_n) := \sum_{\substack{\sigma\in BNC(\chi)\\\sigma\leq\pi}}\varphi_\sigma(T_1, \ldots, T_n)\mu_{BNC}(\sigma, \pi).$$
It was shown in \cite{CNS2014} that $\kappa_{1_\chi}(T_1,\ldots, T_n)$ are the \emph{$(\ell, r)$-cumulants} as defined in \cite{MN2013}*{Definition 5.2}.

\subsection{Bi-free Probability}
We now recall several definitions relating to bi-free probability from \cite{V2013-1}.
\begin{defn}
Let $(\A, \varphi)$ be a non-commutative probability space: that is, let $\A$ be a unital algebra and $\varphi : \A \to \mathbb{C}$ be unital and linear.  A \emph{pair of faces in $\A$} is a pair $(C, D)$ of unital subalgebras of $\A$.  We will call $C$ the \emph{left face} and $D$ the \emph{right face}.
\end{defn}

\begin{defn}
A family $\set{ \left(  C_k, D_k\right)}_{k \in K}$ of pairs of faces in a non-commutative probability space $(\A, \varphi)$ is said to be \emph{bi-freely independent} if there exists a family of vector spaces with specified vector states $\{(\X_k, \mathring{\X}_k, \xi_k)\}_{k \in K}$ and unital homomorphisms 
\[
l_k : C_k \to \mathcal{L}(\X_k)\quad\mbox{ and }\quad r_k : D_k \to \mathcal{L}(\X_k)
\]
such that the joint distribution of the family with respect to $\varphi$ is equal to the joint distribution with respect to the vacuum state on the representation on $\ast_{k \in K} (\X_k, \mathring{\X}_k, \xi_k)$.
\end{defn}

\begin{rem}
It is sometimes useful to think of $\e : \set{1, \ldots, n} \to K$ not as a map but rather as the partition $\e^{-1}(K)$ it induces.
Thus if we write $\sigma \leq \e$ for some partition $\sigma$, we mean that $\sigma$ is a refinement of $\e^{-1}(K)$ and so $\e$ is constant on each block in $\sigma$.
\end{rem}

\begin{defn}
Suppose $\pi, \sigma \in BNC(\chi)$ are such that $\pi \leq \sigma$.  We say $\pi$ is a \emph{lateral refinement} of $\sigma$ and write $\pi \lrleq \sigma$ if the bi-non-crossing diagram for $\pi$ can be obtained from that of $\sigma$ by making lateral ``cuts'' along the spines of blocks of $\pi$ between their ribs; that is, by removing some portion of the vertical lines and then any horizontal lines that are no longer attached to a vertical line in the diagram from $\sigma$.
\end{defn}

In \cite{CNS2014}*{Proposition 4.2.1}, we established the following combinatorial result which was crucial in connecting bi-freeness with cumulants.
It will once again prove useful in this paper.
\begin{thm}
\label{thm:twosums}
Let $\chi : \{1, \ldots, n\} \to \{\ell, r\}$ and $\epsilon : \{1,\ldots, n\} \to K$.  Then for every $\pi \in BNC(\chi)$ such that $\pi \leq \epsilon$,
\[
\sum_{\substack{\sigma \in LR_0(\chi, \epsilon) \\ \sigma \lrgeq \pi}} (-1)^{|\pi| - |\sigma|} = \sum_{\substack{\sigma \in BNC(\chi) \\ \pi \leq \sigma \leq \epsilon}} \mu_{BNC}(\pi, \sigma).
\]
\end{thm}

We desire to extend the following result, \cite{CNS2014}*{Theorem 4.3.1}, to the amalgamated setting.
\begin{thm}
Let $(\mathcal{A}, \varphi)$ be a non-commutative probability space and let $\{(C_k, D_k)\}_{k \in K}$ be a family of pairs of faces from $\A$.  Then $\{(C_k, D_k)\}_{k \in K}$ are bi-free if and only if for all $\chi : \{1,\ldots, n\}\to\set{\ell,r}$, $\epsilon : \{1,\ldots, n\} \to K$ non-constant, and 
\[
T_k \in \left\{
\begin{array}{ll}
C_{\epsilon(k)} & \mbox{if } \chi(k) = \ell  \\
D_{\epsilon(k)} & \mbox{if } \chi(k) = r
\end{array} \right.,
\]
we have 
\[
\kappa_{1_\chi}(T_1, \ldots, T_n) = 0.
\]
Equivalently, for all $\chi : \{1,\ldots, n\}\to\set{\ell,r}$, for all $\epsilon : \{1,\ldots, n\} \to K$, and for all $T_k$ as defined above,
\[
\varphi(T_1 \cdots T_n) = \sum_{\pi \in BNC(\chi)} \left[ \sum_{\substack{\sigma \in BNC(\chi) \\ \pi \leq \sigma \leq \epsilon}} \mu_{BNC}(\pi, \sigma) \right]\varphi_\pi(T_1,\ldots, T_n).
\]
\end{thm}

\section{Bi-Free Families with Amalgamation}
\label{sec:bifreefamilieswithamalgamation}

In this section, we will recall and develop the structures from \cite{V2013-1}*{Section 8} necessary to discuss bi-freeness with amalgamation.
Throughout the paper, $B$ will denote a unital algebra over $\mathbb{C}$.


\subsection{Concrete Structures for Bi-Free Probability with Amalgamation}

To begin the necessary constructions in the amalgamated setting, we need an analogue of a vector space with a specified vector state.
\begin{defn}
A \emph{$B$-$B$-bimodule with a specified $B$-vector state} is a triple $(\X, \mathring{\X}, p)$ where $\X$ is a direct sum of $B$-$B$-bimodules
\[
\X = B \oplus \mathring{\X},
\]
and $p : \X \to B$ is the linear map
\[
p(b \oplus \eta) = b.
\]
\end{defn}
\begin{rem}
\label{rem:propertiesofpointedBBbimodules}
Given a $B$-$B$-bimodule with a specified $B$-vector state $(\X, \mathring{\X}, p)$, for $b_1, b_2 \in B$ and $\eta \in \X$ we have
\[
p(b_1 \cdot \eta \cdot b_2) = b_1 p(\eta) b_2.
\]
\end{rem}

\begin{defn}
Given a $B$-$B$-bimodule with a specified $B$-vector state $(\X, \mathring{\X}, p)$, let $\mathcal{L}(\X)$ denote the set of linear operators on $\X$.  Given $b \in B$, we define two operators $L_b, R_b \in \mathcal{L}(\X)$ by
\[
L_b(\eta) = b \cdot \eta \qquad \mbox{ and } \qquad R_b(\eta) = \eta \cdot b \qquad \text{ for } \eta \in \X.
\]

In addition, we define the unital subalgebras $\mathcal{L}_\ell(\X)$ and $\mathcal{L}_r(\X)$ of $\mathcal{L}(\X)$ by
\begin{align*}
\mathcal{L}_\ell(\X) &:= \{ T \in \mathcal{L}(\X) \, \mid \, TR_b = R_b T \mbox{ for all }b \in B\}\\
\mathcal{L}_r(\X) &:= \{ T \in \mathcal{L}(\X) \, \mid \, TL_b = L_b T \mbox{ for all }b \in B\}.
\end{align*}
We call $\mathcal{L}_\ell(\X)$ and $\mathcal{L}_r(\X)$ the \emph{left} and \emph{right algebras} of $\mathcal{L}(\X)$, respectively.
\end{defn}
\begin{rem}
Note $\mathcal{L}_\ell(\X)$ consists of all operators in $\mathcal{L}(\X)$ that are right $B$-linear; that is, if $T \in \mathcal{L}_\ell(\X)$ then
\[
T( \eta \cdot b) = T(R_b(\eta)) = R_b(T(\eta)) = T(\eta) \cdot b
\]
for all $b \in B$ and $\eta \in \X$.
In the usual treatment of bimodules, what we have denoted $\mathcal{L}_\ell(\X)$ would instead be $\mathcal{L}_r(\X)$ and vice versa, to reflect the fact that its elements are right $B$-linear.
However, we take our left (resp. right) face to be a sub-algebra of $\mc L_\ell(\X)$ (resp. $\mc L_r(\X)$).
Moreover, one sees from the $B$-$B$-bimodule structure that $b \mapsto L_b$ is a homomorphism, $b \mapsto R_b$ is an anti-homomorphism, and the ranges of these maps commute.
Hence
\[
\{L_b \, \mid \, b \in B\} \subseteq \mathcal{L}_\ell(\X) \qquad \mbox{and} \qquad \{R_b \, \mid \, b \in B\} \subseteq \mathcal{L}_r(\X).
\]
Thus, in the context of this paper, $\mathcal{L}_\ell(\X)$ consists of `left' operators and $\mathcal{L}_r(\X)$ consists of `right' operators.
\end{rem}

As we are interested in $\mathcal{L}(\X)$ and amalgamating over $B$, we will need an ``expectation'' from $\mathcal{L}(\X)$ to $B$.
\begin{defn}
\label{defn:expectationofLXontoB}
Given a $B$-$B$-bimodule with a specified $B$-vector state $(\X, \mathring{\X}, p)$, we define the linear map $E_{\mathcal{L}(\X)} : \mathcal{L}(\X) \to B$ by
\[
E_{\mathcal{L}(\X)}(T) = p(T(1_B \oplus 0))
\]
for all $T \in \mathcal{L}(\X)$.  We call $E_{\mathcal{L}(\X)}$ the \emph{expectation of $\mathcal{L}(\X)$ onto $B$}.
\end{defn}
The following important properties justify calling $E_{\mathcal{L}(\X)}$ an expectation.

\begin{prop}
\label{prop:propertiesofEforLX}
Let $(\X, \mathring{\X}, p)$ be a $B$-$B$-bimodule with a specified $B$-vector state.  Then
\[
E_{\mathcal{L}(\X)}(L_{b_1} R_{b_2} T) = b_1 E_{\mathcal{L}(\X)}(T) b_2
\]
for all $b_1, b_2 \in B$ and $T \in \mathcal{L}(\X)$, and
\[
E_{\mathcal{L}(\X)}(TL_b) = E_{\mathcal{L}(\X)}(TR_b)
\]
for all $b \in B$ and $T \in \mathcal{L}(\X)$.
\end{prop}
\begin{proof}
If $b_1, b_2 \in B$ and $T \in \mathcal{L}(\X)$, we see that
\[
E_{\mathcal{L}(\X)}(L_{b_1} R_{b_2} T) = p(L_{b_1} R_{b_2} T(1_B \oplus 0)) = p(L_{b_1} R_{b_2} (E(T) \oplus \eta)) = p( (b_1 E(T) b_2) \oplus (b_1 \cdot \eta \cdot b_2)) = b_1 E(T) b_2
\]
where $\eta \in \mathring{\X}$. The second result is immediate from $L_b(1_B\oplus 0) = R_b(1_B\oplus0)$.
\end{proof}

To complete this section, we recall the construction of the reduced free product of $B$-$B$-bimodules with specified $B$-vector states.

\begin{cons}
\label{cons:freeproductconstruction}
Let $\{(\X_k, \mathring{\X}_k, p_k)\}_{k \in K}$ be $B$-$B$-bimodules with specified $B$-vector states.  For simplicity, let $E_k$ denote $E_{\mathcal{L}(\X_k)}$.  The \emph{free product of $\{(\X_k, \mathring{\X}_k, p_k)\}_{k \in K}$ with amalgamation over $B$} is defined to be the $B$-$B$-bimodule with specified vector state $(\X, \mathring{\mc{X}}, p)$  where $\X = B \oplus \mathring{\X}$ and $\mathring{\mc{X}}$ is the $B$-$B$-bimodule
\begin{align*}
	\mathring{\mc{X}}=\bigoplus_{n\geq 1}\bigoplus_{k_1\neq k_2\neq\cdots\neq k_n} \mathring{\mc{X}}_{k_1}\otimes_B\cdots\otimes_B\mathring{\mc{X}}_{k_n}
\end{align*}
with the left and right actions of $B$ on $\mathring{\X}$ defined by
\begin{align*}
		b \cdot (x_1\otimes\cdots\otimes x_n) &= (L_b x_1)\otimes\cdots\otimes x_n\\
		(x_1\otimes\cdots\otimes x_n) \cdot b&= x_1\otimes\cdots\otimes (R_b x_n).
\end{align*}
We use $\ast_{k \in K} \X_k$ to denote $\X$.
	
For each $k \in K$, we define the left representation $\lambda_k : \mathcal{L}(\X_k) \to \mathcal{L}(\X)$ as follows:  let
\[
V_k : \X \to \X_k \otimes_B \left(B\oplus \bigoplus_{n\geq 1}\bigoplus_{\substack{k_1\neq k_2\neq \cdots \neq k_n\\k_1\neq k}} \mathring{\mc{X}}_{k_1}\otimes_B\cdots\otimes_B\mathring{\mc{X}}_{k_n}\right)
\]
be the standard $B$-$B$-bimodule isomorphism and set
\[
\lambda_k(T) = V_k^{-1}(T \otimes I)V_k.
\]
More precisely,
\begin{align*}
		\lambda_k(T) (b) = E_{k}(T)b \oplus (T-L_{E_k(T)})b,
\end{align*}
and
\begin{align*}
		\lambda_k(T) (x_1\otimes\cdots \otimes x_n) =\left\{
										\begin{array}{ll}
											\left(L_{p_{k}(Tx_1)} x_2\otimes\cdots\otimes x_n \right) \oplus \left(  [(1-p_{k})Tx_1]\otimes\cdots \otimes x_n \right) & \text{if }x_1\in \mathring{\mc{X}}_k\\
												&\\
											\left(L_{E_{k}(T)} x_1\otimes\cdots\otimes x_n\right) \oplus \left([(T-L_{E_{k}(T)})1_B]\otimes x_1\otimes\cdots\otimes x_n \right) & \text{if }x_1\not\in\mathring{\mc{X}}_k
										\end{array}\right.,
	\end{align*}
for all $T \in \mathcal{L}(\X_k)$ (where the tensor of length zero is $1_B$).
Observe that $\lambda_k$ is a homomorphism, $\lambda_k(L_b)=L_b$, and $\lambda_k(\mathcal{L}_\ell(\X_k)) \subseteq \mathcal{L}_\ell(\X)$.

Similarly, for each $k \in K$, we define the map $\rho_k : \mathcal{L}(\X_k) \to \mathcal{L}(\X)$ as follows:  let
\[
W_k : \X \to \left(B\oplus\bigoplus_{n\geq1}\bigoplus_{\substack{k_1\neq k_2\neq\cdots\neq k_n\\ k_n \neq k}}  \mathring{\mc{X}}_{k_1}\otimes_B\cdots\otimes_B\mathring{\mc{X}}_{k_n}\right) \otimes_B \X_k
\]
be the standard $B$-$B$-bimodule isomorphism and define
\[
\rho_k(T) = W_k^{-1}(I \otimes T)W_k.
\]
Again,
\begin{align*}
		\rho_k(T)(b) = b E_k(T) \oplus (T-R_{E_k(T)}) b,
\end{align*}
and
\begin{align*}
		\rho_k(T)(x_1\otimes\cdots \otimes x_n)=\left\{
										\begin{array}{cl}
											 \left(x_1\otimes\cdots\otimes R_{p_{k}(T x_n)}x_{n-1}  \right) \oplus \left(x_1\otimes\cdots \otimes [(1-p_{k})Tx_n]\right) & \text{if }x_n\in \mathring{\mc{X}}_k\\
												&\\
											\left(x_1\otimes\cdots\otimes R_{E_k(T)}x_n\right) \oplus \left(x_1\otimes\cdots\otimes x_n\otimes [(T-R_{E_k(T)})1_B]\right) & \text{if }x_n\not\in\mathring{\mc{X}}_k
										\end{array}\right.,
	\end{align*}
for all $T \in \mathcal{L}(\X_k)$.  Clearly $\rho_k$ is an anti-homomorphism, $\rho_k(R_b)=R_b$, and $\rho_k(\mathcal{L}_r(\X_k)) \subseteq \mathcal{L}_r(\X)$.

In addition, note if $T \in \mathcal{L}(\X_k)$ then
	\begin{align*}
		E_{\mathcal{L}(\X)}(\lambda_k(T))=p(\lambda_k(T)1_B)= p( E_k(T)\oplus [T-L_{E_k(T)}]1_B ) = E_k(T)
	\end{align*}
and similarly $E_{\mathcal{L}(\X)}(\rho_k(T)) = E_k(T)$.  Hence, the above shows that $\mathcal{L}(\X)$ contains each $\mathcal{L}(\X_k)$ in a left-preserving, right-preserving, state-preserving way.
\end{cons}
\begin{rem}
\label{rem:actionscommuteonfreeproductspace}
With computation, we see that $\lambda_i(T)$ and $\rho_j(S)$ commute when $T \in \mathcal{L}_{\ell}(\X_i)$, $S \in \mathcal{L}_r(\X_j)$, and $i \neq j$.
Indeed, notice if $b \in B$ then
\begin{align*}
& \lambda_i(T) \rho_j(S) b \\
&= \lambda_i(T) \left(b E_j(S) \oplus (S-R_{E_j(S)}) b   \right) \\
&= E_{i}(T)bE_j(S) \oplus (T-L_{E_i(T)})bE_j(S) \oplus L_{E_i(T)}(S-R_{E_j(S)}) b  \oplus \left([(T-L_{E_{i}(T)})1_B] \otimes [(S-R_{E_j(S)}) b]   \right),
\end{align*}
whereas
\begin{align*}
&  \rho_j(S) \lambda_i(T)b \\
&= \rho_j(S) \left( E_{i}(T)b \oplus (T-L_{E_i(T)})b   \right) \\
&= E_{i}(T)b E_j(S) \oplus (S-R_{E_j(S)})E_{i}(T)b \oplus  R_{E_j(S)} (T-L_{E_i(T)})b \oplus    \left([(T-L_{E_{i}(T)})b] \otimes [(S-R_{E_j(S)}) 1_B]   \right).
\end{align*}
Since $T \in \mathcal{L}_{\ell}(\X_i)$ and $S \in \mathcal{L}_r(\X_j)$, one sees that
\begin{align*}
L_{E_i(T)}(S-R_{E_j(S)}) b = (S-R_{E_j(S)})L_{E_i(T)} b = (S-R_{E_j(S)})E_i(T) b,\\
R_{E_j(S)} (T-L_{E_i(T)})b = (T-L_{E_i(T)}) R_{E_j(S)} b =  (T-L_{E_i(T)})bE_j(S),
\end{align*}
and
\begin{align*}
[(T-L_{E_{i}(T)})b] \otimes [(S-R_{E_j(S)}) 1_B] & = [(T-L_{E_{i}(T)})R_b 1_B] \otimes [(S-R_{E_j(S)}) 1_B] \\
& = [R_b(T-L_{E_{i}(T)}) 1_B] \otimes [(S-R_{E_j(S)}) 1_B] \\
& = [(T-L_{E_{i}(T)}) 1_B] \otimes [L_b(S-R_{E_j(S)}) 1_B] \\
& = [(T-L_{E_{i}(T)}) 1_B] \otimes [(S-R_{E_j(S)}) L_b1_B] \\
& = [(T-L_{E_{i}(T)})1_B] \otimes [(S-R_{E_j(S)}) b] .
\end{align*}
Thus $\lambda_i(T) \rho_j(S) b = \rho_j(S) \lambda_i(T) b$.  Similar computations show $\lambda_i(T)$ and $\rho_j(T)$ commute on $\mathring{\X}_i$, $\mathring{\X}_j$, and $\mathring{\X}_i \otimes \mathring{\X}_j$, and it is trivial to see that $\lambda_i(T)$ and $\rho_j(T)$ commute on all other components of $\mathring{\X}$.

Note that $\lambda_i(T)$ and $\rho_i(S)$ need not commute, even if $T \in \mathcal{L}_{\ell}(\X_i)$ and $S \in \mathcal{L}_r(\X_i)$.
\end{rem}


\subsection{Abstract Structures for Bi-Free Probability with Amalgamation}

The purpose of this section is to develop an abstract notion of the pair $(\mathcal{L}(\X), E_{\mathcal{L}(\X)})$.
Based on the previous section and Proposition \ref{prop:propertiesofEforLX}, we make the following definition.
\begin{defn}
\label{defn:BBncps}
A \emph{$B$-$B$-non-commutative probability space} is a triple $(\mathcal{A}, E_\mathcal{A}, \varepsilon)$ where $\mathcal{A}$ is a unital algebra, $\varepsilon : B \otimes B^{\mathrm{op}} \to \mathcal{A}$ is a homomorphism such that $\varepsilon|_{B \otimes 1_B}$ and $\varepsilon|_{1_B \otimes B^{\mathrm{op}}}$ are injective, and $E_\mathcal{A} : \mathcal{A} \to B$ is a linear map such that
\[
E_{\mathcal{A}}(\varepsilon(b_1 \otimes b_2)T) = b_1 E_{\mathcal{A}}(T) b_2
\]
for all $b_1, b_2 \in B$ and $T \in \mathcal{A}$, and
\[
E_{\mathcal{A}}(T\varepsilon(b \otimes 1_B)) = E_{\mathcal{A}}(T\varepsilon(1_B \otimes b))
\]
for all $b \in B$ and $T \in \mathcal{A}$.

In addition, we define the unital subalgebras $\mathcal{A}_\ell$ and $\mathcal{A}_r$ of $\mathcal{A}$ by
\begin{align*}
\mathcal{A}_\ell &:= \{ T \in \mathcal{A}  \, \mid \, T\varepsilon(1_B \otimes b) = \varepsilon(1_B \otimes b) T \mbox{ for all }b \in B\}\\
\mathcal{A}_r &:= \{ T \in \mathcal{A}  \, \mid \, T\varepsilon(b \otimes 1_B) = \varepsilon(b \otimes 1_B) T \mbox{ for all }b \in B\}.
\end{align*}
We call $\mathcal{A}_\ell$ and $\mathcal{A}_r$ the \emph{left and right algebras }of $\mathcal{A}$ respectively.
\end{defn}
\begin{rem}
\label{rem:concrete-BBncp-satisfies-abstract-defn}
If $(\X, \mathring{\X}, p)$ is a $B$-$B$-bimodule with a specified $B$-vector state, we see via Proposition \ref{prop:propertiesofEforLX} that $(\mathcal{L}(\X), E_{\mathcal{L}(\X)}, \varepsilon)$ is a $B$-$B$-non-commutative probability space where $E_{\mathcal{L}(\X)}$ is as in Definition \ref{defn:expectationofLXontoB} and $\varepsilon : B \otimes B^{\mathrm{op}} \to B$ is defined by $\varepsilon(b_1 \otimes b_2) = L_{b_1} R_{b_2}$.
As such, in an arbitrary $B$-$B$-non-commutative probability space $(\mathcal{A}, E_\mathcal{A}, \varepsilon)$, we will often use  $L_b$ instead of $\varepsilon(b \otimes 1)$ and $R_b$ instead of $\varepsilon(1 \otimes b)$, in which case $L_b \in \mathcal{A}_\ell$ and $R_b \in \A_r$ for all $b \in B$.
\end{rem}

\begin{rem}
It may appear that Definition \ref{defn:BBncps} is incompatible with the notion of a $B$-probability space in free probability: that is, a pair $(\A, \varphi)$ where $\A$ is a unital algebra containing $B$, and $\varphi : \A \to B$ is a linear map such that $\varphi(b_1Tb_2) = b_1\varphi(T)b_2$ for all $b_1, b_2 \in B$ and $T \in \A$ (see \cite{NSS2002} for example).
However, $\A$ is a $B$-$B$-bimodule by left and right multiplication by $B$, and, by Remark \ref{rem:propertiesofpointedBBbimodules}, $\A$ can be made into a $B$-$B$-bimodule with specified $B$-vector space via $\varphi$.
Hence Remark \ref{rem:concrete-BBncp-satisfies-abstract-defn} implies $\mathcal{L}(\A)$ is a $B$-$B$-non-commutative probability state with
\[
E_{\mathcal{L}(\A)}(T) = \varphi(T)
\]
for all $T \in \mathcal{L}(\A)$.  In addition, we can view $\A$ as a unital subalgebra of both $\mathcal{L}_\ell(\A)$ and $\mathcal{L}_r(\A)$ by left and right multiplication on $\A$ respectively.  

Viewing $\A \subseteq \mathcal{L}_\ell(\A)$, it is clear we can recover the joint $B$-moments of elements of $\A$ from $E_{\mathcal{L}(\A)}$.
Indeed, for $T \in \A \subseteq \mathcal{L}_\ell(\A)$ we have
\[
E_{\mathcal{L}(\A)}(L_{b_1} T L_{b_2}) = E_{\mathcal{L}(\A)}(L_{b_1} T R_{b_2}) = E_{\mathcal{L}(\A)}(L_{b_1} R_{b_2} T) = b_1 E_{\mathcal{L}(\A)}(T) b_2,
\]
which is consistent with the defining property of $\varphi$.
\end{rem}

One should note that Definition \ref{defn:BBncps} differs slightly from \cite{V2013-1}*{Definition 8.3}.  However, given Proposition \ref{prop:propertiesofEforLX} and the following result which demonstrates that a $B$-$B$-non-commutative probability space embeds into $\mathcal{L}(\X)$ for a $B$-$B$-bimodule with a specified $B$-vector state $\X$, Definition \ref{defn:BBncps} indeed specifies the correct abstract objects to study.
\begin{thm}
\label{thm:representingbbncps}
Let $(\mathcal{A}, E_\mathcal{A}, \varepsilon)$ be a $B$-$B$-non-commutative probability space.  Then there exists a $B$-$B$-bimodule with a specified $B$-vector state $(\X, \mathring{\X}, p)$ and a unital homomorphism $\theta : \mathcal{A} \to \mathcal{L}(\X)$ such that 
\[
\theta(L_{b_1} R_{b_2}) = L_{b_1} R_{b_2}, \quad \theta(\mathcal{A}_\ell) \subseteq \mathcal{L}_\ell(\X), \quad
\theta(\mathcal{A}_r) \subseteq \mathcal{L}_r(\X), \quad
\mathrm{and} \quad
E_{\mathcal{L}(\X)}(\theta(T)) = E_\mathcal{A}(T)
\]
for all $b_1, b_2 \in B$ and $T \in \mathcal{A}$.
\end{thm}

\begin{proof}
Consider the vector space over $\mathbb{C}$
\[
\X = B \oplus \Y,
\]
where
\[
\Y = \ker(E_\A) / \mathrm{span}\{TL_b - TR_b \, \mid \, T \in \A, b \in B\}.
\]
Note $\Y$ is a well-defined quotient vector space since $E_{\A}(TL_b - TR_b) = 0$ by Definition \ref{defn:BBncps}.   We will postpone describing the $B$-$B$-module structure on $\X$ until later in the proof.

Let $q : \ker(E_\A) \to \Y$ denote the canonical quotient map.
Then, for $T, A \in \A$ with $E_\mathcal{A}(A) = 0$ and $b \in B$, we define $\theta(T) \in \mathcal{L}(\X)$ by
\[
\theta(T)(b) = E_\mathcal{A}(TL_b) \oplus q(TL_b - L_{E_\mathcal{A}(TL_b)})
\]
and
\[
\theta(T)(q(A)) = E_\mathcal{A}(TA) \oplus q(TA - L_{E_\mathcal{A}(TA)}).
\]
Note $\theta(T)$ is a well-defined linear operator due to the definition of $\Y$.

To see that $\theta$ is a homomorphism, note $\theta$ is clearly linear.  To see that $\theta$ is multiplicative, fix $T, S \in \A$.  If $b \in B$, then
\[
\theta(T)(b) = E_\mathcal{A}(TL_b) \oplus q(TL_b - L_{E(_{\A}TL_b)}).
\]
Thus
\begin{align*}
\theta(S)(\theta(T)(b))  = & \,\,E_{\A}(SL_{E_{\A}(TL_b)}) \oplus q\left(SL_{E_{\A}(TL_b)} - L_{E_{\A}(SL_{E_{\A}(TL_b)})}    \right) \\
& + E_{\A}(S (TL_b - L_{E_{\A}(TL_b)})) \oplus q\left(S(TL_b - L_{E_{\A}(TL_b)}) - L_{E_{\A}(S(TL_b - L_{E_{\A}(TL_b)}))}  \right)  \\
= & \,\,E_{\A}(STL_b) \oplus q\left(STL_b - L_{E_{\A}(STL_b)}\right)  \\
= & \,\,\theta(ST)(b).
\end{align*}
Similarly, if $q(A)\in \Y$ then
\[
\theta(T)(q(A)) = E_{\A}(TA) \oplus q(TA - L_{E_{\A}(TA)}).
\]
Thus
\begin{align*}
\theta(S)(\theta(T)(q(A))) = &\,\, E_{\A}\left(SL_{E_{\A}(TA)}\right) \oplus q\left(SL_{E_{\A}(TA)} - L_{E_{\A}(SL_{E_{\A}(TA)})} \right) \\
& + E_{\A}(S(TA - L_{E_{\A}(TA)})) \oplus q(S(TA - L_{E_{\A}(TA)}) - L_{E_{\A}(S(TA - L_{E_{\A}(TA)}))} )\\
= & \,\,E_{\A}(STA) \oplus q\left( STA - L_{E_{\A}(STA)}\right) \\
= & \,\,\theta(ST)(q(A)).
\end{align*}
Hence $\theta$ is a homomorphism.

To make $\X$ a $B$-$B$-bimodule, we define
\[
b \cdot \xi = \theta(L_b)(\xi) \qquad \mbox{and}\qquad \xi \cdot b = \theta(R_b)(\xi)
\]
for all $\xi \in \X$ and $b \in B$.  It is clear that this is a well-defined $B$-$B$-bimodule structure on $\X$ since $\theta$ is a homomorphism.  

To demonstrate that $\X$ is indeed a $B$-$B$-bimodule with a specified vector state, we must show that $\Y$ is invariant under this $B$-$B$-bimodule structure, and that the $B$-$B$-bimodule structure when restricted to $B \subseteq \X$ is the canonical one.
If $b, b' \in B$ and $q(A) \in \Y$, then
\[
\theta(L_b)(b') = E_{\A}(L_b L_{b'}) \oplus q(L_b L_{b'} - L_{E_{\A}(L_b L_{b'})}) = bb' \oplus q(L_{bb'} - L_{bb'}) = bb' \oplus 0
\]
and
\begin{align*}
\theta(L_b)(q(A)) &= E_{\A}(L_b A) \oplus q(L_b A - L_{E_{\A}(L_b A)})\\
&= bE_{\A}(A) \oplus q(L_b A - L_{E_{\A}(L_b A)}) = 0 \oplus q(L_b A - L_{E_{\A}(L_b A)}).
\end{align*}
Similarly,
\[
\theta(R_b)(b') = E_{\A}(R_b L_{b'}) \oplus q(R_b L_{b'} - L_{E_{\A}(R_b L_{b'})}) = b'b \oplus q(L_{b'}R_b - L_{b'}L_{b}) = b'b \oplus 0
\]
and
\begin{align*}
\theta(R_b)(q(A)) &= E_{\A}(R_b A) \oplus q(R_b A - L_{E(R_b A)})\\
&= E_{\A}(A)b \oplus q(R_b A - L_{E_{\A}(R_b A)}) = 0 \oplus q(R_b A - L_{E_{\A}(R_b A)}).
\end{align*}
Thus $\X$ is a $B$-$B$-bimodule with a specified $B$-vector state.

Since $\theta$ is a homomorphism, it is clear that $\theta(\mathcal{A}_\ell) \subseteq \mathcal{L}_\ell(\X)$ and $\theta(\mathcal{A}_r) \subseteq \mathcal{L}_r(\X)$ due to the definition of the $B$-$B$-bimodule structure on $\X$.  Finally, if $T \in \A$ then
\[
E_{\mathcal{L}(\X)}(\theta(T)) = p(\theta(T) (1_B \oplus 0)) = p(E_\mathcal{A}(T) \oplus q(T - L_{E_\mathcal{A}(T)})) = E_\mathcal{A}(T).\qedhere
\]
\end{proof}


\subsection{Bi-Free Families of Pairs of $B$-Faces}

With the notion of a $B$-$B$-non-commutative probability space from Definition \ref{defn:BBncps}, we are now able to define the main concept of this paper, following \cite{V2013-1}*{Definition 8.5}.
\begin{defn}
\label{defn:pairofBfaces}
Let $(\mathcal{A}, E_\mathcal{A}, \varepsilon)$ be a $B$-$B$-non-commutative probability space.  A \emph{pair of $B$-faces of $\A$} is a pair $(C, D)$ of unital subalgebras of $\A$ such that
\[
\varepsilon(B \otimes 1_B) \subseteq C \subseteq \A_\ell \qquad \mathrm{and}\qquad \varepsilon(1_B \otimes B^{op}) \subseteq D \subseteq \A_r.
\]

A family $\{(C_k, D_k)\}_{k \in K}$ of pair of $B$-faces of $\A$ is said to be \emph{bi-free with amalgamation over $B$} (or simply \emph{bi-free over $B$}) if there exist $B$-$B$-bimodules with specified $B$-vector states $\{(\X_k, \mathring{\X}_k, p_k)\}_{k \in K}$ and unital homomorphisms $l_k : C_k \to \mathcal{L}_{\ell}(\X_k)$, $r_k : D_k \to \mathcal{L}_{r}(\X_k)$ such that the joint distribution of $\{(C_k, D_k)\}_{k \in K}$ with respect to $E_\mathcal{A}$ is equal to the joint distribution of the images $\{((\lambda_k \circ l_k)(C_k), (\rho_k \circ r_k)(D_k))\}_{k \in K}$ inside $\mathcal{L}(\ast_{k \in K} \X_k)$ with respect to $E_{\mathcal{L}(\ast_{k \in K} \X_k)}$.
\end{defn}
It will be an immediate consequence of Theorem \ref{thm:bifreeequivalenttouniversalpolys} that the selection of representations in Definition \ref{defn:pairofBfaces} does not matter (see \cite{V2013-1}*{Proposition 2.9}).
Note that if $\{(C_k, D_k)\}_{k \in K}$ is bi-free over $B$, then $\set{C_k}_{k\in K}$ is free with amalgamation over $B$ (as is $\set{D_k}_{k\in K}$) and $C_i$ and $D_j$ commute in distributions whenever $i \neq j$.

To conclude this section, we give the following example.
\begin{exam}
Let $\mathfrak{M}$ be a type II$_1$ factor with a faithful normal tracial state $\tau_\mathfrak{M}$ and let $\mathfrak{N}$ be a von Neumann subalgebra of $\mathfrak{M}$ containing $1_\mathfrak{M}$.  Then there exists a canonical trace-preserving (and thus faithful) conditional expectation $E_\mathfrak{N} : \mathfrak{M} \to \mathfrak{N}$ given by $E_\mathfrak{N}(T) = P_\mathfrak{N} T P_\mathfrak{N}$ for all $T \in \mathfrak{M}$ where $P_\mathfrak{N}$ is the projection of $L^2(\mathfrak{M}, \tau_\mathfrak{M})$ onto $L^2(\mathfrak{N}, \tau_\mathfrak{M})$.  Thus $\mathfrak{M}$ has a $\mathfrak{N}$-$\mathfrak{N}$-bimodule structure with specified $\mathfrak{N}$-vector state induced by $E_\mathfrak{N}$ and $\mathfrak{M}$ has a canonical left and right action on this bimodule.

Suppose $\mathfrak{N}$ is a von Neumann algebra such that there exists type II$_1$ factors $\mathfrak{M}_1$ and $\mathfrak{M}_2$ containing $\mathfrak{N}$.  It is elementary to verify that $\mathfrak{M}_1 \ast_\mathfrak{N} \mathfrak{M}_2$ can be made into a $\mathfrak{N}$-$\mathfrak{N}$-non-commutative probability space with expectation $E_\mathfrak{N} : \mathfrak{M}_1 \ast_\mathfrak{N} \mathfrak{M}_2 \to \mathfrak{N}$ the canonical conditional expectation.  Moreover, it is clear that $(\mathfrak{M}_1, \mathfrak{M}_1^{\mathrm{op}})$ and $(\mathfrak{M}_2, \mathfrak{M}_2^{\mathrm{op}})$ are bi-free over $\mathfrak{N}$.
\end{exam}

\section{Operator-Valued Bi-Multiplicative Functions}
\label{sec:OperatorValuedBiMultiplicativeFunctions}

In this section, we will extend the notion of multiplicative functions (see \cite{NSS2002}*{Section 2} or \cite{S1998}*{Section 2}) in order to study $B$-$B$-non-commutative probability spaces.

\subsection{A Partial Ordering and Notation}

Given $\chi : \set{1, \ldots, n} \to \set{\ell, r}$ and $\pi \in BNC(\chi)$, we consider the following additional ordering on $\set{1, \ldots, n}$.
\begin{defn}
Let $\chi : \{1,\ldots, n\} \to \{\ell, r\}$.  The total ordering $\prec_\chi$ on $\{1,\ldots, n\}$ is defined as follows: for $a, b\in\{1,\ldots, n\}$, $a\prec_\chi b$ if $s_\chi^{-1} (a)< s_\chi^{-1}(b)$.
Given a subset $V \subseteq \{1,\ldots, n\}$, we will say that $V$ is a $\chi$-interval if it is an interval with respect to the ordering $\prec_\chi$.
In addition we define $\min_{\prec_\chi}(V)$ and $\max_{\prec_\chi}(V)$ to be the minimal and maximal elements of $V$ with respect to the ordering $\prec_\chi$.
\end{defn}

Note that if $\pi \in BNC(\chi)$, then $\prec_\chi$ orders the nodes of $\pi$ from top to bottom along the left side of its bi-non-crossing diagram, then bottom to top along the right side.
\begin{nota}
Let $\chi: \{1,\ldots, n\} \to \{\ell, r\}$.
Given a subset $S \subseteq \{1,\ldots, n\}$ we define $\chi|_S : S \to \{\ell, r\}$ to be the restriction of $\chi$ to $S$.
Similarly, given an $n$-tuple of objects $(T_1, \ldots, T_n)$, we define $(T_1, \ldots, T_n)|_S$ to be the $|S|$-tuple where the elements in positions not indexed by an element of $S$ are removed.
Finally, given $\pi \in BNC(\chi)$ such that $S$ is a union of blocks of $\pi$, we define $\pi|_S \in BNC(\chi|_S)$ to be the bi-non-crossing partition formed by taking the blocks of $\pi$ contained in $S$.  
\end{nota}

\subsection{Definition of Bi-Multiplicative Functions}

With the above definitions and notation, we can begin the necessary constructions for the operator-valued bi-free cumulants.  Note simple examples follow the definition along with a heuristic version of Definition \ref{defnbimultiplicative} will be given in Remark \ref{remonbimultiplicative} which should be of aid to an experienced free probabilist.  
\begin{defn}
\label{defnbimultiplicative}
Let $(\A, E, \varepsilon)$ be a $B$-$B$-non-commutative probability space and let 
\[
\Phi : \bigcup_{n\geq 1} \bigcup_{\chi : \{1,\ldots, n\} \to \{\ell, r\}} BNC(\chi) \times \mathcal{A}_{\chi(1)} \times \cdots \times \mathcal{A}_{\chi(n)} \to B
\]
be a function that is linear in each $\mathcal{A}_{\chi(k)}$.  We say that $\Phi$ is \emph{bi-multiplicative} if the following four conditions hold:
\begin{enumerate}[label=(\roman*)]
\item\label{def:bimult:i} Let $\chi : \{1,\ldots, n\} \to \{\ell, r\}$, let $T_k \in \mathcal{A}_{\chi(k)}$, let $b \in B$, and let
\[
q = \max\{ k \in \{1,\ldots, n\} \, \mid \, \chi(k) \neq \chi(n)\}.
\]
If $\chi(n) = \ell$ then
\[
\Phi_{1_\chi}(T_1, \ldots, T_{n-1}, T_nL_b) = \left\{
\begin{array}{ll}
\Phi_{1_\chi}(T_1, \ldots, T_{q-1}, T_q R_b, T_{q+1}, \ldots, T_n) & \mbox{if } q \neq -\infty  \\
\Phi_{1_\chi}(T_1, \ldots, T_{n-1}, T_n)b & \mbox{if } q = -\infty
\end{array} \right. .
\]
If $\chi(n) = r$ then 
\[
\Phi_{1_\chi}(T_1, \ldots, T_{n-1}, T_nR_b) = \left\{
\begin{array}{ll}
\Phi_{1_\chi}(T_1, \ldots, T_{q-1}, T_q L_b, T_{q+1}, \ldots, T_n) & \mbox{if } q \neq -\infty  \\
b\Phi_{1_\chi}(T_1, \ldots, T_{n-1}, T_n) & \mbox{if } q = -\infty
\end{array} \right. .
\]
\item\label{def:bimult:ii} Let $\chi : \{1,\ldots, n\} \to \{\ell, r\}$, let $T_k \in \mathcal{A}_{\chi(k)}$, let $b \in B$, let $p \in \{1,\ldots, n\}$, and let
\[
q = \max\{ k \in \{1,\ldots, n\} \, \mid \, \chi(k) = \chi(p), k < p\}.
\]
If $\chi(p) = \ell$ then
\[
\Phi_{1_\chi}(T_1, \ldots, T_{p-1}, L_bT_p, T_{p+1}, \ldots, T_n) = \left\{
\begin{array}{ll}
\Phi_{1_\chi}(T_1, \ldots, T_{q-1}, T_qL_b, T_{q+1}, \ldots, T_n) & \mbox{if } q \neq -\infty  \\
b \Phi_{1_\chi}(T_1, T_2, \ldots, T_n) & \mbox{if } q = -\infty
\end{array} \right. .
\]
If $\chi(p) = r$ then
\[
\Phi_{1_\chi}(T_1, \ldots, T_{p-1}, R_bT_p, T_{p+1}, \ldots, T_n) = \left\{
\begin{array}{ll}
\Phi_{1_\chi}(T_1, \ldots, T_{q-1}, T_qR_b, T_{q+1}, \ldots, T_n) & \mbox{if } q \neq -\infty  \\
\Phi_{1_\chi}(T_1, T_2, \ldots, T_n) b & \mbox{if } q = -\infty
\end{array} \right. .
\]
\item\label{def:bimult:iii} Let $\chi : \{1,\ldots, n\} \to \{\ell, r\}$, $T_k \in \mathcal{A}_{\chi(k)}$, and $\pi \in BNC(\chi)$.
Suppose that $V_1, \ldots, V_m$ are $\chi$-intervals which partition $\set{1, \ldots, n}$, each a union of blocks of $\pi$.
Further, suppose $V_1, \ldots, V_m$ are ordered by $\prec_\chi$.
Then
\[
\Phi_\pi(T_1, \ldots, T_n) = \Phi_{\pi|_{V_1}}\left((T_1, \ldots, T_n)|_{V_1}\right) \cdots \Phi_{\pi|_{V_m}}\left((T_1, \ldots, T_n)|_{V_m}\right).
\]
\item\label{def:bimult:iv} Let $\chi : \{1,\ldots, n\} \to \{\ell, r\}$, $T_k \in \mathcal{A}_{\chi(k)}$, and $\pi \in BNC(\chi)$.
Suppose that $V$ and $W$ partition $\{1,\ldots, n\}$, each of which is a union of blocks of $\pi$, $V$ is a $\chi$-interval, and
\[
\min_{\prec_\chi}(\{1,\ldots, n\}), \max_{\prec_\chi}(\{1,\ldots, n\}) \in W.
\]
Let
\[
\theta = \max_{\prec_\chi}\left(\left\{k \in W  \, \mid \, k \prec_\chi \min_{\prec_\chi}(V)\right\}\right) \qquad\mbox{ and } \qquad \gamma = \min_{\prec_\chi}\left(\left\{k \in W  \, \mid \, \max_{\prec_\chi}(V) \prec_\chi k\right\}\right).
\]
Then
\begin{align*}
\Phi_\pi(T_1, \ldots, T_n) &= \left\{
\begin{array}{ll}
\Phi_{\pi|_{W}}\left(\left(T_1, \ldots, T_{\theta-1}, T_\theta L_{\Phi_{\pi|_{V}}\left((T_1,\ldots, T_n)|_{V}\right)}, T_{\theta+1}, \ldots, T_n\right)|_{W}\right)  & \mbox{if } \chi(\theta) = \ell \\
\Phi_{\pi|_{W}}\left(\left(T_1, \ldots, T_{\theta-1}, R_{\Phi_{\pi|_{V}}\left((T_1,\ldots, T_n)|_{V}\right)} T_\theta, T_{\theta+1}, \ldots, T_n\right)|_{W}\right)  & \mbox{if } \chi(\theta) = r 
\end{array} \right. \\
&= \left\{
\begin{array}{ll}
\Phi_{\pi|_{W}}\left(\left(T_1, \ldots, T_{\gamma-1},  L_{\Phi_{\pi|_{V}}\left((T_1,\ldots, T_n)|_{V}\right)} T_\gamma, T_{\gamma+1}, \ldots, T_n\right)|_{W}\right)  & \mbox{if } \chi(\gamma) = \ell  \\
\Phi_{\pi|_{W}}\left(\left(T_1, \ldots, T_{\gamma-1}, T_\gamma R_{\Phi_{\pi|_{V}}\left((T_1,\ldots, T_n)|_{V}\right)}, T_{\gamma+1}, \ldots, T_n\right)|_{W}\right) & \mbox{if } \chi(\gamma) = r
\end{array} \right. .
\end{align*}
\end{enumerate}
\end{defn}

\begin{ex}
For an example on how to apply Properties (i) and (ii) of Definition \ref{defnbimultiplicative}, consider a bi-multiplicative function $\Phi$ and $\chi :\{1, 2, 3, 4, 5\} \to \{\ell, r\}$ such that $\chi^{-1}(\{\ell\}) = \{1, 2, 4\}$ and $\chi^{-1}(\{r\}) = \{3, 5\}$.  Then $\pi = 1_\chi$ is represented diagrammatically as
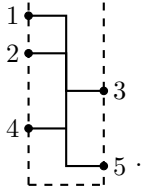
\[
\begin{tikzpicture}[baseline]
\draw[thick,dashed] (-.5,-.25) -- (-.5,2.25);
\draw[thick,dashed] (.5,-.25) -- (.5,2.25);
\draw[thick,dashed] (-.5,-.25) -- (.5,-.25);
	\node[ left] at (-.5, 2) {1};
	\draw[fill=black] (-.5,2) circle (0.05);
	\node[left] at (-.5, 1.5) {2};
	\draw[fill=black] (-.5,1.5) circle (0.05);
	\node[right] at (.5, 1) {3};
	\draw[fill=black] (.5,1) circle (0.05);
	\node[left] at (-.5, .5) {4};
	\draw[fill=black] (-.5,.5) circle (0.05);
	\node[right] at (.5,0) {5};
	\draw[fill=black] (.5,0) circle (0.05);
	\draw[thick] (-.5,2) -- (0,2) -- (0, 1) -- (.5,1);
	\draw[thick] (-.5,1.5) -- (0,1.5) -- (0, 0) -- (.5,0);
	\draw[thick] (-.5,.5) -- (0,.5);
	\end{tikzpicture}.
\]
Using Properties (i) and (ii), we obtain that
\[
\Phi_\pi(T_1, L_{b_1}T_2, R_{b_2}T_3, T_4, T_5 R_{b_3}) = \Phi_\pi(T_1L_{b_1}, T_2, T_3, T_4 L_{b_3}, T_5)b_2.
\]
Thus one should view Properties (i) and (ii) as being able to move elements of $B$ along the dotted lines shown.
\end{ex}
\begin{ex}
The diagram on the left below demonstrates a bi-non-crossing partition satisfying the assumptions of Property (iii) of Definition \ref{defnbimultiplicative} with $V_1 = \{1, 3, 4 \}$, $V_2 = \{5, 7, 8\}$, and $V_3 = \{2, 6\}$.  The diagram on the right below demonstrates a bi-non-crossing partition satisfying the assumptions of Property (iv) of Definition \ref{defnbimultiplicative} with either $V = \{3,4,5,7,8 \}$ and $W = \{1, 2, 6\}$ (in which case $\theta = 1$ and $\gamma = 6$), or $V = \{4,5,8\}$ and $W = \{1,2,3,6,7\}$ (in which case $\theta = 3$ and $\gamma = 7$).
\begin{align*}
	\begin{tikzpicture}[baseline]
	\node[left] at (-.5, 3.5) {1};
	\draw[blue,fill=blue] (-.5,3.5) circle (0.05);
	\node[right] at (.5, 3) {2};
	\draw[red,fill=red] (.5,3) circle (0.05);
	\node[left] at (-.5, 2.5) {3};
	\draw[blue,fill=blue] (-.5,2.5) circle (0.05);
	\node[left] at (-.5, 2) {4};
	\draw[blue,fill=blue] (-.5,2) circle (0.05);
	\node[left] at (-.5, 1.5) {5};
	\draw[ggreen,fill=ggreen] (-.5,1.5) circle (0.05);
	\node[right] at (.5, 1) {6};
	\draw[red,fill=red] (.5,1) circle (0.05);
	\node[right] at (.5, .5) {7};
	\draw[ggreen,fill=ggreen] (.5,.5) circle (0.05);
	\node[left] at (-.5,0) {8};
	\draw[ggreen,fill=ggreen] (-.5,0) circle (0.05);
	\draw[thick, blue] (-.5,3.5) -- (0,3.5) -- (0, 2) -- (-.5,2);
	\draw[thick, blue] (-.5,2.5) -- (0,2.5);
	\draw[thick, ggreen] (-.5,1.5) -- (0,1.5) -- (0, 0) -- (-.5,0);
	\draw[thick, red] (.5,3) -- (.25,3) -- (0.25, 1) -- (.5,1);
	\draw[thick, ggreen] (0,.5) -- (.5,.5);
	\end{tikzpicture}
\qquad
	\begin{tikzpicture}[baseline]
	\node[left] at (-.5, 3.5) {1};
	\draw[fill=black] (-.5,3.5) circle (0.05);
	\node[right] at (.5, 3) {2};
	\draw[fill=black] (.5,3) circle (0.05);
	\node[left] at (-.5, 2.5) {3};
	\draw[fill=black] (-.5,2.5) circle (0.05);
	\node[left] at (-.5, 2) {4};
	\draw[fill=black] (-.5,2) circle (0.05);
	\node[left] at (-.5, 1.5) {5};
	\draw[fill=black] (-.5,1.5) circle (0.05);
	\node[right] at (.5, 1) {6};
	\draw[fill=black] (.5,1) circle (0.05);
	\node[right] at (.5, .5) {7};
	\draw[fill=black] (.5,.5) circle (0.05);
	\node[left] at (-.5,0) {8};
	\draw[fill=black] (-.5,0) circle (0.05);
	\draw[thick] (-.5,3.5) -- (0.25,3.5) -- (0.25, 1) -- (.5,1);
	\draw[thick] (.5,3) -- (0.25,3);
	\draw[thick] (-.5,2.5) -- (0,2.5) -- (0, .5) -- (.5,.5);
	\draw[thick] (-.5,2) -- (-0.25,2) -- (-0.25, 0) -- (-.5,0);
	\draw[thick] (-.5,1.5) -- (-0.25,1.5);		
	\end{tikzpicture}
	\end{align*}
\end{ex}
\begin{rem}
\label{remonbimultiplicative}
Although Definition \ref{defnbimultiplicative} is cumbersome (due to the necessity of specifying cases based on whether certain terms are left or right operators), its properties can be viewed as direct analogues of those of a multiplicative map as described in \cite{NSS2002}*{Section 2.2}.  Indeed, for $\pi \in BNC(\chi)$ and a bi-multiplicative map $\Phi$, each expression of $\Phi_\pi(T_1, \ldots, T_n)$ in Definition \ref{defnbimultiplicative} comes from viewing $s_\chi^{-1}  \circ \pi \in NC(n)$, rearranging the $n$-tuple $(T_1, \ldots, T_n)$ to $(T_{s_\chi(1)}, \ldots, T_{s_\chi(n)})$, replacing any occurrences of $L_bT_j$, $T_j L_b$, $R_b T_j$, and $T_j R_b$ with $bT_j$, $T_j b$, $T_j b$, and $bT_j$ respectively, applying one of the properties of a multiplicative map from \cite{NSS2002}*{Section 2.2}, and reversing the above identifications.
In particular, these properties reduce to those of a multiplicative map when $\chi^{-1}(\{\ell\}) = \set{1,\ldots, n}$.
We use the more complex Definition \ref{defnbimultiplicative} as it will be easier to verify for functions later on.

Since a bi-multiplicative function satisfies all of these properties, it is easy to see that if $\Phi$ is bi-multiplicative, then $\Phi_\pi(T_1, \ldots, T_n)$ is determined by the values 
\[
\{\Phi_{1_{\chi'}}(S_1, \ldots, S_m) \, \mid \, m \in \mathbb{N}, \chi' : \{1,\ldots, m\} \to \{\ell, r\}, S_k \in \mathcal{A}_{\chi(k)}\}.
\]
There may be multiple ways to reduce $\Phi$ to an expression involving elements from the above set, but Definition \ref{defnbimultiplicative} implies that all such reductions are equal.
\end{rem}

Note that Definition \ref{defnbimultiplicative} automatically implies additional properties for bi-multiplicative functions.  Indeed one can either verify the following proposition via Definition \ref{defnbimultiplicative} and casework, or can appeal to the fact that the properties of bi-multiplicative functions can be described via the properties of multiplicative functions as in Remark \ref{remonbimultiplicative} and use the fact that multiplicative functions have additional properties (e.g. see \cite{S1998}*{Remark 2.1.3}).
\begin{prop}
\label{propenhancedproperties}
Let $(\A, E, \varepsilon)$ be a $B$-$B$-non-commutative probability space and let
\[
\Phi : \bigcup_{n\geq 1} \bigcup_{\chi : \{1,\ldots, n\} \to \{\ell, r\}} BNC(\chi) \times \mathcal{A}_{\chi(1)} \times \cdots \times \mathcal{A}_{\chi(n)} \to B
\]
be a bi-multiplicative function.  Given any $\chi : \{1,\ldots, n\} \to \{\ell, r\}$, $\pi \in BNC(\chi)$, and $T_k \in \mathcal{A}_{\chi(k)}$ Properties (i) and (ii) of Definition \ref{defnbimultiplicative} hold when $1_\chi$ is replaced with $\pi$.
\end{prop}


\section{Bi-Free Operator-Valued Moment Function is Bi-Multiplicative}
\label{sec:verifyingrecursivedefinitionfromuniversalpolynomialshasdesiredproperties}

In this section, we will define the bi-free operator-valued moment function based on recursively defined functions $E_\pi(T_1,\ldots, T_n)$ that appear via actions on free product spaces.
However, it is not immediate that it is bi-multiplicative.
The proof of this result requires substantial case work, to which this section is dedicated.

\subsection{Definition of the Bi-Free Operator-Valued Moment Function}

We will begin with the recursive definition of expressions that appear in the operator-valued moment polynomials.
These will arise in the proof of Theorem \ref{thm:bifreeequivalenttouniversalpolys}.

\begin{defn}
\label{defn:recursivedefinitionofEpi}
Let $(\A, E, \varepsilon)$ be a $B$-$B$-non-commutative probability space.  For $\chi : \{1,\ldots,n\} \to \{\ell, r\}$, $\pi \in BNC(\chi)$, and $T_1, \ldots, T_n \in \A$, we define $E_\pi(T_1,\ldots, T_n) \in B$ via the following recursive process.
Let $V$ be the block of $\pi$ that terminates closest to the bottom, so $\min(V)$ is largest among all blocks of $\pi$. Then:
\begin{itemize}
\item If $\pi$ contains exactly one block (that is, $\pi = 1_\chi)$, we define $E_{1_\chi}(T_1, \ldots, T_n) = E(T_1 \cdots T_n)$.
\item If $V = \set{k+1, \ldots, n}$ for some $k \in \set{1, \ldots, n-1}$, then $\min(V)$ is not adjacent to any spines of $\pi$ and we define
\[
E_\pi(T_1, \ldots, T_n) = \left\{
\begin{array}{ll}
E_{\pi|_{V^c}}(T_1, \ldots, T_k L_{E_{\pi|_V}(T_{k+1},\ldots, T_n)}) & \mbox{if } \chi(\min(V)) = \ell  \\
E_{\pi|_{V^c}}(T_1, \ldots, T_k R_{E_{\pi|_V}(T_{k+1},\ldots, T_n)}) & \mbox{if } \chi(\min(V)) = r
\end{array} \right..
\]
In the long run, it will not matter if we choose $L$ or $R$ by the first part of this recursive definition and by Definition \ref{defn:BBncps}.
\item Otherwise, $\min(V)$ is adjacent to a spine. Let $W$ denote the block of $\pi$ corresponding to the spine adjacent to $\min(V)$, and
let $k$ be the first element of $W$ below where $V$ terminates -- that is, $k$ is the smallest element of $W$ that is larger than $\min(V)$.  We define
\[
E_\pi(T_1, \ldots, T_n) = \left\{
\begin{array}{ll}
E_{\pi|_{V^c}}((T_1, \ldots, T_{k-1}, L_{E_{\pi|_V}((T_{1},\ldots, T_n)|_V)} T_k, T_{k+1}, \ldots, T_n)|_{V^c}) & \mbox{if } \chi(\min(V)) = \ell  \\
E_{\pi|_{V^c}}((T_1, \ldots, T_{k-1}, R_{E_{\pi|_V}((T_{1},\ldots, T_n)|_V)} T_k, T_{k+1}, \ldots, T_n)|_{V^c}) & \mbox{if } \chi(\min(V)) = r
\end{array} \right..
\]
\end{itemize}
\end{defn}
\begin{exam}
Let $\pi$ be the following bi-non-crossing partition.
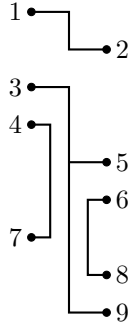
\begin{align*}
	\begin{tikzpicture}[baseline]
	\node[left] at (-.5, 4) {1};
	\draw[black,fill=black] (-.5,4) circle (0.05);
	\node[right] at (.5, 3.5) {2};
	\draw[black,fill=black] (.5,3.5) circle (0.05);
	\node[left] at (-.5, 3) {3};
	\draw[black,fill=black] (-.5,3) circle (0.05);
	\node[left] at (-.5, 2.5) {4};
	\draw[black,fill=black] (-.5,2.5) circle (0.05);
	\node[right] at (.5, 2) {5};
	\draw[black, fill=black] (.5,2) circle (0.05);
	\node[right] at (.5, 1.5) {6};
	\draw[black,fill=black] (.5,1.5) circle (0.05);
	\node[left] at (-.5, 1) {7};
	\draw[black, fill=black] (-.5,1) circle (0.05);
	\node[right] at (.5, .5) {8};
	\draw[black, fill=black] (.5,.5) circle (0.05);
	\node[right] at (.5,0) {9};
	\draw[black, fill=black] (.5,0) circle (0.05);
	\draw[thick, black] (-.5,4) -- (0,4) -- (0, 3.5) -- (.5,3.5);
	\draw[thick, black] (-.5,2.5) -- (-.25,2.5) -- (-.25, 1) -- (-.5,1);
	\draw[thick, black] (.5,.5) -- (0.25,.5) -- (0.25, 1.5) -- (.5,1.5);
	\draw[thick, black] (-.5,3) -- (0,3) -- (0, 0) -- (.5,0);
		\draw[thick, black] (.5,2) -- (0,2);
	\end{tikzpicture}
\end{align*}
Then
\[
E_\pi(T_1, \ldots, T_9) = E\left(T_1T_2 L_{E\left(T_3 L_{E(T_4T_7)}  T_5 R_{E(T_6T_8)} T_9 \right)}    \right)
\]
via the following sequence of diagrams (where $X = L_{E\left(T_3 L_{E(T_4T_7)}  T_5 R_{E(T_6T_8)} T_9 \right)}$):
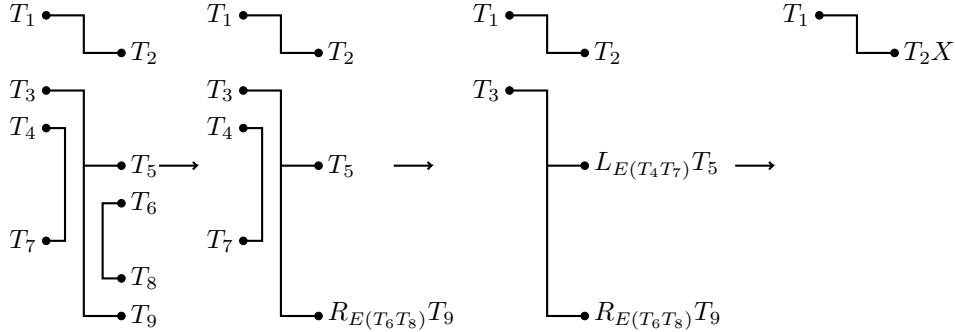
\begin{align*}
	\begin{tikzpicture}[baseline]
	\node[left] at (-.5, 4) {$T_1$};
	\draw[black,fill=black] (-.5,4) circle (0.05);
	\node[right] at (.5, 3.5) {$T_2$};
	\draw[black,fill=black] (.5,3.5) circle (0.05);
	\node[left] at (-.5, 3) {$T_3$};
	\draw[black,fill=black] (-.5,3) circle (0.05);
	\node[left] at (-.5, 2.5) {$T_4$};
	\draw[black,fill=black] (-.5,2.5) circle (0.05);
	\node[right] at (.5, 2) {$T_5$};
	\draw[black, fill=black] (.5,2) circle (0.05);
	\node[right] at (.5, 1.5) {$T_6$};
	\draw[black,fill=black] (.5,1.5) circle (0.05);
	\node[left] at (-.5, 1) {$T_7$};
	\draw[black, fill=black] (-.5,1) circle (0.05);
	\node[right] at (.5, .5) {$T_8$};
	\draw[black, fill=black] (.5,.5) circle (0.05);
	\node[right] at (.5,0) {$T_9$};
	\draw[black, fill=black] (.5,0) circle (0.05);
	\draw[thick, black] (-.5,4) -- (0,4) -- (0, 3.5) -- (.5,3.5);
	\draw[thick, black] (-.5,2.5) -- (-.25,2.5) -- (-.25, 1) -- (-.5,1);
	\draw[thick, black] (.5,.5) -- (0.25,.5) -- (0.25, 1.5) -- (.5,1.5);
	\draw[thick, black] (-.5,3) -- (0,3) -- (0, 0) -- (.5,0);
	\draw[thick, black] (.5,2) -- (0,2);
	\draw[thick] (1, 2) -- (1.5, 2) -- (1.45, 1.95);
	\draw[thick] (1.5, 2) -- (1.45, 2.05);
	\end{tikzpicture}
	\begin{tikzpicture}[baseline]
	\node[left] at (-.5, 4) {$T_1$};
	\draw[black,fill=black] (-.5,4) circle (0.05);
	\node[right] at (.5, 3.5) {$T_2$};
	\draw[black,fill=black] (.5,3.5) circle (0.05);
	\node[left] at (-.5, 3) {$T_3$};
	\draw[black,fill=black] (-.5,3) circle (0.05);
	\node[left] at (-.5, 2.5) {$T_4$};
	\draw[black,fill=black] (-.5,2.5) circle (0.05);
	\node[right] at (.5, 2) {$T_5$};
	\draw[black, fill=black] (.5,2) circle (0.05);
	\node[left] at (-.5, 1) {$T_7$};
	\draw[black,fill=black] (-.5,1) circle (0.05);
	\node[right] at (.5,0) {$R_{E(T_6T_8)}T_9$};
	\draw[black, fill=black] (.5,0) circle (0.05);
	\draw[thick, black] (-.5,4) -- (0,4) -- (0, 3.5) -- (.5,3.5);
	\draw[thick, black] (-.5,2.5) -- (-.25,2.5) -- (-.25, 1) -- (-.5,1);
	\draw[thick, black] (-.5,3) -- (0,3) -- (0, 0) -- (.5,0);
	\draw[thick, black] (.5,2) -- (0,2);
	\draw[thick] (1.5, 2) -- (2, 2) -- (1.95, 1.95);
	\draw[thick] (2, 2) -- (1.95, 2.05);
	\end{tikzpicture}
	\begin{tikzpicture}[baseline]
	\node[left] at (-.5, 4) {$T_1$};
	\draw[black,fill=black] (-.5,4) circle (0.05);
	\node[right] at (.5, 3.5) {$T_2$};
	\draw[black,fill=black] (.5,3.5) circle (0.05);
	\node[left] at (-.5, 3) {$T_3$};
	\draw[black,fill=black] (-.5,3) circle (0.05);
	\node[right] at (.5, 2) {$L_{E(T_4T_7)}T_5$};
	\draw[black,fill=black] (.5,2) circle (0.05);
	\node[right] at (.5,0) {$R_{E(T_6T_8)}T_9$};
	\draw[black, fill=black] (.5,0) circle (0.05);
	\draw[thick, black] (-.5,4) -- (0,4) -- (0, 3.5) -- (.5,3.5);
	\draw[thick, black] (-.5,3) -- (0,3) -- (0, 0) -- (.5,0);
	\draw[thick, black] (.5,2) -- (0,2);
	\draw[thick] (2.5, 2) -- (3, 2) -- (2.95, 1.95);
	\draw[thick] (3, 2) -- (2.95, 2.05);
	\end{tikzpicture}
	\begin{tikzpicture}[baseline]
	\node[left] at (-.5, 4) {$T_1$};
	\draw[black,fill=black] (-.5,4) circle (0.05);
	\node[right] at (.5, 3.5) {$T_2X$};
	\draw[black,fill=black] (.5,3.5) circle (0.05);
	\draw[thick, black] (-.5,4) -- (0,4) -- (0, 3.5) -- (.5,3.5);
	\end{tikzpicture}.
\end{align*}
\end{exam}

Note that the definition of $E_\pi(T_1,\ldots, T_n)$ is invariant under $B$-$B$-non-commutative probability space embeddings, such as those listed in Theorem \ref{thm:representingbbncps}.
Observe that in the context of Definition \ref{defn:recursivedefinitionofEpi}, we ignore the notions of left and right operators.
However, we are ultimately interested in the following.
\begin{defn}
Let $(\mathcal{A}, E, \varepsilon)$ be a $B$-$B$-non-commutative probability space.  The \emph{bi-free operator-valued moment function}
\[
\mathcal{E} : \bigcup_{n\geq 1} \bigcup_{\chi : \{1,\ldots, n\} \to \{\ell, r\}} BNC(\chi) \times \mathcal{A}_{\chi(1)} \times \cdots \times \mathcal{A}_{\chi(n)} \to B
\]
is defined by
\[
\mathcal{E}_\pi(T_1, \ldots, T_n) = E_\pi(T_1, \ldots, T_n)
\]
for each $\chi : \{1, \ldots, n\} \to \{\ell, r\}$, $\pi \in BNC(\chi)$, and $T_k \in \mathcal{A}_{\chi(k)}$.
\end{defn}

Our next goal is the prove the following which is not apparent from Definition \ref{defn:recursivedefinitionofEpi}.

\begin{thm}
\label{thm:samantha}
The operator-valued bi-free moment function $\mathcal{E}$ on $\A$ is bi-multiplicative.
\end{thm}

We divide the proof of the above theorem into several lemmata, verifying various of properties from Definition \ref{defnbimultiplicative}.  Properties (i) and (ii) are immediate but, unfortunately, the remaining properties are not as easily verified.


\begin{lem}
\label{lemeasy}
The operator-valued bi-free moment function $\mathcal{E}$ satisfies Properties (i) and (ii) of Definition \ref{defnbimultiplicative}.
\end{lem}

\subsection{Verification of Property (iii) from Definition \ref{defnbimultiplicative} for $\mathcal{E}$}

\begin{lem}
\label{lemproductofdisjointblocksforE}
The operator-valued bi-free moment function $\mathcal{E}$ satisfies Property (iii) of Definition \ref{defnbimultiplicative}.
\end{lem}
\begin{proof}
We claim it suffices to consider the case that $\min_{\prec_\chi}(V_k)$ and $\max_{\prec_\chi}(V_k)$ are in the same block of $\pi$ for each $k \in \{1,\ldots, m\}$.
Indeed, suppose the result holds with this additional assumption.
Fix $V_1, \ldots, V_m$ satisfying the assumptions of Property (iii) of Definition \ref{defnbimultiplicative}.
For each $k$, partition $V_k$ into $\chi$-intervals $V_{k,1}, \ldots, V_{k,p_k}$ composed of blocks of $\pi$ ordered by $\prec_\chi$, with $\min_{\prec_\chi}(V_{k,q}) \sim_\pi \max_{\prec_\chi}(V_{k,q})$.

Since $\bigcup^m_{k=1} \bigcup^{p_k}_{q=1} V_{k,q}$ satisfies the assumptions of Property (iii) of Definition \ref{defnbimultiplicative} and the additional assumption, one obtains
\[
\mathcal{E}_\pi(T_1, \ldots, T_n) = \left(\prod^{p_1}_{q=1} \mathcal{E}_{\pi|_{V_{1,q}}}\left((T_1, \ldots, T_n)|_{V_{1,q}}\right)\right) \cdots\left(\prod^{p_m}_{q=1} \mathcal{E}_{\pi|_{V_{m,q}}}\left((T_1, \ldots, T_n)|_{V_{m,q}}\right)\right).
\]
In each of the above products, we write the terms from left to right in order of increasing $q$.
Again, by applying the case where the additional assumption holds, one obtains that
\[
\left(\prod^{p_k}_{q=1} \mathcal{E}_{\pi|_{V_{k,q}}}\left((T_1, \ldots, T_n)|_{V_{k,q}}\right)\right) = \mathcal{E}_{\pi|_{V_k}}\left((T_1, \ldots, T_n)|_{V_k}\right)
\]
for each $k \in \{1,\ldots, m\}$.
Hence we may assume that $\min_{\prec_\chi}(V_k)\sim_\pi\max_{\prec_\chi}(V_k)$ for each $k \in \{1,\ldots, m\}$.
\par 
To continue the proof, we proceed by induction on $m$ with the case $m = 1$ being trivial.  Assume Property (iii) of Definition \ref{defnbimultiplicative} is satisfied for $\mathcal{E}$ for all smaller values of $m$.
Fix $V_1, \ldots, V_m$ and note that either $1 \in V_1$ (i.e. $\chi(1) = \ell$) or $1 \in V_m$ (i.e. $\chi(1) = r$).
We will treat the case when $1 \in V_1$; for the other case, consult a mirror.
Let $V'_1 \subseteq V_1$ be the block of $\pi$ containing $1$ and $\max_{\prec_\chi}(V_1)$.
The proof is divided into three cases.
\par
\underline{\textit{Case 1: $\min(V_k) > \max(V_1)$ for all $k \neq 1$.}}  As an example of this case, consider the following diagram where $V_1 = \{1, 2, 3 \}$, $V_2 = \{4, 6\}$, and $V_3 = \{5, 7, 8, 9\}$.
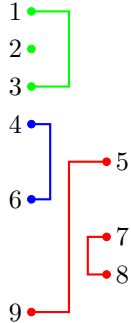
\begin{align*}
	\begin{tikzpicture}[baseline]
	\node[left] at (-.5, 4) {1};
	\draw[green,fill=green] (-.5,4) circle (0.05);
	\node[left] at (-.5, 3.5) {2};
	\draw[green,fill=green] (-.5,3.5) circle (0.05);
	\node[left] at (-.5, 3) {3};
	\draw[green,fill=green] (-.5,3) circle (0.05);
	\node[left] at (-.5, 2.5) {4};
	\draw[blue,fill=blue] (-.5,2.5) circle (0.05);
	\node[right] at (.5, 2) {5};
	\draw[red, fill=red] (.5,2) circle (0.05);
	\node[left] at (-.5, 1.5) {6};
	\draw[blue,fill=blue] (-.5,1.5) circle (0.05);
	\node[right] at (.5, 1) {7};
	\draw[red, fill=red] (.5,1) circle (0.05);
	\node[right] at (.5, .5) {8};
	\draw[red, fill=red] (.5,.5) circle (0.05);
	\node[left] at (-.5,0) {9};
	\draw[red, fill=red] (-.5,0) circle (0.05);
	\draw[thick, green] (-.5,4) -- (0,4) -- (0, 3) -- (-.5,3);
	\draw[thick, blue] (-.5,2.5) -- (-.25,2.5) -- (-.25, 1.5) -- (-.5,1.5);
	\draw[thick, red] (.5,.5) -- (0.25,.5) -- (0.25, 1) -- (.5,1);
	\draw[thick, red] (.5,2) -- (0,2) -- (0, 0) -- (-.5,0);
	\end{tikzpicture}
\end{align*}
In this case, drawing a horizontal line directly beneath $\max(V_1)$ will hit no spines in $\pi$ and $V_1 \subseteq \chi^{-1}(\{\ell\})$.
Let $V'_1 = \{1 = q_1 < q_2 < \cdots < q_p\}$ and $V_0 = \bigcup^m_{k=2} V_k$.
By Definition \ref{defn:recursivedefinitionofEpi} we may find $b_1, \ldots, b_{p-1} \in B$ depending only on $(T_1, \ldots, T_n)|_{V_1}$ and $\pi|_{V_1}$, so that writing $T'_{q_k} = T_{q_k}L_{b_k}$ we have
\[
E_\pi(T_1, \ldots, T_n)
= E\left(T'_{q_1} T'_{q_2} \cdots T'_{q_{p-1}}T_{q_p} L_{E_{\pi|_{V_0}}((T_1, \ldots, T_n)|_{V_0})} \right)
= E\left(T'_{q_1} T'_{q_2} \cdots T'_{q_{p-1}}T_{q_p} R_{E_{\pi|_{V_0}}((T_1, \ldots, T_n)|_{V_0})} \right).
\]
By the assumptions in this case, each $T_k \in \mathcal{A}_\ell$ for all $k \in V'_1$ and since right $B$-operators commute with elements of $\mathcal{A}_\ell$, we obtain
\begin{align*}
E_\pi(T_1, \ldots, T_n) &= E\left(T'_{q_1} T'_{q_2} \cdots T'_{q_p}R_{E_{\pi|_{V_0}}((T_1, \ldots, T_n)|_{V_0})} \right)  \\
 &= E\left(R_{E_{\pi|_{V_0}}((T_1, \ldots, T_n)|_{V_0})}T'_{q_1} T'_{q_2} \cdots T'_{q_p} \right) \\
 &= E(T'_{q_1} T'_{q_2} \cdots T'_{q_p}) E_{\pi|_{V_0}}((T_1, \ldots, T_n)|_{V_0}) \\
 &=  E_{\pi|_{V_1}}((T_1, \ldots, T_n)|_{V_1})E_{\pi|_{V_0}}((T_1, \ldots, T_n)|_{V_0}) \\
 &= \mathcal{E}_{\pi|_{V_1}}\left((T_1, \ldots, T_n)|_{V_1}\right) \mathcal{E}_{\pi|_{V_2}}\left((T_1, \ldots, T_n)|_{V_1}\right) \cdots \mathcal{E}_{\pi|_{V_m}}\left((T_1, \ldots, T_n)|_{V_m}\right)
 \end{align*}
with the last step following by the inductive hypothesis.

\underline{\textit{Case 2: There exists a $k \neq 1$ such that $\min(V_k) < \max(V_1)$.}}  In this case $V_m$ must terminate on the right above $\max(V_1)$; that is $\min(V_m) < \max(V_1)$ and $\chi(\min(V_m)) = r$.
We divide the proof into two further cases.

\underline{\textit{Case 2a: $\max(V_1) < \max(V_m)$.}}  As an example of this case, consider the following diagram where $V_1 = \{1, 3 \}$, $V_2 = \{4, 6, 7, 8, 9\}$, and $V_3 = \{2, 5\}$.
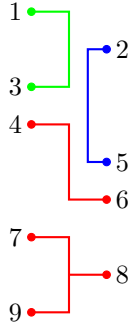
\begin{align*}
	\begin{tikzpicture}[baseline]
	\node[left] at (-.5, 4) {1};
	\draw[green,fill=green] (-.5,4) circle (0.05);
	\node[right] at (.5, 3.5) {2};
	\draw[blue,fill=blue] (.5,3.5) circle (0.05);
	\node[left] at (-.5, 3) {3};
	\draw[green,fill=green] (-.5,3) circle (0.05);
	\node[left] at (-.5, 2.5) {4};
	\draw[red,fill=red] (-.5,2.5) circle (0.05);
	\node[right] at (.5, 2) {5};
	\draw[blue, fill=blue] (.5,2) circle (0.05);
	\node[right] at (.5, 1.5) {6};
	\draw[red,fill=red] (.5,1.5) circle (0.05);
	\node[left] at (-.5, 1) {7};
	\draw[red, fill=red] (-.5,1) circle (0.05);
	\node[right] at (.5, .5) {8};
	\draw[red, fill=red] (.5,.5) circle (0.05);
	\node[left] at (-.5,0) {9};
	\draw[red, fill=red] (-.5,0) circle (0.05);
	\draw[thick, green] (-.5,4) -- (0,4) -- (0, 3) -- (-.5,3);
	\draw[thick, blue] (.5,3.5) -- (0.25,3.5) -- (0.25, 2) -- (.5,2);
	\draw[thick, red] (-.5,.0) -- (-0,0) -- (-0, 1) -- (-.5,1);
	\draw[thick, red] (.5,.5) -- (-0,.5);
	\draw[thick, red] (-.5,2.5) -- (0,2.5) -- (0, 1.5) -- (.5,1.5);
	\end{tikzpicture}
\end{align*}
Again $V_1 \subseteq \chi^{-1}(\{\ell\})$.
With the same conventions as above, by Definition \ref{defn:recursivedefinitionofEpi} we obtain
\[
E_\pi(T_1, \ldots, T_n) = E\left(T'_{q_1} T'_{q_2} \cdots T'_{q_{p_1}} R_{E_{\pi|_{V_0}}((T_1, \ldots, T_n)|_{V_0})} T'_{q_{p_1+1}} \cdots T'_{q_{p_2}}\right),
\]
where $p_1$ is the smallest element of $V_1'$ greater than $\min(V_m)$.
By the assumptions in this case, each $T_k \in \mathcal{A}_\ell$ for all $k \in V'_1$, and since right $B$-operators commute with elements of $\mathcal{A}_\ell$, one obtains
\begin{align*}
E_\pi(T_1, \ldots, T_n) &=  E\left(T'_{q_1} T'_{q_2} \cdots T'_{q_{p_1}} R_{E_{\pi|_{V_0}}((T_1, \ldots, T_n)|_{V_0})} T'_{q_{p_1+1}} \cdots T'_{q_{p_2}}\right) \\
 &= E\left(R_{E_{\pi|_{V_0}}((T_1, \ldots, T_n)|_{V_0})}T'_{q_1} T'_{q_2} \cdots T'_{q_{p_2}} \right) \\
 &= E(T'_{q_1} T'_{q_2} \cdots T'_{q_{p_2}}) E_{\pi|_{V_0}}((T_1, \ldots, T_n)|_{V_0}) \\
 &=  E_{\pi|_{V_1}}((T_1, \ldots, T_n)|_{V_1})E_{\pi|_{V_0}}((T_1, \ldots, T_n)|_{V_0})\\
&= \mathcal{E}_{\pi|_{V_1}}\left((T_1, \ldots, T_n)|_{V_1}\right) \mathcal{E}_{\pi|_{V_2}}\left((T_1, \ldots, T_n)|_{V_1}\right) \cdots \mathcal{E}_{\pi|_{V_m}}\left((T_1, \ldots, T_n)|_{V_m}\right),
\end{align*}
with the last step following by the inductive hypothesis.

\underline{\textit{Case 2b: $\max(V_1) > \max(V_m)$.}}  As an example of this case, consider the following diagram where $V_1 = \{1, 5 \}$, $V_2 = \{ 6, 8, 9\}$, $V_3 = \{4, 7\}$, and $V_4 = \{2, 3\}$.
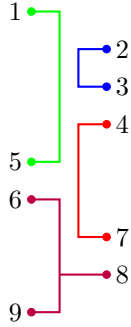
\begin{align*}
	\begin{tikzpicture}[baseline]
	\node[left] at (-.5, 4) {1};
	\draw[green,fill=green] (-.5,4) circle (0.05);
	\node[right] at (.5, 3.5) {2};
	\draw[blue,fill=blue] (.5,3.5) circle (0.05);
	\node[right] at (.5, 3) {3};
	\draw[blue,fill=blue] (.5,3) circle (0.05);
	\node[right] at (.5, 2.5) {4};
	\draw[red,fill=red] (.5,2.5) circle (0.05);
	\node[left] at (-.5, 2) {5};
	\draw[green, fill=green] (-.5,2) circle (0.05);
	\node[left] at (-.5, 1.5) {6};
	\draw[purple,fill=purple] (-.5,1.5) circle (0.05);
	\node[right] at (.5, 1) {7};
	\draw[red, fill=red] (.5,1) circle (0.05);
	\node[right] at (.5, .5) {8};
	\draw[purple, fill=purple] (.5,.5) circle (0.05);
	\node[left] at (-.5,0) {9};
	\draw[purple, fill=purple] (-.5,0) circle (0.05);
	\draw[thick, green] (-.5,4) -- (-.1250,4) -- (-.1250, 2) -- (-.5,2);
	\draw[thick, blue] (.5,3.5) -- (0.125,3.5) -- (0.125, 3) -- (.5,3);
	\draw[thick, purple] (-.5,.0) -- (-0.125,0) -- (-0.125, 1.5) -- (-.5,1.5);
	\draw[thick, purple] (.5,.5) -- (-0.125,.5);
	\draw[thick, red] (.5,2.5) -- (0.125,2.5) -- (0.125, 1) -- (.5,1);
	\end{tikzpicture}
\end{align*}
Let $V_0 = \bigcup^{m-1}_{k=1} V_k$.
By Definition \ref{defn:recursivedefinitionofEpi} we may again find $T'_{q_1}, \ldots, T'_{q_{p_1}}$ and $S \in \A$ where $T'_k$ differs from $T_k$ by a left multiplication operator, so that
\[
E_\pi(T_1, \ldots, T_n) = E\left(T'_{q_1} T'_{q_2} \cdots T'_{q_{p_1}} R_{E_{\pi|_{V_m}}((T_1, \ldots, T_n)|_{V_m})} S\right). 
\]
Since right $B$-operators commute with elements of $\mathcal{A}_\ell$, one obtains
\begin{align*}
E_\pi(T_1, \ldots, T_n) &=  E\left(T'_{q_1} T'_{q_2} \cdots T'_{q_{p_1}} R_{E_{\pi|_{V_m}}((T_1, \ldots, T_n)|_{V_m})} S\right) \\
 &= E\left(R_{E_{\pi|_{V_m}}((T_1, \ldots, T_n)|_{V_m})}T'_{q_1} T'_{q_2} \cdots T'_{q_{p_1}}S \right) \\
 &= E(T'_{q_1} T'_{q_2} \cdots T'_{q_{p_1}}S) E_{\pi|_{V_m}}((T_1, \ldots, T_n)|_{V_m}) \\
 &=  E_{\pi|_{V_0}}((T_1, \ldots, T_n)|_{V_0})E_{\pi|_{V_m}}((T_1, \ldots, T_n)|_{V_m})\\
 &= \mathcal{E}_{\pi|_{V_1}}\left((T_1, \ldots, T_n)|_{V_1}\right) \mathcal{E}_{\pi|_{V_2}}\left((T_1, \ldots, T_n)|_{V_1}\right) \cdots \mathcal{E}_{\pi|_{V_m}}\left((T_1, \ldots, T_n)|_{V_m}\right)
  \end{align*}
with the last step following by the inductive hypothesis.
\end{proof}

\subsection{Verification of Property (iv) from Definition \ref{defnbimultiplicative} for $\mathcal{E}$}

We begin with the following intermediate step on the way to verifying that $\mathcal{E}$ satisfies Property (iv).

\begin{lem}
\label{lemreductionofbimultiplicativeinsideablockforE}
The operator-valued bi-free moment function $\mathcal{E}$ satisfies Property (iv) of Definition \ref{defnbimultiplicative} with the additional assumption that there exists a block $W_0 \subseteq W$ of $\pi$ such that 
\[
\theta, \gamma, \min_{\prec_\chi}(\{1,\ldots, n\}),  \max_{\prec_\chi}(\{1,\ldots, n\}) \in W_0.
\]
\end{lem}
\begin{proof}
We will present only the proof of the case $\chi(\theta) = \ell$ as the other cases are similar.

It is easy to see there exists a partition $V_1, \ldots, V_m$ of $V$ by $\chi$-intervals, ordered by $\prec_\chi$, with $V_k$ a union of blocks of $\pi$ such that $\min_{\prec_\chi}(V_k) \sim_\pi \max_{\prec_\chi}(V_k)$.
Recall that
\[
\theta := \max_{\prec_\chi}\left(\left\{k \in W  \, \mid \, k \prec_\chi \min_{\prec_\chi}(V)\right\}\right)
\qquad\text{and}\qquad
\gamma := \min_{\prec_\chi}\left(\left\{k \in W  \, \mid \, \max_{\prec_\chi}(V) \prec_\chi k\right\}\right),
\]
and so under $\prec_\chi$, $\theta$ immediately precedes $\min_{\prec_\chi}(V_1)$ and $\gamma$ immediately follows $\max_{\prec_\chi}(V_m)$.

The proof is now divided into three cases.

\underline{\textit{Case 1: $\chi(\gamma) = \ell$.}}  As an example of this case, consider the following diagram where $W = W_0 = \{1, 5, 9 \}$, $V_1 = \{2, 3 \}$, $V_2 = \{4, 6, 7, 8\}$, $\theta = 1$, and $\gamma = 9$.
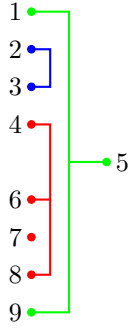
\begin{align*}
	\begin{tikzpicture}[baseline]
	\node[left] at (-.5, 4) {1};
	\draw[green,fill=green] (-.5,4) circle (0.05);
	\node[left] at (-.5, 3.5) {2};
	\draw[blue,fill=blue] (-.5,3.5) circle (0.05);
	\node[left] at (-.5, 3) {3};
	\draw[blue,fill=blue] (-.5,3) circle (0.05);
	\node[left] at (-.5, 2.5) {4};
	\draw[red,fill=red] (-.5,2.5) circle (0.05);
	\node[right] at (.5, 2) {5};
	\draw[green,fill=green] (.5,2) circle (0.05);
	\node[left] at (-.5, 1.5) {6};
	\draw[red,fill=red] (-.5,1.5) circle (0.05);
	\node[left] at (-.5, 1) {7};
	\draw[red,fill=red] (-.5,1) circle (0.05);
	\node[left] at (-.5, .5) {8};
	\draw[red,fill=red] (-.5,.5) circle (0.05);
	\node[left] at (-.5,0) {9};
	\draw[green,fill=green] (-.5,0) circle (0.05);
	\draw[thick, green] (-.5,4) -- (0,4) -- (0, 0) -- (-.5,0);
	\draw[thick, green] (0,2) -- (.5,2);
	\draw[thick, blue] (-.5,3.5) -- (-.25,3.5) -- (-.25, 3) -- (-.5,3);
	\draw[thick, red] (-.5,2.5) -- (-.25,2.5) -- (-.25, .5) -- (-.5,.5);
	\draw[thick, red] (-.5,1.5) -- (-0.25,1.5);
	\end{tikzpicture}
\end{align*}

In this case $V_k \subseteq \chi^{-1}(\{\ell\})$ for all $k \in \{1,\ldots, m\}$.
Write $X_k = L_{E_{\pi|_{V_k}}((T_1, \ldots, T_n)|_{V_k})}$ and $W_0 = \{q_1 < q_2 < \cdots < q_{k_{m+1}}\}$.
Then
\[
E_\pi(T_1, \ldots, T_n) = E\left( T'_{q_1} T'_{q_2} \cdots T'_{q_{k_1}} X_1 T'_{q_{k_1+1}} \cdots T'_{q_{k_m}} X_m  T'_{q_{k_m+1}} \cdots T'_{q_{k_{m+1}}}  \right),
\]
where $T_k'$ is $T_k$, potentially multiplied on the left and/or right by appropriate $L_b$ and $R_b$.
Here $T_\theta$ appears left of $X_1$, $\gamma = q_{k_{m+1}}$, and every operator between the two is either some $X_k$ or a right operator.
Hence, by the commutation of left $B$-operators with elements of $\mathcal{A}_r$, we obtain
\[
E_\pi(T_1, \ldots, T_n) = E\left( T'_{q_1} \cdots T'_{q_{j-1}}   R_{b}L_{b'}\left(T_{\theta} X_1 X_2 \cdots X_m\right) R_{b''} T'_{q_{j+1}} \cdots T'_{q_{k_{m+1}}}  \right)
\]
for some $b,b',b'' \in B$.  Since
\[
\mathcal{E}_{\pi|_{V_1}}((T_1, \ldots, T_n)|_{V_1}) \cdots \mathcal{E}_{\pi|_{V_m}}((T_1, \ldots, T_n)|_{V_m}) = \mathcal{E}_{\pi|_{V}}((T_1, \ldots, T_n)|_{V}),
\]
by Lemma \ref{lemproductofdisjointblocksforE} we have
\begin{align*}
\mathcal{E}_\pi(T_1, \ldots, T_n) &= E\left(  T'_{q_1} \cdots T'_{q_{j-1}}   R_{b}L_{b'}\left(T_{\theta} L_{\mathcal{E}_{\pi|_{V}}((T_1, \ldots, T_n)|_{V})}\right) R_{b''} T'_{q_{j+1}} \cdots T'_{q_{k_{m+1}}} \right) \\
&= \mathcal{E}_{\pi|_{W}}\left(\left.\left(T_1, \ldots, T_{\theta-1}, T_\theta L_{\mathcal{E}_{\pi|_{V}}\left((T_1,\ldots, T_n)|_{V}\right)}, T_{\theta+1}, \ldots, T_n\right)\right|_{W}\right)
\end{align*}
where the last step follows as $\mc E_{\pi |_W}$ ignores arguments corresponding to $V$.

\underline{\textit{Case 2: $\chi(\gamma) = r$ and $\theta < \gamma$.}}  As an example of this case, consider the following diagram where $W = W_0 = \{1, 3, 6 \}$, $V_1 = \{2, 4 \}$, $V_2 = \{5,  8\}$, $V_3 = \{7, 9\}$, $\theta = 1$, and $\gamma = 6$.
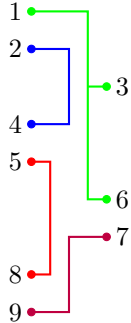
\begin{align*}
	\begin{tikzpicture}[baseline]
	\node[left] at (-.5, 4) {1};
	\draw[green,fill=green] (-.5,4) circle (0.05);
	\node[left] at (-.5, 3.5) {2};
	\draw[blue,fill=blue] (-.5,3.5) circle (0.05);
	\node[right] at (.5, 3) {3};
	\draw[green,fill=green] (.5,3) circle (0.05);
	\node[left] at (-.5, 2.5) {4};
	\draw[blue,fill=blue] (-.5,2.5) circle (0.05);
	\node[left] at (-.5, 2) {5};
	\draw[red,fill=red] (-.5,2) circle (0.05);
	\node[right] at (.5, 1.5) {6};
	\draw[green,fill=green] (.5,1.5) circle (0.05);
	\node[right] at (.5, 1) {7};
	\draw[purple,fill=purple] (.5,1) circle (0.05);
	\node[left] at (-.5, .5) {8};
	\draw[red,fill=red] (-.5,.5) circle (0.05);
	\node[left] at (-.5,0) {9};
	\draw[purple,fill=purple] (-.5,0) circle (0.05);
	\draw[thick, green] (-.5,4) -- (0.25,4) -- (0.25, 1.5) -- (.5,1.5);
	\draw[thick, green] (0.25,3) -- (.5,3);
	\draw[thick, blue] (-.5,3.5) -- (0,3.5) -- (0, 2.5) -- (-.5,2.5);
	\draw[thick, red] (-.5,2) -- (-.25,2) -- (-.25, .5) -- (-.5,.5);
	\draw[thick, purple] (.5,1) -- (-0,1) -- (-0, 0) -- (-.5,0);
	\end{tikzpicture}
\end{align*}

Since $\chi(\theta) = \ell$ and $\chi(\gamma) = r$ there exists a $p \in \{1,\ldots, m\}$ such that $\min(V_k) > \gamma$ if and only if $k > p$, $V_k \subseteq \chi^{-1}(\{\ell\})$ for all $k < p$, and $\chi(\min(V_p)) = \ell$.  Let
\[
Y = \left\{
\begin{array}{ll}
\bigcup_{k \geq p} V_k & \mbox{if } p < m \mbox{ and } \min(V_{p+1}) < \max(V_p)  \\
\bigcup_{k > p} V_k & \mbox{otherwise}
\end{array} \right..
\]
For example, $Y = \{5, 7, 8, 9\}$ in the above diagram.

Let us deal with the first possible case for $Y$, with the second following similarly.
In this case $V_k \subseteq \chi^{-1}(\{\ell\})$ for all $k \in \{1,\ldots, p-1\}$.
Write $X_k = L_{E_{\pi|_{V_k}}((T_1, \ldots, T_n)|_{V_k})}$ and $W_0 = \{q_1 < q_2 < \cdots < q_{k_{p+1}}\}$.
Then
\begin{align*}
E_\pi(T_1, \ldots, T_n) = E\left( T'_{q_1} \cdots T'_{q_{k_1}} X_1 T'_{q_{k_1+1}} \cdots T'_{q_{k_{p-1}}} X_{p-1}  T'_{q_{k_{p-1}+1}} \cdots T'_{q_{k_{p}}} L_{E_{\pi|_{Y}}((T_1, \ldots, T_n)|_{Y})} T'_{q_{k_{p}+1}} \cdots T'_{q_{k_{p+1}}} \right)
\end{align*}
where $T_k'$ is $T_k$, potentially multiplied on the left and/or right by appropriate $L_b$ and $R_b$.
Here $T_\theta$ appears left of $X_1$, $\gamma = q_{k_{p+1}}$, and every operator between the two is either some $X_k$, $L_{E_{\pi|_{Y}}((T_1, \ldots, T_n)|_{Y})}$, or a right operator.
Hence, by the commutation of left $B$-operators with elements of $\mathcal{A}_r$, we obtain
\begin{align*}
E_\pi(T_1, \ldots, T_n) = E\left( T'_{q_1} \cdots T'_{q_{j-1}}R_{b} L_{b'} \left(T_{\theta} X_1 \cdots X_p L_{E_{\pi|_{Y}}((T_1, \ldots, T_n)|_{Y})} \right) R_{b''} T'_{q_{j+1}}  \cdots T'_{q_{k_{p+1}}} \right)
\end{align*}
for some $b, b', b'' \in B$.  Since 
\[
E_{\pi|_{V_1}}((T_1, \ldots, T_n)|_{V_1}) \cdots E_{\pi|_{V_p}}((T_1, \ldots, T_n)|_{V_p}) E_{\pi|_{Y}}((T_1, \ldots, T_n)|_{Y}) = E_{\pi|_{V}}((T_1,\ldots, T_n)|_{V}),
\]
by Lemma \ref{lemproductofdisjointblocksforE} we have
\begin{align*}
\mathcal{E}_\pi(T_1, \ldots, T_n) &= E\left(  T'_{q_1} \cdots T'_{q_{j-1}}   R_{b}L_{b'}\left(T_{\theta} L_{\mathcal{E}_{\pi|_{V}}((T_1, \ldots, T_n)|_{V})}\right) R_{b''} T'_{q_{j+1}} \cdots T'_{q_{k_{p+1}}} \right) \\
&= \mathcal{E}_{\pi|_{W}}\left(\left.\left(T_1, \ldots, T_{\theta-1}, T_\theta L_{\mathcal{E}_{\pi|_{V}}\left((T_1,\ldots, T_n)|_{V}\right)}, T_{\theta+1}, \ldots, T_n\right)\right|_{W}\right),
\end{align*}
where the last step follows as $\mc E_{\pi |_W}$ ignores arguments corresponding to $V$.

\underline{\textit{Case 3: $\chi(\gamma) = r$ and $\theta > \gamma$.}}  This case is a reflection of the second case plus a small argument.  Since $\chi(\theta) = \ell$ and $\chi(\gamma) = r$ there exists a $p \in \{1,\ldots, m\}$ such that $\min(V_k) > \gamma$ if and only $k < p$, $V_k \subseteq \chi^{-1}(\{r\})$ for all $k > p$, and $\chi(\min(V_p)) = r$.  Let
\[
Y = \left\{
\begin{array}{ll}
\bigcup_{k \leq p} V_k & \mbox{if } p > 1 \mbox{ and } \min(V_{p-1}) < \max(V_p)  \\
\bigcup_{k < p} V_k & \mbox{otherwise}
\end{array} \right. .
\]
Let us deal with the second possible case for $Y$, with the first following similarly.
In this case $V_k \subseteq \chi^{-1}(\{\ell\})$ for all $k \in \{p+1,\ldots, m\}$.
Write $X_k = R_{E_{\pi|_{V_k}}((T_1, \ldots, T_n)|_{V_k})}$ and $W_0 = \{q_1 < q_2 < \cdots < q_{k_{m-p+2}}\}$.
Then
\begin{align*}
E_\pi(T_1, \ldots, T_n) = E\left( T'_{q_1} \cdots T'_{q_{k_1}} X_m T'_{q_{k_1+1}} \cdots T'_{q_{k_{m-p+1}}} X_{p}  T'_{q_{k_{m-p+1}+1}} \cdots T'_{q_{k_{m-p+2}}} R_{E_{\pi|_{Y}}((T_1, \ldots, T_n)|_{Y})} \right)
\end{align*}
where $T_k'$ is $T_k$, potentially multiplied on the left and/or right by appropriate $L_b$ and $R_b$.
Here $\theta = q_{k_{m-p+2}}$, $T_\gamma$ occurs left of $X_m$, and every operator between the two is either some $X_k$ or a left operator.
Hence, by the commutation of right $B$-operators with elements of $\mathcal{A}_\ell$, one obtains
\begin{align*}
E_\pi(T_1, \ldots, T_n) &= E\left( T'_{q_1} \cdots T'_{q_{k_{m-p+2}-1}} L_{b} T_\theta X_m \cdots X_{p} R_{E_{\pi|_{Y}}((T_1, \ldots, T_n)|_{Y})} \right)\\
&= E\left(  T'_{q_1} \cdots T'_{q_{jk_{m-p+2}-1}} L_{b} T_\theta  R_{E_{\pi|_{Y}}((T_1, \ldots, T_n)|_{Y})E_{\pi|_{V_{p}}}((T_1, \ldots, T_n)|_{V_{p}})    \cdots E_{\pi|_{V_m}}((T_1, \ldots, T_n)|_{V_m})}\right)\\
&= E\left(  T'_{q_1} \cdots T'_{q_{jk_{m-p+2}-1}} L_{b} T_\theta  L_{E_{\pi|_{Y}}((T_1, \ldots, T_n)|_{Y})E_{\pi|_{V_{p}}}((T_1, \ldots, T_n)|_{V_p})    \cdots  E_{\pi|_{V_m}}((T_1, \ldots, T_n)|_{V_m})}\right)
\end{align*}
for some $b \in B$.  Since 
\[
E_{\pi|_{Y}}((T_1, \ldots, T_n)|_{Y})E_{\pi|_{V_{p}}}((T_1, \ldots, T_n)|_{V_{p}})    \cdots  E_{\pi|_{V_m}}((T_1, \ldots, T_n)|_{V_m}) = E_{\pi|_{V}}((T_1,\ldots, T_n)|_{V})
\]
by Lemma \ref{lemproductofdisjointblocksforE}, we have
\begin{align*}
\mathcal{E}_\pi(T_1, \ldots, T_n) &= E\left( T'_{q_1} \cdots T'_{q_{jk_{m-p+2}-1}} L_{b'} T_\theta   L_{\mathcal{E}_{\pi|_{V}}((T_1, \ldots, T_n)|_{V})} \right) \\
&= \mathcal{E}_{\pi|_{W}}\left(\left.\left(T_1, \ldots, T_{\theta-1}, T_\theta L_{\mathcal{E}_{\pi|_{V}}\left((T_1,\ldots, T_n)|_{V}\right)}, T_{\theta+1}, \ldots, T_n\right)\right|_{W}\right)
\end{align*}
where the last step follows as $\mc E_{\pi |_W}$ ignores arguments corresponding to $V$.
\end{proof}

In addition to the above, we will need to verify slightly enhanced versions of Properties (i) and (ii) of Definition \ref{defnbimultiplicative} for $\mathcal{E}$.

\begin{lem}
\label{lemenhancedpropertyiforE}
The operator-valued bi-free moment function $\mathcal{E}$ satisfies the $q = -\infty$ case of Property (i) of Definition \ref{defnbimultiplicative} when $1_\chi$ is replaced with an arbitrary $\pi \in BNC(\chi)$.
\end{lem}

\begin{proof}
We will assume $\chi(n) = \ell$ as the case where $\chi(n) = r$ will follow \emph{mutatis mutandis}.
In the case $\chi(n) = \ell$, it is easy to see that $\chi \equiv \ell$.
Let $V_1, \ldots, V_m$ be a partition of $\set{1, \ldots, n}$ into $\chi$-intervals ordered by $\prec_\chi$ with each $V_k$ a union of blocks of $\pi$, such that $\min_{\prec_\chi}(V_k) \sim_\pi \max_{\prec_\chi}(V_k)$, and let $V'_k \in \pi$ be the block containing them.
By Lemma \ref{lemproductofdisjointblocksforE}, we may reduce to the case where $m = 1$.

Writing $V_1' = \set{1 = q_1 < q_2 < \cdots < q_{p+1} = n}$, for some $b_j \in B$ depending only on $(T_1, \ldots, T_n)|_{(V'_1)^c}$ and on $\pi$,
\[
E_\pi(T_1, \ldots, T_n) = E\left( T_{q_1} L_{b_1} T_{q_2} L_{b_2} \cdots T_{q_p} L_{b_p} T_{q_{p+1}}\right).
\]
Hence, by the commutation of right $B$-operators with elements of $\mathcal{A}_\ell$, we obtain
\begin{align*}
\mathcal{E}_\pi(T_1, \ldots, T_n)b & = E\left( T_{q_1} L_{b_1} T_{q_2} L_{b_2} \cdots T_{q_p} L_{b_p} T_{q_{p+1}}\right)b\\
& = E\left( R_b T_{q_1} L_{b_1} T_{q_2} L_{b_2} \cdots T_{q_p} L_{b_p} T_{q_{p+1}}\right)\\
& = E\left( T_{q_1} L_{b_1} T_{q_2} L_{b_2} \cdots T_{q_p} L_{b_p} T_{q_{p+1}}R_b \right)\\
& = E\left( T_{q_1} L_{b_1} T_{q_2} L_{b_2} \cdots T_{q_p} L_{b_p} T_{q_{p+1}}L_b \right)\\
& = \mathcal{E}_\pi(T_1, \ldots, T_n L_b).\qedhere
\end{align*}
\end{proof}

\begin{lem}
\label{lemenhancedpropertyiiforE}
The operator-valued bi-free moment function $\mathcal{E}$ satisfies the $q = -\infty$ case of Property (ii) of Definition \ref{defnbimultiplicative} when $1_\chi$ is replaced with an arbitrary $\pi \in BNC(\chi)$.
\end{lem}

\begin{proof}
We will assume $\chi(p) = \ell$ as the case where $\chi(p) = r$ will follow \emph{mutatis mutandis}.
Let $V_1, \ldots, V_m$ be a partition of $\set{1, \ldots, n}$ into $\chi$-intervals ordered by $\prec_\chi$ with each $V_k$ a union of blocks of $\pi$, such that $\min_{\prec_\chi}(V_k) \sim_\pi \max_{\prec_\chi}(V_k)$, and let $V'_k \in \pi$ be the block containing them.
By definitions, notice $p \in V'_1$.  Thus Lemma \ref{lemproductofdisjointblocksforE} implies we may reduce to the case where $m = 1$.

Writing $V_1' = \set{q_1 < q_2 < \cdots < q_{k}}$, for some $b_j \in B$ depending only on $(T_1, \ldots, T_n)|_{(V'_1)^c}$ and on $\pi$, and for some $S \in \A$,
\[
\mathcal{E}_\pi(T_1, \ldots, T_n) = E\left( T_{q_1} R_{b_1}  \cdots T_{q_z} R_{b_z} T_p S\right)
\]
for some $z < k$.  
Hence, by the commutation of left $B$-operators with elements of $\mathcal{A}_r$, we obtain
\begin{align*}
b\mathcal{E}_\pi(T_1, \ldots, T_n) & =bE\left( T_{q_1} R_{b_1}  \cdots T_{q_z} R_{b'_z} T_p S\right)\\
& =E\left(L_bT_{q_1} R_{b_1}  \cdots T_{q_z} R_{b_z} T_p S\right)\\
& = E\left( T_{q_1} R_{b_1}  \cdots T_{q_z} R_{b_z} L_bT_p S\right)\\
& = \mathcal{E}_\pi(T_1, \ldots, T_{p-1}, L_b T_p, T_{p+1}, \ldots, T_n).  \qedhere
\end{align*}  
\end{proof}
\begin{lem}
\label{lemfullreductionofproductsinvolvingblocksinsideblocks}
The operator-valued bi-free moment function $\mathcal{E}$ satisfies Property (iv) of Definition \ref{defnbimultiplicative}.
\end{lem}
\begin{proof}
Again, only the proof of the first case where $\chi(\theta) = \ell$ will be presented.
We proceed by induction on $m$, the number of blocks $U\in\pi$ with
\[
U\subseteq W, \qquad
\min_{\prec_\chi}(U) \prec_{\chi} \min_{\prec_\chi}(V), \qquad
\mbox{and} \qquad
\max_{\prec_\chi}(V)  \prec_{\chi} \max_{\prec_\chi}(U).
\]
However, we first must deal with the case $m = 0$.
Let
\begin{align*}
W_1 = \set{k \in \{1,\ldots, n\} \, \mid \, k \preceq_\chi \theta} \qquad
\mbox{and} \qquad
W_2 = \set{k \in \{1,\ldots, n\} \, \mid \, \gamma \preceq_\chi k}.
\end{align*}
By the assumptions on $W$, both $W_1$ and $W_2$ are $\chi$-intervals that are unions of blocks of $\pi$ such that $W = W_1 \sqcup W_2$, and $W_1 \subseteq \chi^{-1}(\ell)$.
Therefore by Lemmata \ref{lemproductofdisjointblocksforE} and \ref{lemenhancedpropertyiforE},
\begin{align*}
\mathcal{E}_\pi(T_1,\ldots, T_n) &=\mathcal{E}_{\pi|_{W_1}}((T_1,\ldots, T_n)|_{W_1})\mathcal{E}_{\pi|_{V}}((T_1,\ldots, T_n)|_{V})\mathcal{E}_{\pi|_{W_2}}((T_1,\ldots, T_n)|_{W_2}) \\
& = \mathcal{E}_{\pi|_{W_1}}((T_1,\ldots, T_{\theta - 1}, T_\theta L_{\mathcal{E}_{\pi|_{V}}((T_1,\ldots, T_n)|_{V})}, T_{\theta + 1}, \ldots, T_n)|_{W_1})\mathcal{E}_{\pi|_{W_2}}((T_1,\ldots, T_n)|_{W_2}) \\
& =\mathcal{E}_{\pi|_{W}}((T_1,\ldots, T_{\theta - 1}, T_\theta L_{\mathcal{E}_{\pi|_{V}}((T_1,\ldots, T_n)|_{V})}, T_{\theta + 1}, \ldots, T_n)|_{W}).
\end{align*}
Note that we would invoke Lemma \ref{lemenhancedpropertyiiforE} instead of \ref{lemenhancedpropertyiforE} in the case $\chi(\theta) = r$.

For the base case of the inductive argument, suppose $m = 1$.
Let
\begin{align*}
\alpha_1 &= \min_{\prec_\chi}(W_0),  & & \alpha_2 = \max_{\prec_\chi}(\{ k \in W_0 \, \mid \, k \preceq_{\chi} \theta   \}), \\
\beta_1& = \max_{\prec_\chi}(W_0), & \mbox{and} \qquad \qquad \qquad &  \beta_2 = \min_{\prec_\chi}(\{ k \in W_0 \, \mid \, \gamma \preceq_\chi k   \}).
\end{align*}
Furthermore, let
\begin{align*}
W'_1 &= \{k \in \{1,\ldots, n\} \, \mid \, k \prec_\chi \alpha_1 \}, & &
W'_2 = \{k \in \{1,\ldots, n\} \, \mid \,  \beta_1 \prec_\chi k\}, \\
W''_1 &= \{k \in \{1,\ldots, n\} \, \mid \, \alpha_2 \prec k \preceq_\chi \theta\},  & \mbox{and} \qquad \quad &
W''_2 = \{k \in \{1,\ldots, n\} \, \mid \, \gamma \preceq_\chi k \prec_\chi \beta_2\}.
\end{align*}
Therefore, if
\begin{align*}
X'_1 &= \mathcal{E}_{\pi|_{W'_1}}((T_1,\ldots, T_n)|_{W'_1}), & & X'_2 = \mathcal{E}_{\pi|_{W'_2}}((T_1,\ldots, T_n)|_{W'_2}), \\
X''_1 &= \mathcal{E}_{\pi|_{W''_1}}((T_1,\ldots, T_n)|_{W''_1}), & \mbox{and} \qquad \qquad \qquad & X''_2 = \mathcal{E}_{\pi|_{W''_2}}((T_1,\ldots, T_n)|_{W''_2}),
\end{align*} 
then by Lemmata \ref{lemproductofdisjointblocksforE} and \ref{lemreductionofbimultiplicativeinsideablockforE},
\begin{align*}
\mathcal{E}_\pi(T_1,\ldots, T_n)
&= X'_1 \mathcal{E}_{\pi|_{W_0 \cup W''_1 \cup W''_2 \cup V}}((T_1,\ldots, T_n)|_{W_0 \cup W''_1 \cup W''_2 \cup V})X'_2 \\
&= X'_1 \mathcal{E}_{\pi|_{W_0 }}\left(\left(T_1,\ldots, T_{\alpha_2 - 1}, T_{\alpha_2}L_{\mathcal{E}_{\pi|_{W''_1 \cup W''_2 \cup V}}((T_1, \ldots, T_n)|_ {W''_1 \cup W''_2 \cup V})}, T_{\alpha_2+1},\ldots, T_n\right)|_{W_0}\right)X'_2 \\
&= X'_1\mathcal{E}_{\pi|_{W_0 }}\left(\left(T_1,\ldots, T_{\alpha_2 - 1}, T_{\alpha_2} L_{X''_1 \mathcal{E}_{\pi|_{V}}((T_1, \ldots, T_n)|_ {V})   X''_2}, T_{\alpha_2+1},\ldots, T_n\right)|_{W_0}\right)X'_2.
\end{align*}

If $W''_1$ is empty, then $\alpha_2 = \theta$ and 
\[
T_{\alpha_2} L_{X''_1  \mathcal{E}_{\pi|_{V}}((T_1, \ldots, T_n)|_ {V})   X''_2} = T_{\theta} L_{\mathcal{E}_{\pi|_{V}}((T_1, \ldots, T_n)|_ {V})} L_{X''_2}.
\]
On the other hand, if $W''_1$ is non-empty, then Lemma \ref{lemenhancedpropertyiforE} implies that
\[
X''_1\mathcal{E}_{\pi|_{V}}((T_1, \ldots, T_n)|_ {V}) = \mathcal{E}_{\pi|_{W''_1}}((T_1, \ldots, T_{\theta - 1}, T_\theta L_{\mathcal{E}_{\pi|_{V}}((T_1, \ldots, T_n)|_ {V})}, T_{\theta + 1}, \ldots, T_n)|_ {W''_1}).
\] 
The result follows now from Lemmata \ref{lemproductofdisjointblocksforE} and \ref{lemreductionofbimultiplicativeinsideablockforE} in the direction opposite the above.

Inductively, suppose that the result holds when $m \geq 1$.
Suppose $W$ contains blocks $W_0, \ldots, W_m$ of $\pi$ which satisfy the above inequalities.
Without loss of generality, we may assume that 
\[
\min_{\prec_\chi}(W_0) \prec_\chi \cdots \prec_\chi \min_{\prec_\chi}(W_m).
\]
Note that as $\pi$ is bi-non-crossing, this implies
\[
\max_{\prec_\chi}(W_m) \prec_\chi \cdots \prec_\chi \max_{\prec_\chi}(W_0).
\]
Let $\alpha_1, \alpha_2, \beta_1, \beta_2$, $W'_1, W'_2, X'_1,$ and $X'_2$ be as above.
Hence applying Lemmata \ref{lemproductofdisjointblocksforE} and \ref{lemreductionofbimultiplicativeinsideablockforE} once again gives
\begin{align*}
\mathcal{E}_\pi&(T_1,\ldots, T_n) \\
&= X'_1\mathcal{E}_{\pi|_{(W'_1 \cup W'_2)^c}}((T_1,\ldots, T_n)|_{(W'_1 \cup W'_2)^c})X'_2\\
&= X'_1\mathcal{E}_{\pi|_{W_0 }}\left(\left(T_1,\ldots, T_{\alpha_2 - 1}, T_{\alpha_2}L_{\mathcal{E}_{\pi|_{(W_0 \cup W'_1 \cup W'_2)^c}}((T_1, \ldots, T_n)|_ {(W_0 \cup W'_1 \cup W'_2)^c})}, T_{\alpha_2+1},\ldots, T_n\right)|_{W_0}\right)X'_2.
\end{align*}
Now, by the inductive hypothesis, we see that
\begin{align*}
\mathcal{E}_{\pi|_{(W_0 \cup W'_1 \cup W'_2)^c}}&((T_1, \ldots, T_n)|_ {(W_0 \cup W'_1 \cup W'_2)^c}) \\
&  = \mathcal{E}_{\pi|_{(W_0 \cup W'_1 \cup W'_2)^c \setminus V}}((T_1, \ldots, T_{\theta - 1}, T_\theta L_{\mathcal{E}_{\pi|_V}((T_1,\ldots, T_n)|_{V})}, T_{\theta + 1}, \ldots, T_n)|_ {(W_0 \cup W'_1 \cup W'_2)^c\setminus V}).
\end{align*}
Hence, by substituting this expression into the above computation and applying Lemmata \ref{lemproductofdisjointblocksforE} and \ref{lemreductionofbimultiplicativeinsideablockforE} in the opposite order, the inductive step is complete.
\end{proof}

With this result, the proof of Theorem \ref{thm:samantha} is complete.


\section{Operator-Valued Bi-Free Cumulants}
\label{sec:operatorvaluedbifreecumulants}

\subsection{Cumulants via Convolution}
Following \cite{S1998}*{Definition 2.1.6}, we begin with a definition of operator-valued convolution.

\begin{defn}
Let $(\mathcal{A}, E, \varepsilon)$ be a $B$-$B$-non-commutative probability space, let
\[
\Phi : \bigcup_{n\geq 1} \bigcup_{\chi : \{1,\ldots, n\} \to \{\ell, r\}} BNC(\chi) \times \mathcal{A}_{\chi(1)} \times \cdots \times \mathcal{A}_{\chi(n)} \to B,
\]
and let $f \in IA(BNC)$.  We define \emph{the convolution of $\Phi$ and $f$}, denoted $\Phi \ast f$, by
\[
(\Phi \ast f)_\pi(T_1, \ldots, T_n) := \sum_{\substack{\sigma \in BNC(\chi)\\ \sigma \leq \pi}} \Phi_\sigma(T_1,\ldots, T_n) f(\sigma, \pi)
\]
for all $\chi : \{1,\ldots, n\} \to \{\ell, r\}$, $\pi \in BNC(\chi)$, and $T_k \in \mathcal{A}_{\chi(k)}$.
\end{defn}
\begin{rem}
One can check that if $\Phi$ is as above and $f,g \in IA(BNC)$ then $(\Phi \ast f) \ast g = \Phi \ast (f \ast g)$.
\end{rem}
\begin{defn}
\label{def:kappa}
Let $(\mathcal{A}, E, \varepsilon)$ be a $B$-$B$-non-commutative probability space and let $\mathcal{E}$ be the operator-valued bi-free moment function on $\A$.  The \emph{operator-valued bi-free cumulant function} is the function
\[
\kappa : \bigcup_{n\geq 1} \bigcup_{\chi : \{1,\ldots, n\} \to \{\ell, r\}} BNC(\chi) \times \mathcal{A}_{\chi(1)} \times \cdots \times \mathcal{A}_{\chi(n)} \to B
\]
defined by
\[
\kappa := \mathcal{E} \ast \mu_{BNC}.
\]
\end{defn}

Note for $\chi : \{1,\ldots, n\} \to \{\ell, r\}$, $\pi \in BNC(\chi)$, and $T_k \in \mathcal{A}_{\chi(k)}$ that
\[
\kappa_\pi(T_1,\ldots, T_n) = \sum_{\sigma \leq \pi} \mathcal{E}_\sigma(T_1, \ldots, T_n) \mu_{BNC}(\sigma, \pi)
\qquad \mbox{ and }\qquad
\mathcal{E}_\pi(T_1, \ldots, T_n) = \sum_{\sigma \leq \pi} \kappa_\sigma(T_1, \ldots, T_n).
\]

\subsection{Convolution Preserves Bi-Multiplicativity}
It is now straightforward to demonstrate the operator-valued bi-free cumulant function is bi-multiplicative.

\begin{thm}
\label{thmbimulticonvoledwithmultiisbimult}
Let $(\mathcal{A}, E, \varepsilon)$ be a $B$-$B$-non-commutative probability space, let
\[
\Phi : \bigcup_{n\geq 1} \bigcup_{\chi : \{1,\ldots, n\} \to \{\ell, r\}} BNC(\chi) \times \mathcal{A}_{\chi(1)} \times \cdots \times \mathcal{A}_{\chi(n)} \to B, 
\]
and let $f \in IA(BNC)$.  If $\Phi$ is bi-multiplicative and $f$ is multiplicative, then $\Phi \ast f$ is bi-multiplicative.
\end{thm}

\begin{cor}
\label{cor:cumulantsarebimultiplicative}
The operator-valued bi-free cumulant function is bi-multiplicative.
\end{cor}

\begin{proof}[Proof of Theorem \ref{thmbimulticonvoledwithmultiisbimult}]
Clearly $(\Phi \ast f)_\pi$ is linear in each entry.
Furthermore, Proposition \ref{propenhancedproperties} establishes that $\Phi \ast f$ satisfies Properties (i) and (ii) of Definition \ref{defnbimultiplicative}.
Thus it remains to verify Properties (iii) and (iv).

Suppose the hypotheses of Property (iii).
We see that
\begin{align*}
(\Phi \ast f)_\pi&(T_1,\ldots, T_n) \\
&= \sum_{\sigma \leq \pi} \Phi_\sigma(T_1, \ldots, T_n) f(\sigma, \pi)      \\
 &= \sum_{\sigma \leq \pi} \Phi_{\sigma|_{V_1}}((T_1,\ldots, T_n)|_{V_1})   \cdots \Phi_{\sigma|_{V_m}}((T_1,\ldots, T_n)|_{V_m})  f(\sigma|_{V_1}, \pi|_{V_1}) \cdots f(\sigma|_{V_m}, \pi|_{V_m})  \\
  &= (\Phi \ast f)_{\pi|_{V_1}}((T_1,\ldots, T_n)|_{V_1})  \cdots (\Phi \ast f)_{\pi|_{V_m}}((T_1,\ldots, T_n)|_{V_m}),
\end{align*}
using the bi-multiplicativity of $\Phi$ and the multiplicativity of $f$.

To see Property (iv) holds, note under the hypotheses of its initial case,
\begin{align*}
(\Phi \ast f)_\pi&(T_1,\ldots, T_n) \\
 &= \sum_{\sigma \leq \pi} \Phi_\sigma(T_1, \ldots, T_n) f(\sigma, \pi)      \\
 &= \sum_{\sigma \leq \pi} \Phi_{\sigma|_{W}}((T_1,\ldots, T_{\theta-1}, T_\theta L_{\Phi_{\sigma|_{V}}((T_1, \ldots, T_n)|_V)}, T_{\theta + 1}, \ldots, T_n)|_{W})  f(\sigma|_{V}, \pi|_{V}) f(\sigma|_{W}, \pi|_{W})  \\
  &= \sum_{\sigma \leq \pi} \Phi_{\sigma|_{W}}((T_1,\ldots, T_{\theta-1}, T_\theta L_{\Phi_{\sigma|_{V}}((T_1, \ldots, T_n)|_V)f(\sigma|_{V}, \pi|_{V})}, T_{\theta + 1}, \ldots, T_n)|_{W})   f(\sigma|_{W}, \pi|_{W})  \\
  &= (\Phi \ast f)_{\pi|_{W}}((T_1,\ldots, T_{\theta-1}, T_\theta L_{\Phi_{\sigma|_{V}}((T_1, \ldots, T_n)|_V)}, T_{\theta + 1}, \ldots, T_n)|_{W}),
\end{align*}
again by the corresponding properties of $\Phi$ and $f$.
The proof of the remaining three statements in Property (iv) is identical.
\end{proof}

\subsection{Bi-Moment and Bi-Cumulant Functions}
\label{subsec:bimomentandbicumulantfunctions}

Inspired by \cite{S1998}*{Section 3.2}, we define the formal classes of bi-moment and bi-cumulant functions and give an important relation between them.
It follows readily that the operator-valued bi-free moment and cumulant functions on a $B$-$B$-non-commutative probability space are examples of these types of functions, respectively.

We begin with the following useful notation.
\begin{nota}
Let $\chi : \{1, \ldots, n\} \to \{\ell, r\}$, $\pi \in BNC(\chi)$, and $q \in \{1,\ldots, n\}$.
We denote by $\chi|_{\setminus q}$ the restriction of $\chi$ to the set $\{1,\ldots, n\} \setminus \{q\}$.
In addition, if $q \neq n$, we define $\pi|_{q = q+1} \in BNC(\chi|_{\setminus q})$ to be the bi-non-crossing partition which results from identifying $q$ and $q+1$ in $\pi$ (i.e. if $q$ and $q+1$ are in the same block as $\pi$ then $\pi|_{q=q+1}$ is obtained from $\pi$ by just removing $q$ from the block in which $q$ occurs, while if $q$ and $q+1$ are in different blocks, $\pi|_{q=q+1}$ is obtained from $\pi$ by merging the two blocks and then removing $q$).
\end{nota}

\begin{defn}
\label{defnbimomentandbicumulantfunctions}
Let $(\mathcal{A}, E, \varepsilon)$ be a $B$-$B$-non-commutative probability space and let
\[
\Phi : \bigcup_{n\geq 1} \bigcup_{\chi : \{1,\ldots, n\} \to \{\ell, r\}} BNC(\chi) \times \mathcal{A}_{\chi(1)} \times \cdots \times \mathcal{A}_{\chi(n)} \to B
\]
be bi-multiplicative.  We say that $\Phi$ is a \emph{bi-moment function} if whenever $\chi : \{1,\ldots, n\} \to \{\ell, r\}$ is such that there exists a $q \in \{1,\ldots, n-1\}$ with $\chi(q) = \chi(q+1)$, then
\[
\Phi_{1_\chi}(T_1, \ldots, T_n) = \Phi_{1_{\chi|_{\setminus q}}}(T_1, \ldots, T_{q-1}, T_qT_{q+1}, T_{q+2},\ldots, T_n)
\]
for all $T_k \in \mathcal{A}_{\chi(k)}$.  Similarly, we say that $\Phi$ is a \emph{bi-cumulant function} if whenever $\chi : \{1,\ldots, n\} \to \{\ell, r\}$ and $\pi \in BNC(\chi)$ are such that there exists a $q \in \{1,\ldots, n-1\}$ with $\chi(q) = \chi(q+1)$, then
\[
\Phi_{1_{\chi|_{\setminus q}}}(T_1, \ldots, T_{q-1}, T_qT_{q+1}, T_{q+2},\ldots, T_n) = \Phi_{1_\chi}(T_1, \ldots, T_n) + \sum_{\substack{\pi \in BNC(\chi)\\ |\pi| =2, q\nsim_\pi q+1}} \Phi_\pi(T_1, \ldots, T_n)
\]
for all $T_k \in \mathcal{A}_{\chi(k)}$.
\end{defn}

\begin{rem}
The operator-valued bi-free moment function $\mathcal{E}$ is a bi-moment function.
\end{rem}

Before relating the notions of bi-moment and bi-cumulant functions, we note the following alternate formulations.
\begin{lem}
\label{lemequivalentnotionsofbimomentandbicumulant}
Let $(\mathcal{A}, E, \varepsilon)$ be a $B$-$B$-non-commutative probability space and let
\[
\Phi : \bigcup_{n\geq 1} \bigcup_{\chi : \{1,\ldots, n\} \to \{\ell, r\}} BNC(\chi) \times \mathcal{A}_{\chi(1)} \times \cdots \times \mathcal{A}_{\chi(n)} \to B
\]
be bi-multiplicative.  Then $\Phi$ is a bi-moment function if and only if whenever $\chi : \{1,\ldots, n\} \to \{\ell, r\}$ and $\pi \in BNC(\chi)$ are such that there exists a $q \in \{1,\ldots, n-1\}$ with $\chi(q) = \chi(q+1)$ and $q \sim_\pi q+1$, then
\[
\Phi_\pi(T_1, \ldots, T_n) = \Phi_{\pi|_{q = q+1}}(T_1, \ldots, T_{q-1}, T_qT_{q+1}, T_{q+2}, \ldots T_n)
\]
for all $T_k \in \mathcal{A}_{\chi(k)}$.  Similarly, $\Phi$ is a bi-cumulant function if and only if whenever $\chi : \{1,\ldots, n\} \to \{\ell, r\}$ is such that there exists a $q \in \{1,\ldots, n-1\}$ with $\chi(q) = \chi(q+1)$, we have
\[
\Phi_{\pi}(T_1, \ldots, T_{q-1}, T_qT_{q+1}, T_{q+2}, \ldots T_n) = \sum_{\substack{\sigma \in BNC(\chi) \\ \sigma|_{q = q+1} = \pi}} \Phi_\sigma(T_1, \ldots, T_n)
\]
for all $\pi \in BNC(\chi|_{\setminus q})$.
\end{lem}

To establish the lemma, one uses bi-multiplicativity to reduce to the case of full partitions and then applies Definition \ref{defnbimomentandbicumulantfunctions}.

\begin{thm}
\label{thmrelationbetweenbimomentandbicumulantfunctions}
Let $(\mathcal{A}, E, \varepsilon)$ be a $B$-$B$-non-commutative probability space and let
\[
\Phi , \Psi: \bigcup_{n\geq 1} \bigcup_{\chi : \{1,\ldots, n\} \to \{\ell, r\}} BNC(\chi) \times \mathcal{A}_{\chi(1)} \times \cdots \times \mathcal{A}_{\chi(n)} \to B
\]
be related by the formulae
\[
\Phi = \Psi \ast \zeta_{BNC}, \qquad \mbox{or equivalently} \qquad \Psi = \Phi \ast \mu_{BNC}.
\]
Then $\Phi$ is a bi-moment function if and only if $\Psi$ is a bi-cumulant function. 
\end{thm}
\begin{proof}
To begin, note $\Phi$ is bi-multiplicative if and only if $\Psi$ is bi-multiplicative by Theorem \ref{thmbimulticonvoledwithmultiisbimult}.

Suppose $\Psi$ is a bi-cumulant function.  If $\chi : \{1,\ldots, n\} \to \{\ell, r\}$ is such that there exists a $q \in \{1,\ldots, n-1\}$ with $\chi(q) = \chi(q+1)$, then for all $T_k \in \mathcal{A}_{\chi(k)}$
\begin{align*}
\Phi_{1_{\chi|_{\setminus q}}}(T_1,\ldots, T_{q-1}, T_q T_{q+1}, T_{q+2}, \ldots, T_n) &= \sum_{\pi \in BNC(\chi|_{\setminus q})} \Psi_{\pi}(T_1,\ldots, T_{q-1}, T_q T_{q+1}, T_{q+2}, \ldots, T_n)  \\
 &=  \sum_{\pi \in BNC(\chi|_{\setminus q})} \sum_{\substack{\sigma \in BNC(\chi) \\ \sigma|_{q = q+1} = \pi}} \Psi_{\sigma}(T_1,\ldots, T_n) \\
 &=  \sum_{\sigma \in BNC(\chi)} \Psi_{\sigma}(T_1,\ldots, T_n) \\
 &=  \Phi_{1_\chi}(T_1,\ldots, T_n).
\end{align*}
Thus $\Phi$ is a bi-moment function.

For the other direction, suppose $\Phi$ is a bi-moment function.
Let $\chi : \{1, \ldots, n\} \to \{\ell, r\}$.  We will proceed by induction on $n$.
If $n = 1$, there is nothing to check.
If $n = 2$, then
\[
\Psi_{1_{\chi|_{1=2}}}(T_1T_2) = \Phi_{1_{\chi|_{1=2}}}(T_1T_2) = \Phi_{1_\chi}(T_1, T_2) = \Psi_{1_\chi}(T_1, T_2) + \Psi_{0_\chi}(T_1, T_2)
\]
as required.
\par 
Suppose that the formula from Definition \ref{defnbimomentandbicumulantfunctions} holds for $n-1$.
Then using the induction hypothesis and bi-multiplicativity of $\Psi$, we see for all $\pi \in BNC(\chi|_{\setminus q}) \setminus \{1_{\chi|_{\setminus q}}\}$ that
\[
\Psi_{\pi}(T_1, \ldots, T_{q-1}, T_qT_{q+1}, T_{q+2}, \ldots, T_n) = \sum_{\substack{\sigma \in BNC(\chi) \\ \sigma|_{q = q+1} = \pi}} \Psi_\sigma(T_1, \ldots, T_n).
\]
Hence
\begin{align*}
\Psi_{1_{\chi|_{\setminus q}}}&(T_1, \ldots, T_{q-1}, T_qT_{q+1}, T_{q+2}, \ldots, T_n) \\
&= \Phi_{1_{\chi|_{\setminus q}}}(T_1, \ldots, T_{q-1}, T_qT_{q+1}, T_{q+2}, \ldots, T_n) - \sum_{\substack{\pi \in BNC(\chi|_{\setminus q}) \\  \pi \neq 1_{\chi|_{\setminus q}}}}   \Psi_{\pi}(T_1, \ldots, T_{q-1}, T_qT_{q+1}, T_{q+2}, \ldots T_n) \\
 &= \Phi_{1_{\chi}}(T_1, \ldots, T_n) - \sum_{\substack{\pi \in BNC(\chi|_{\setminus q}) \\  \pi \neq 1_{\chi|_{\setminus q}}}} \sum_{\substack{\sigma \in BNC(\chi) \\ \sigma|_{q = q+1} = \pi}} \Psi_\sigma(T_1, \ldots, T_n)  \\
  &= \sum_{\sigma\in BNC(\chi)}\Psi_{\sigma}(T_1, \ldots, T_n) - \sum_{\substack{\sigma \in BNC(\chi) \\ \sigma|_{q = q+1} \neq 1_{\chi|_{\setminus q}}}} \Psi_\sigma(T_1, \ldots, T_n)  \\
 &= \sum_{\substack{\sigma \in BNC(\chi) \\ \sigma|_{q = q+1} = 1_{\chi|_{\setminus q}}}} \Psi_\sigma(T_1, \ldots, T_n). \qedhere
\end{align*}
\end{proof}
\begin{cor}
\label{corcumulantfunctionisabicumulantfunction}
The operator-valued bi-free cumulant function $\kappa$ is a bi-cumulant function.
\end{cor}

\subsection{Vanishing of Operator-Valued Bi-Free Cumulants}

The following demonstrates, like with classical and free cumulants, that operator-valued bi-free cumulants of order at least two vanish provided at least one entry is an element of $B$.
\begin{prop}
\label{prop:inputaL_borR_bandyougetzerocumulants}
Let $(\mathcal{A}, E, \varepsilon)$ be a $B$-$B$-non-commutative probability space, $\chi : \{1,\ldots, n\} \to \{\ell, r\}$ with $n \geq 2$, and $T_k \in \mathcal{A}_{\chi(k)}$.
If there exist $q \in \{1,\ldots, n\}$ and $b \in B$ such that $T_q = L_b$ if $\chi(q) = \ell$ or $T_q = R_b$ if $\chi(q) = r$, then
\[
\kappa(T_1, \ldots, T_n) = 0.
\]
\end{prop}
\begin{proof}
The base case can be readily established by computing directly the cumulants of order two.

For the inductive step, suppose the result holds for all $\chi : \{1, \ldots, k\} \to \{\ell, r\}$ with $k < n$.  Fix $\chi : \{1,\ldots, n\} \to \{\ell, r\}$ and $T_k \in \mathcal{A}_{\chi(k)}$. 
Suppose that for some $q \in \set{1, \ldots, n}$ we have $\chi(q) = \ell$ and $T_q = L_b$ with $b \in B$, as the argument for the right side is similar.

Let
\[
p = \max\set{k \in \{1,\ldots, n\} \, \mid \, \chi(k) = \ell, k < q}.
\]
The proof is now divided into two cases.

\underline{\textit{Case 1: $p \neq -\infty$.}}
In this case we notice that
\begin{align*}
\kappa_{1_\chi}(T_1, \ldots, T_n)
&= \mathcal{E}_{1_\chi}(T_1, \ldots, T_n) - \sum_{\substack{\pi \in BNC(\chi)\\ \pi \neq 1_\chi}}\kappa_{\pi}(T_1, \ldots, T_n)    \\
&= \mathcal{E}_{1_\chi}(T_1, \ldots, T_n) - \sum_{\substack{\pi \in BNC(\chi)\\ \set{q} \in \pi}}\kappa_{\pi}(T_1, \ldots, T_{p-1}, T_p, T_{p+1}, \ldots, T_{q-1}, L_b, T_{q+1}, \ldots, T_n)    \\
&= \mathcal{E}_{1_\chi}(T_1, \ldots, T_n) - \sum_{\sigma \in BNC(\chi|_{\setminus q})}\kappa_{\sigma}(T_1, \ldots, T_{p-1}, T_p L_b, T_{p+1}, \ldots, T_{q-1}, T_{q+1}, \ldots, T_n),
\end{align*}
by induction and Proposition \ref{propenhancedproperties}.  Since
\begin{align*}
\mathcal{E}_{1_\chi}(T_1, \ldots, T_n) &= E(T_1\cdots T_n)   \\
 &= E(T_1 \cdots T_{q-1} L_b T_{q+1} \cdots T_n) \\
 &=  E(T_1 \cdots T_{p-1} L_b T_{p+1}\cdots T_n) \\
 &=  \sum_{\sigma \in BNC(\chi|_{\setminus q})}\kappa_{\sigma}(T_1, \ldots, T_{p-1}, T_p L_b, T_{p+1}, \ldots, T_{q-1}, T_{q+1}, \ldots, T_n),
\end{align*}
the proof is complete in this case.

\underline{\textit{Case 2: $p = -\infty$.}}  In this case, notice that
\begin{align*}
\kappa_{1_\chi}(T_1, \ldots, T_n)
&= \mathcal{E}_{1_\chi}(T_1, \ldots, T_n) - \sum_{\substack{\pi \in BNC(\chi)\\ \pi \neq 1_\chi}}\kappa_{\pi}(T_1, \ldots, T_n)    \\
 &= \mathcal{E}_{1_\chi}(T_1, \ldots, T_n) - \sum_{\substack{\pi \in BNC(\chi)\\ \set{q} \in \pi}}\kappa_{\pi}(T_1, \ldots, T_{q-1}, L_b, T_{q+1}, \ldots, T_n)    \\
 &= \mathcal{E}_{1_\chi}(T_1, \ldots, T_n) - \sum_{\sigma \in BNC(\chi|_{\setminus q})} b\kappa_{\sigma}(T_1, \ldots, T_{q-1}, T_{q+1}, \ldots, T_n),
\end{align*}
by induction and Proposition \ref{propenhancedproperties}.  Since
\begin{align*}
\mathcal{E}_{1_\chi}(T_1, \ldots, T_n) &= E(T_1\cdots T_n)   \\
 &= E(T_1 \cdots T_{q-1} L_b T_{q+1} \cdots T_n) \\
 &=   E(L_bT_1 \cdots T_{q-1} T_{q+1}\cdots T_n) \\
  &=   bE(T_1 \cdots T_{q-1} T_{q+1}\cdots T_n) \\
 &=  \sum_{\sigma \in BNC(\chi|_{\setminus q})} b\kappa_{\sigma}(T_1, \ldots, T_{q-1}, T_{q+1}, \ldots, T_n),
\end{align*}
the proof is complete in this case as well.
\end{proof}

\section{Universal Moment Polynomials for Bi-Free Families with Amalgamation}
\label{sec:universalmomentpolysforbifreewithamalgamation}

In this section, we will generalize \cite{CNS2014}*{Corollary 4.1.2} to demonstrate that algebras will be bi-free with amalgamation over $B$ if and only if certain moment expressions hold.  To do so, we will need to extend the definition of $E_\pi(T_1,\ldots, T_n)$ to an extension of the shaded $LR$ diagrams as defined in Section \ref{sec:BiNonCrossingPartitionsAndDiagrams}. 


\subsection{Equivalence of Bi-Free with Amalgamation and Universal Moment Polynomials}

\begin{defn}
Let $\chi : \{1,\ldots, n\} \to \{\ell, r\}$ and let $\epsilon : \{1,\ldots, n\} \to K$.
Let $LR^\lat_k(\chi, \epsilon)$ denote the closure of $LR_k(\chi, \epsilon)$ under lateral refinement.
Observe that every diagram in $LR^\lat_k(\chi, \epsilon)$ still has $k$ strings reaching its top, as lateral refinements may only introduce cuts between ribs.
Given $D_1, D_2 \in LR^\lat_k(\chi, \epsilon)$, we write $D_2 \lrleq D_1$ if $D_2$ can be obtained by laterally refining $D_1$.
Finally, we denote
\[
LR^\lat(\chi, \epsilon) := \bigcup_{k\geq 0} LR^\lat_k(\chi, \epsilon).
\]
\end{defn}

\begin{defn}
\label{defn:ELRrecursiveproof}
Let $\{(\X_k, \mathring{\X}_k, p_k)\}_{k \in K}$ be $B$-$B$-bimodules with specified $B$-vector states, let $\lambda_k$ and $\rho_k$ be as defined in Construction \ref{cons:freeproductconstruction}, and let $\X = (\ast_B)_{k \in K} \X_k$.  Let $\chi : \{1,\ldots, n\} \to \{\ell, r\}$, $\epsilon : \{1,\ldots, n\} \to K$, $D \in LR^\lat(\chi, \epsilon)$, and $T_k \in \mathcal{L}_{\chi(k)}(\X_{\epsilon(k)})$.  Define $\mu_k(T_k) = \lambda_{\epsilon(k)}(T_k)$ if $\chi(k) = \ell$ and $\mu_k(T_k) = \rho_{\epsilon(k)}(T_k)$ if $\chi(k) = r$.  We define $E_D(\mu_1(T_1),\ldots, \mu_n(T_n))$ recursively as follows: Apply the same recursive process as in Definition \ref{defn:recursivedefinitionofEpi} until every block of $\pi$ has a spine reaching the top. If every block of $D$ has a spine reaching the top, enumerate them from left to right according to their spines as $V_1,\ldots, V_m$ with $V_j=\{k_{j,1}<\cdots <k_{j,q_j}\}$, and set
\[
	E_D(\mu_1(T_1),\ldots,\mu_n(T_n))=[(1-p_{\epsilon(k_{1,1})})T_{k_{1,1}} \cdots T_{k_{1,q_1}} 1_B] \otimes \cdots \otimes [(1-p_{\epsilon(k_{m,1})})T_{k_{m,1}} \cdots T_{k_{m,q_m}} 1_B].
\]
\end{defn}

\begin{lem}
\label{lem:actingonFreeproductspace}
With the notation as in Definition \ref{defn:ELRrecursiveproof},
\[
\mu_1(T_1) \cdots \mu_n(T_n)1_B = \sum_{k=0}^n \sum_{D \in LR^\lat_k(\chi, \epsilon)} \left[ \sum_{\substack{ D' \in LR_k(\chi, \epsilon) \\ D' \geq_{\mathrm{lat}} D}} (-1)^{|D| - |D'|}   \right] E_{D}(\mu_1(T_1),\ldots, \mu_n(T_n)),
\]
where $|D|$ and $|D'|$ denote the number of blocks of $D$ and $D'$ respectively.  In particular,
\[
E_{\mathcal{L}(\X)}(\mu_1(T_1) \mu_2(T_2) \cdots \mu_n(T_n)) = \sum_{\pi \in BNC(\chi)} \left[ \sum_{\substack{\sigma\in BNC(\chi)\\\pi\leq\sigma\leq\e}}\mu_{BNC}(\pi, \sigma) \right] \mathcal{E}_{\pi}(\mu_1(T_1),\ldots, \mu_n(T_n)).
\]
\end{lem}

\begin{proof}
To begin, note that the second claim follows from the first by Definition \ref{defn:ELRrecursiveproof} and by Theorem \ref{thm:twosums}.  To prove the main claim we will proceed by induction on the number of operators $n$.  The case $n = 1$ is trivial.

For the inductive step, we will assume that $\chi(1) = \ell$ as the proof in the case $\chi(1) = r$ will follow by similar arguments. Let $\chi_0 = \chi|_{\{2,\ldots, n\}}$ and $\epsilon_0 = \epsilon|_{\{2,\ldots, n\}}$.  By induction, we obtain that
\[
\mu_2(T_2) \cdots \mu_n(T_n)1_B = \sum_{k=0}^{n-1} \sum_{D_0 \in LR^\lat_k(\chi_0, \epsilon_0)} \left[ \sum_{\substack{ D'_0 \in LR_k(\chi_0, \epsilon_0) \\ D'_0 \geq_{\mathrm{lat}} D_0}} (-1)^{|D_0| - |D'_0|}   \right] E_{D_0}(\mu_2(T_2),\ldots, \mu_n(T_n)).
\]
The result will follow by applying $\lambda_1(T_1)$ to each $E_{D_0}(\mu_2(T_2),\ldots, \mu_n(T_n))$, checking the correct terms appear, collecting the same terms, and verifying the coefficients are correct.

Fix $D_0 \in LR^\lat_k(\chi_0, \epsilon_0)$.  We can write
\[
E_{D_0}(\mu_2(T_2),\ldots, \mu_n(T_n)) = [(1-p_{\epsilon(k_{1})}) S_1 1_B] \otimes \cdots \otimes [(1-p_{\epsilon(k_{m})}) S_m1_B]
\]
for some operators $S_p \in \mathrm{alg}(\lambda_{k_p}(\mathcal{L}_\ell(\X_{k_p})), \rho_{k_p}(\mathcal{L}_r(\X_{k_p})))$.  To demonstrate the correct terms appear, we divide the analysis into three cases.

\underline{\textit{Case 1: $m = 0$.}}  In this case $E_{D_0}(\mu_2(T_2),\ldots, \mu_n(T_n)) = b \in B$.  As such, we see that
\begin{align*}
\lambda_{\epsilon(1)}(T_1)E_{D_0}(\mu_2(T_2),\ldots, \mu_n(T_n)) & = E(T_1)b \oplus [(1-p_{\epsilon(1)})T_1 b]
\end{align*}
If $D_1$ is the $LR$-diagram obtained from $D_0$ by placing a node shaded $\epsilon(1)$ at the top and $D_2$ is the $LR$-diagram obtained from $D_0$ by placing a node $\epsilon(1)$ at the top and drawing a spine from this node to the top, then since 
\[
E(\mu_1(T_1) L_b) = E(\mu_1(T_1) R_b) = E(R_b \mu_1(T_1)) = E(T_1) b
\]
and
\[
T_1 b = T_1 R_b 1_B = T_1 L_b 1_B
\]
one easily sees that 
\[
E(T_1)b = E_{D_1}(\mu_1(T_1), \mu_2(T_2),\ldots, \mu_n(T_n)) \quad \mbox{and} \quad (1-p_{\epsilon(1)})T_1 b = E_{D_2}(\mu_1(T_1), \mu_2(T_2),\ldots, \mu_n(T_n)).
\]

\underline{\textit{Case 2: $m \neq 0$ and $\epsilon(1) \neq \epsilon(k_1)$.}}  In this case, $(1-p_{\epsilon(k_{1})}) S_1 1_B$ is in a space orthogonal to $\mathring{\X}_{\epsilon(1)}$.  Thus
\begin{align*}
	\lambda_{\epsilon(1)}(T_1)E_{D_0}(\mu_2(T_2),\ldots, \mu_n(T_n))  =& \left([L_{E(T_1)}(1-p_{\epsilon(k_{1})}) S_1 1_B] \otimes \cdots \otimes [(1-p_{\epsilon(k_{m})}) S_m1_B]\right) \\
			&  \oplus \left([(1-p_{\epsilon(1)})T_1 1_B]  \otimes [(1-p_{\epsilon(k_{1})}) S_1 1_B] \otimes \cdots \otimes [(1-p_{\epsilon(k_{m})}) S_m1_B]\right).
\end{align*}
If $D_1$ is the $LR$-diagram obtained from $D_0$ by placing a node shaded $\epsilon(1)$ at the top and $D_2$ is the $LR$-diagram obtained from $D_0$ by placing a node $\epsilon(1)$ at the top and drawing a spine from this node to the top, then since 
\[
L_{E(T_1)}(1-p_{\epsilon(k_{1})}) S_1 1_B = (1-p_{\epsilon(k_{1})}) L_{E(T_1)}S_1 1_B,
\]
one easily sees that 
\begin{align*}
[L_{E(T_1)}(1-p_{\epsilon(k_{1})}) S_1 1_B] \otimes \cdots \otimes [(1-p_{\epsilon(k_{m})}) S_m1_B] = E_{D_1}(\mu_1(T_1), \mu_2(T_2),\ldots, \mu_n(T_n)) \\
[(1-p_{\epsilon(1)})T_1 1_B]  \otimes [(1-p_{\epsilon(k_{1})}) S_1 1_B] \otimes \cdots \otimes [(1-p_{\epsilon(k_{m})}) S_m1_B] = E_{D_2}(\mu_1(T_1), \mu_2(T_2),\ldots, \mu_n(T_n)).
\end{align*}

\underline{\textit{Case 3: $m \neq 0$ and $\epsilon(1) = \epsilon(k_1)$.}}  In this case, there is a spine in $D$ that reaches the top and is the same shading as $T_1$.  Thus $(1-p_{\epsilon(k_{1})}) S_1 1_B \in\mathring{\X}_{\epsilon(1)}$, so
\begin{align*}
	\lambda_{\epsilon(1)}(T_1)E_{D_0}(\mu_2(T_2),\ldots, \mu_n(T_n))  =& \left([L_{p_{\epsilon(1)}\left(T_1(1-p_{\epsilon(1)}) S_1 1_B\right)}(1-p_{\epsilon(k_{2})}) S_2 1_B] \otimes \cdots \otimes [(1-p_{\epsilon(k_{m})}) S_m1_B]\right) \\
		&  \oplus \left([(1-p_{\epsilon(1)})T_1(1-p_{\epsilon(1)}) S_1 1_B] \otimes \cdots \otimes [(1-p_{\epsilon(k_{m})}) S_m1_B]\right).
\end{align*}
Expanding $T_1(1-p_{\epsilon(1)}) S_11_B = T_1S_11_B - T_1E(S_1)$, the above becomes
\begin{align*}
	\lambda_{\epsilon(1)}(T_1)E_{D_0}(\mu_2(T_2),\ldots, \mu_n(T_n))  =& \left([L_{E(T_1S_1)}(1-p_{\epsilon(k_{2})}) S_2 1_B] \otimes \cdots \otimes [(1-p_{\epsilon(k_{m})}) S_m1_B]\right) \\
		& \oplus (-1) \left([L_{p_{\epsilon(1)}\left(T_1E(S_1)\right)}(1-p_{\epsilon(k_{2})}) S_2 1_B] \otimes \cdots \otimes [(1-p_{\epsilon(k_{m})}) S_m1_B]\right) \\
		&  \oplus \left([(1-p_{\epsilon(1)})T_1S_1 1_B] \otimes \cdots \otimes [(1-p_{\epsilon(k_{m})}) S_m1_B]\right)\\
		&  \oplus (-1)\left([(1-p_{\epsilon(1)})T_1E(S_1)] \otimes \cdots \otimes [(1-p_{\epsilon(k_{m})}) S_m1_B]\right).
\end{align*}
Let $D_1$ be the $LR$-diagram obtained from $D_0$ by placing a node shaded $\epsilon(1)$ at the top and terminating the left-most spine at this node, $D_2$ be the $LR$-diagram obtained by laterally refining $D_1$ by cutting the spine attached to the top node directly beneath the top node, $D_3$ be the $LR$-diagram obtained from $D_0$ by placing a node shaded $\epsilon(1)$ at the top and connecting this node to the left-most spine, and $D_4$ be the $LR$-diagram obtained by laterally refining $D_3$ by cutting the spine attached to the top node directly beneath the top node.  As in the previous case, we see (by applying Lemma \ref{lemproductofdisjointblocksforE} if $m = 1$) that
\[
[L_{E(T_1S_1)}(1-p_{\epsilon(k_{2})}) S_2 1_B] \otimes \cdots \otimes [(1-p_{\epsilon(k_{m})}) S_m1_B] = E_{D_1}(\mu_1(T_1), \mu_2(T_2),\ldots, \mu_n(T_n))
\]
and 
\[
[(1-p_{\epsilon(1)})T_1S_1 1_B] \otimes \cdots \otimes [(1-p_{\epsilon(k_{m})}) S_m1_B] = E_{D_3}(\mu_1(T_1), \mu_2(T_2),\ldots, \mu_n(T_n)).
\]

We will demonstrate that
\[
[L_{p_{\epsilon(1)}\left(T_1E(S_1)\right)}(1-p_{\epsilon(k_{2})}) S_2 1_B] \otimes \cdots \otimes [(1-p_{\epsilon(k_{m})}) S_m1_B] = E_{D_2}(\mu_1(T_1), \mu_2(T_2),\ldots, \mu_n(T_n))
\]
and leave to the reader the proof that
\[
[(1-p_{\epsilon(1)})T_1E(S_1)] \otimes \cdots \otimes [(1-p_{\epsilon(k_{m})}) S_m1_B] = E_{D_4}(\mu_1(T_1), \mu_2(T_2),\ldots, \mu_n(T_n)).
\]
Notice that 
\[
L_{p_{\epsilon(1)}\left(T_1E(S_1)\right)} = L_{p_{\epsilon(1)}\left(T_1R_{E(S_1)}1_B\right)} = L_{p_{\epsilon(1)}\left(T_11_B\right) E(S_1)} = L_{p_{\epsilon(1)}\left(T_11_B\right)}L_{E(S_1)},
\]
and so
\[
L_{p_{\epsilon(1)}\left(T_1E(S_1)\right)}(1-p_{\epsilon(k_{2})}) S_2 1_B = (1-p_{\epsilon(k_{2})}) L_{p_{\epsilon(1)}\left(T_11_B\right)}L_{E(S_1)}S_2 1_B.
\]
Thus, unless $m = 1$, $L_{p_{\epsilon(1)}\left(T_11_B\right)}$ appears as it should in the definition of $E_{D_2}(\mu_1(T_1), \mu_2(T_2),\ldots, \mu_n(T_n))$ although the $E(S_1)$ term may not be as it should.  To obtain the desired result, we make the following corrections.

Recall that $S_1$ corresponds to the left-most-spine of $D_0$ reaching the top.
Let $W \subseteq \{2,\ldots, n\}$ be the set of $k$ for which $T_k$ appears in the expression for $S_1$.
Note that $S_1$ will be of the form $C_{E(S'_1)} C_{E(S'_2)} \cdots C_{E(S'_{p-1})} S'_p$, where each $E(S'_k)$ is the moment of a disjoint $\chi$-interval $W_k$ composed of blocks of $D_2$ with the property that $\min_{\prec_\chi}(W_k)$ and $\max_{\prec_\chi}(W_k)$ lie in the same block ($C$ denotes either $L$ or $R$, as appropriate).
Observe that $W = \bigcup^p_{k=1} W_k$.
Therefore, by Proposition \ref{prop:propertiesofEforLX}
\[
E(S_1) = E(C_{E(S'_1)} C_{E(S'_2)} \cdots C_{E(S'_{p-1})} S'_p)= E(S''_1) \cdots E(S''_p),
\]
where the $S''_1,\ldots, S''_p$ are $S'_1,\ldots, S'_p$ up to reordering by $\prec_\chi$. Let $W'_1,\ldots, W'_k$ be $W_1,\ldots, W_k$ under this new ordering. 

Through Lemma \ref{lemproductofdisjointblocksforE} in the case $m=1$ and computations similar to the reverse of those used in cases 1 and 2 of Lemma \ref{lemreductionofbimultiplicativeinsideablockforE} one can verify these terms can be moved into the correct positions.

Finally, it is clear that the coefficients of each $E_D(\mu_1(T_1),\ldots, \mu_n(T_n))$ are correct for each $D \in LR^\lat(\chi, \epsilon)$.
Alternatively, one can check the coefficients in the second claim by noting that the coefficients did not depend on the algebra $B$, setting $B=\mathbb{C}$, and using \cite{CNS2014}*{Corollary 4.2.5}.
\end{proof}
\begin{thm}
\label{thm:bifreeequivalenttouniversalpolys}
Let $(\mathcal{A}, E_\mathcal{A}, \varepsilon)$ be a $B$-$B$-non-commutative probability space and let $\{(C_k, D_k)\}_{k \in K}$ be a family of pair of $B$-faces of $\A$.  Then $\{(C_k, D_k)\}_{k \in K}$ are bi-free with amalgamation over $B$ if and only if for all $\chi : \{1,\ldots, n\} \to \{\ell, r\}$, $\epsilon : \{1,\ldots, n\} \to K$, and 
\[
T_k \in \left\{
\begin{array}{ll}
C_{\epsilon(k)} & \mbox{if } \chi(k) = \ell  \\
D_{\epsilon(k)} & \mbox{if } \chi(k) = r
\end{array} \right.,
\]
the formula
\[
E_{\A}(T_1 \cdots T_n) = \sum_{\pi \in BNC(\chi)} \left[ \sum_{\substack{\sigma\in BNC(\chi)\\\pi\leq\sigma\leq\e}}\mu_{BNC}(\pi, \sigma) \right] \mathcal{E}_{\pi}(T_1,\ldots, T_n)
\]
holds.
\end{thm}
\begin{proof}
If $\{(C_k, D_k)\}_{k \in K}$ are bi-free with amalgamation over $B$, then Lemma \ref{lem:actingonFreeproductspace} implies the desired formula holds.

Conversely, suppose that the formula holds.  By Theorem \ref{thm:representingbbncps} there exists a $B$-$B$-bimodule with a specified $B$-vector state $(\X, \mathring{\X}, p)$ and a unital homomorphism $\theta : \mathcal{A} \to \mathcal{L}(\X)$ such that 
\[
\theta(\varepsilon(b_1 \otimes b_2)) = L_{b_1} R_{b_2}, \quad \theta(\mathcal{A}_\ell) \subseteq \mathcal{L}_\ell(\X),\quad \theta(\mathcal{A}_r) \subseteq \mathcal{L}_r(\X), \quad \mathrm{and}\quad E_{\mathcal{L}(\X)}(\theta(T)) = E_\mathcal{A}(T)
\]
for all $b_1, b_2 \in B$ and $T \in \mathcal{A}$.  For each $k \in K$, let $(\X_k, \mathring{\X}_k, p_k)$ be a copy of $(\X, \mathring{\X}, p)$ and $l_k$ and $r_k$ be copies of $\theta\colon \mathcal{A}\to \mathcal{L}(\X_k)$.  Since the formula holds, Lemma \ref{lem:actingonFreeproductspace} implies $\{(C_k,D_k)\}_{k\in K}$ are bi-free over $B$.
\end{proof}


\section{Additivity of Operator-Valued Bi-Free Cumulants}
\label{sec:additivity}

\subsection{Equivalence of Bi-Freeness with Vanishing of Mixed Cumulants}
We now state the operator-valued analogue of the main result of \cite{CNS2014}, namely, that bi-freeness of families of pairs of $B$-faces is equivalent to the vanishing of their mixed cumulants.

\begin{thm}
\label{thmequivalenceofbifreeandcombintoriallybifree}
Let $(\mathcal{A}, E, \varepsilon)$ be a $B$-$B$-non-commutative probability space and let $\{(C_k, D_k)\}_{k \in K}$ be a family of pairs of $B$-faces from $\A$.  Then $\{(C_k, D_k)\}_{k \in K}$ are bi-free with amalgamation over $B$ if and only if for all $\chi : \{1,\ldots, n\} \to \{\ell, r\}$, $\epsilon : \{1,\ldots, n\} \to K$, and 
\[
T_k \in \left\{
\begin{array}{ll}
C_{\epsilon(k)} & \mbox{if } \chi(k) = \ell  \\
D_{\epsilon(k)} & \mbox{if } \chi(k) = r
\end{array} \right.
\]
the equation
\[\kappa_{1_\chi}(T_1, \ldots, T_n) = 0\]
holds unless $\epsilon$ is constant.
\end{thm}

\begin{proof}
Suppose $\{(C_k, D_k)\}_{k \in K}$ are bi-free over $B$.  Fix a shading $\epsilon : \{1,\ldots, n\} \to K$ and let $\chi : \{1, \ldots, n\} \to \set{\ell, r}$.
If $T_1, \ldots, T_n$ are operators as above, Theorem \ref{thm:bifreeequivalenttouniversalpolys} implies
\[
	\mc E_{1_\chi}\paren{ T_1, \ldots, T_n}
= \sum_{{\pi \in BNC(\alpha)}} \left[  \sum_{\substack{\sigma \in BNC(\chi) \\ \pi \leq \sigma \leq \epsilon}} \mu_{BNC}(\pi, \sigma) \right] \mc E_{\pi}\paren{ T_1, \ldots, T_n}.
\]
Therefore
\[
\mc E_{1_\chi}\paren{ T_1, \ldots, T_n } = \sum_{\substack{\sigma \in BNC(\chi) \\ \sigma \leq \epsilon}} \kappa_\sigma\paren{T_1, \ldots, T_n}
\]
by Definition \ref{def:kappa}.
Using the above formula, we will proceed inductively to show that $\kappa_\sigma\paren{T_1, \ldots, T_n} = 0$ if $\sigma \in BNC(\chi)$ and $\sigma \nleq \epsilon$.
The base case where $n = 1$ is immediate.

For the inductive case, suppose the result holds for every $q < n$.
Suppose $\e$ is not constant and note $1_\chi \nleq \epsilon$.
Then
\[
\sum_{\sigma \in BNC(\chi)} \kappa_\sigma\paren{ T_1, \ldots, T_n }
= \mc E_{1_\chi}\paren{ T_1, \ldots, T_n }
= \sum_{\substack{\sigma \in BNC(\chi) \\ \sigma \leq \epsilon}} \kappa_\sigma\paren{T_1, \ldots, T_n}.
\]
On the other hand, by induction and the recursive properties of bi-multiplicative functions, $\kappa_\sigma\paren{T_1, \ldots, T_n} = 0$ provided $\sigma \in BNC(\chi)\setminus \set{1_\chi}$ and $\sigma \nleq \epsilon$.
Consequently,
\[
\sum_{\sigma \in BNC(\chi)} \kappa_\sigma\paren{ T_1, \ldots, T_n } = \kappa_{1_\chi}\paren{ T_1, \ldots, T_n } + \sum_{\substack{\sigma \in BNC(\chi) \\ \sigma \leq \epsilon}} \kappa_\sigma\paren{ T_1, \ldots, T_n }.
\]
Combining these two equations gives $\kappa_{1_\chi}\paren{ T_1, \ldots, T_n } = 0$, completing the inductive step.

Conversely, suppose all mixed cumulants vanish.
Then we have
\begin{align*}
\mathcal{E}_{1_\chi}\paren{ T_1, \ldots, T_n }
 &= \sum_{\sigma \in BNC(\chi)} \kappa_\sigma\paren{ T_1, \ldots, T_n }\\
 &=  \sum_{\substack{\sigma \in BNC(\chi) \\ \sigma \leq \epsilon}} \kappa_\sigma\paren{ T_1, \ldots, T_n } \\
 &=  \sum_{\substack{\sigma \in BNC(\chi) \\ \sigma \leq \epsilon}} \sum_{\substack{\pi \in BNC(\chi)\\ \pi \leq \sigma}}\mathcal{E}_\pi\paren{ T_1, \ldots, T_n } \mu_{BNC}(\pi, \sigma) \\
 &=  \sum_{{\pi\in BNC(\chi) }} \left[\sum_{\substack{\sigma \in BNC(\chi)\\ \pi \leq \sigma \leq \epsilon}} \mu_{BNC}(\pi, \sigma) \right] \mathcal{E}_\pi\paren{ T_1, \ldots, T_n }.
\end{align*}
Hence Theorem \ref{thm:bifreeequivalenttouniversalpolys} implies $\{(C_k, D_k)\}_{k \in K}$ are bi-free over $B$.
\end{proof}

\subsection{Moment and Cumulant Series}

In this section, we will begin the study of pairs of $B$-faces generated by operators.
\begin{rem}
Let $(\mathcal{A}, E, \varepsilon)$ be a $B$-$B$-non-commutative probability space and let $(C, D)$ be a pair of $B$-faces such that
\[
C = \mathrm{alg}(\{L_b \, \mid \, b \in B\} \cup \{z_i\}_{i \in I}\})
\qquad
\text{and}
\qquad
D = \mathrm{alg}(\{R_b \, \mid \, b \in B\} \cup \{z_j\}_{j \in J}\}).
\]
We desire to compute the joint distribution of $(C, D)$ under $E$.
Clearly it suffices to compute
\[
E((C_{b_1} z_{\alpha(1)} C_{b'_1}) \cdots (C_{b_n} z_{\alpha(n)} C_{b'_n})) = \mathcal{E}_{1_\alpha}(C_{b_1} z_{\alpha(1)} C_{b'_1}, \ldots, C_{b_n} z_{\alpha(n)} C_{b'_n})
\]
where $\alpha : \{1,\ldots, n\} \to I \sqcup J$, $b_1, b'_1, \ldots, b_n, b'_n \in B$, and $C_{b_k} = L_{b_k}$, $C_{b'_k} = L_{b'_k}$ if $\alpha(k) \in I$ and $C_{b_k} = R_{b_k}$, $C_{b'_k} = R_{b'_k}$ if $\alpha(k) \in J$.
However, because $\mathcal{E}$ is bi-multiplicative we can reduce this to computing
\[
\mathcal{E}_{1_{\chi_\alpha}}(z_{\alpha(1)} C_{b_1},  z_{\alpha(2)} C_{b_2}, \ldots,  z_{\alpha(n-1)} C_{b_{n-1}},  z_{\alpha(n)}).
\]
Similarly, to compute all possible cumulants, it suffices to compute
\[
\kappa_{1_{\chi_\alpha}}(z_{\alpha(1)} C_{b_1},  z_{\alpha(2)} C_{b_2}, \ldots,  z_{\alpha(n-1)} C_{b_{n-1}},  z_{\alpha(n)}).
\]
As such we make the following definition.
\end{rem}

\begin{defn}
\label{defnmomentandcumulantseries}
Let $(\mathcal{A}, E, \varepsilon)$ be a $B$-$B$-non-commutative probability space and let $(C, D)$ be a pair of $B$-faces such that
\[
C = \mathrm{alg}(\{L_b \, \mid \, b \in B\} \cup \{z_i\}_{i \in I}\})
\qquad
\text{and}
\qquad
D = \mathrm{alg}(\{R_b \, \mid \, b \in B\} \cup \{z_j\}_{j \in J}\}).
\]
The \emph{moment series} of $z = ((z_i)_{i \in I}, (z_j)_{j \in J})$ is the collection of maps
\[
\{\mu^z_\alpha : B^{n-1} \to B \, \mid \, n \in \mathbb{N}, \alpha :\{1,\ldots, n\} \to I \sqcup J\}
\]
given by
\[
\mu^z_\alpha(b_1, \ldots, b_{n-1}) = \mathcal{E}_{1_{\chi_\alpha}}(z_{\alpha(1)} C_{b_1},  z_{\alpha(2)} C_{b_2}, \ldots,  z_{\alpha(n-1)} C_{b_{n-1}},  z_{\alpha(n)}),
\]
where $C_{b_k} = L_{b_k}$ if $\alpha(k) \in I$ and $C_{b_k} = R_{b_k}$ otherwise.

Similarly, the \emph{cumulant series} of $z$ is the collection of maps
\[
\{\kappa^z_\alpha : B^{n-1} \to B \, \mid \, n \in \mathbb{N}, \alpha :\{1,\ldots, n\} \to I \sqcup J\}
\]
given by
\[
\kappa^z_\alpha(b_1, \ldots, b_{n-1}) = \kappa_{1_{\chi_\alpha}}(z_{\alpha(1)} C_{b_1},  z_{\alpha(2)} C_{b_2}, \ldots,  z_{\alpha(n-1)} C_{b_{n-1}},  z_{\alpha(n)}).
\]

Note that if $n = 1$, we have $\mu_\alpha^z = E(z_{\alpha(1)}) = \kappa_\alpha^z$.
\end{defn}

\begin{prop}
Let $(\mathcal{A}, E)$ be a $B$-$B$-non-commutative probability space, and for $\epsilon\in \{',''\}$ let $\{z_i^\epsilon\}_{i\in I}\subset \A_\ell$ and $\{z_j^\epsilon\}_{j\in J}\subset \A_r$. If 
\[
C^\epsilon = \alg\paren{\set{L_b : b \in B} \cup \set{z_i^\epsilon}_{i\in I}}\qquad\text{and}\qquad
D^\epsilon = \alg\paren{\set{R_b : b \in B} \cup \set{z_j^\epsilon}_{j\in J}}
\]
are such that $(C', D')$ and $(C'', D'')$ are bi-free, then
\[
\kappa_\alpha^{z'+z''} = \kappa_\alpha^{z'}+\kappa_\alpha^{z''}.
\]
\end{prop}
\begin{proof}
This follows directly from Theorem \ref{thmequivalenceofbifreeandcombintoriallybifree}.
\end{proof}


\section{Multiplicative Convolution for Families of Pairs of $B$-Faces}
\label{sec:MultiplicativeConvolution}

In this section, we will demonstrate how operator-valued bi-free cumulants involving products of operators may be computed.
The main theorem of this section, Theorem \ref{thmadvcanedcumulantreduction}, also gives rise to an extension of \cite{CNS2014}*{Theorem 5.2.1} in the case $B = \mathbb{C}$.

\subsection{Operator-Valued Bi-Free Cumulants of Products}

\begin{nota}
Given two partitions $\pi, \sigma \in BNC(\chi)$, we let $\pi \vee \sigma$ denote the smallest element of $BNC(\chi)$ greater than $\pi$ and $\sigma$.
\end{nota}
\begin{nota}
\label{nota:expandingnodesonthesameside}
Let $m,n \in \mathbb{N}$ with $m < n$ be fixed, and fix a sequence of integers 
\[
k(0) = 0 < k(1) < \cdots < k(m) = n
\]
For $\chi : \{1,\ldots, m\} \to \{\ell, r\}$, we define $\widehat{\chi} : \{1,\ldots, n\} \to \{\ell, r\}$ via
\[
\widehat{\chi}(q) = \chi(p_q)
\]
where $p_q$ is the unique element of $\{1,\ldots, m\}$ such that $k(p_q-1) < q \leq k(p_q)$.  
\end{nota}
\begin{rem}
In the context of Notation \ref{nota:expandingnodesonthesameside}, there exists an embedding of $BNC(\chi)$ into $BNC(\widehat{\chi})$ via $\pi \mapsto \widehat{\pi}$ where the $p^{\mathrm{th}}$ node of $\pi$ is replaced by the block $(k(p-1)+1, \ldots, k(p))$.
Observe that all of these nodes appear on the side $p$ was on originally.
Alternatively, this map can be viewed as an analogue of the map on non-crossing partitions from \cite{NSBook}*{Notation 11.9} after applying $s^{-1}_\chi$.

It is easy to see that $\widehat{1_\chi} = 1_{\widehat{\chi}}$, $\widehat{0_\chi}$ is the partition with blocks $\{(k(p-1)+1, \ldots, k(p))  \}_{p=1}^m$, and $\pi \mapsto \widehat{\pi}$ is injective and preserves the partial ordering on $BNC$.  Furthermore the image of $BNC(\chi)$ under this map is
\[
\widehat{BNC}(\chi) = \left[\widehat{0_\chi}, \widehat{1_\chi}\right] = \left[\widehat{0_\chi}, 1_{\widehat{\chi}}\right] \subseteq BNC(\widehat{\chi}).
\]
Finally, since the lattice structure is preserved by this map, we see that $\mu_{BNC}(\sigma, \pi) = \mu_{BNC}(\widehat{\sigma}, \widehat{\pi})$.
\end{rem}

\begin{rem}
\label{rem:partial-mobius-inversion}
Recall that since $\mu_{BNC}$ is the M\"{o}bius function on the lattice of bi-non-crossing partitions, we have for each $\sigma,\pi \in BNC(\chi)$ with $\sigma \leq \pi$ that
\[
\sum_{\substack{ \tau \in BNC(\chi) \\ \sigma \leq \tau \leq \pi  }} \mu_{BNC}(\tau, \pi) =  \left\{
\begin{array}{ll}
1 & \mbox{if } \sigma = \pi  \\
0 & \mbox{otherwise }
\end{array} \right. .
\]
Therefore, it is easy to see that the partial M\"{o}bius inversion from \cite{NSBook}*{Proposition 10.11} holds in our setting; that is, if $f, g : BNC(\chi) \to B$ are such that 
\[
f(\pi) = \sum_{\substack{\sigma \in BNC(\chi) \\ \sigma \leq \pi}} g(\sigma)
\]
for all $\pi \in BNC(\chi)$, then for all $\pi, \sigma \in BNC(\chi)$ with $\sigma \leq \pi$, we have the relation
\[
\sum_{\substack{\tau \in BNC(\chi) \\ \sigma \leq \tau \leq \pi }} f(\tau) \mu_{BNC}(\tau, \pi) = \sum_{\substack{\omega \in BNC(\chi) \\ \omega \vee \sigma = \pi }} g(\omega).
\]
\end{rem}

We now describe the operator-valued bi-free cumulants involving products of operators in terms of the above notation, following the spirit of \cite{NSBook}*{Theorem 11.12}.
\begin{thm}
\label{thmadvcanedcumulantreduction}
Let $(\mathcal{A}, E, \varepsilon)$ be a $B$-$B$-non-commutative probability space, $m,n \in \mathbb{N}$ with $m < n$, $\chi : \{1,\ldots, m\} \to \{\ell, r\}$, and
\[
k(0) = 0 < k(1) < \cdots < k(m) = n.
\]
If $\pi \in BNC(\chi)$ and $T_k \in \mathcal{A}_{\widehat{\chi}(k)}$ for all $k \in \{1,\ldots, n\}$, then
\[
\kappa_\pi\left(T_1 \cdots T_{k(1)}, T_{k(1)+1} \cdots T_{k(2)}, \ldots, T_{k(m-1)+1} \cdots T_{k(m)} \right) =  \sum_{\substack{\sigma \in BNC(\widehat{\chi})\\ \sigma \vee \widehat{0_\chi} = \widehat{\pi}}} \kappa_\sigma(T_1, \ldots, T_n).
\]
In particular, for $\pi = 1_\chi$, we have
\[
\kappa_{1_\chi}\left(T_1 \cdots T_{k(1)}, T_{k(1)+1} \cdots T_{k(2)}, \ldots, T_{k(m-1)+1} \cdots T_{k(m)}\right) = \sum_{\substack{\sigma \in BNC(\widehat{\chi})\\ \sigma \vee \widehat{0_\chi} = 1_{\widehat{\chi}}}} \kappa_\sigma(T_1, \ldots, T_n).
\]
\end{thm}
\begin{proof}
For $j \in \{1,\ldots, m\}$, let $S_j = T_{k(j-1)+1} \cdots T_{k(j)}$.  Then, by Definition \ref{defn:recursivedefinitionofEpi},
\begin{align*}
\kappa_\pi(S_1, \ldots, S_m) &= \sum_{\substack{\tau \in BNC(\chi) \\ \tau \leq \pi}} \mathcal{E}_\tau(S_1, \ldots, S_m) \mu_{BNC}(\tau, \pi) \\
& = \sum_{\substack{\tau \in BNC(\chi)\\  \tau \leq \pi}} \mathcal{E}_{\widehat{\tau}}(T_1, \ldots, T_n) \mu_{BNC}(\widehat{\tau}, \widehat{\pi}) \\
& = \sum_{\substack{\sigma \in BNC(\widehat{\chi}) \\ \widehat{0_\chi} \leq \sigma \leq \widehat{\pi}}} \mathcal{E}_{\sigma}(T_1, \ldots, T_n) \mu_{BNC}(\sigma, \widehat{\pi}) \\
& = \sum_{\substack{\sigma \in BNC(\widehat{\chi})\\ \sigma \vee \widehat{0_\chi} = \widehat{\pi}}} \kappa_\sigma(T_1, \ldots, T_n),
\end{align*}
with the last line following from Remark \ref{rem:partial-mobius-inversion}.
\end{proof}

\subsection{Multiplicative Convolution of Bi-Free Two-Faced Families}

Recall from \cite{CNS2014}*{Definition 5.1.1} that given any $\chi : \{1, \ldots, n\} \to \{\ell, r\}$ the (below on the left, above on the right) Kreweras complement of a bi-non-crossing partition $\pi \in BNC(\chi)$ is the element $K_{BNC}(\pi)$ of $BNC(\chi)$ obtained by applying $s_\chi$ to the (right) Kreweras complement in $NC(n)$ of $s^{-1}_\chi \cdot \pi$; that is,
\[
K_{BNC}(\pi) := s_\chi \cdot K_{NC}(s^{-1}_\chi \cdot \pi).
\]
Using this Kreweras complement, \cite{CNS2014}*{Theorem 5.2.1} constructed the multiplicative convolution of bi-free two-faced families where each pair of faces had a single left operator and a single right operator.  Using Theorem \ref{thmadvcanedcumulantreduction}, we can extend \cite{CNS2014}*{Theorem 5.2.1} as follows, using the proof of \cite{NSBook}*{Theorem 14.4}.

\begin{prop}
\label{prop:multiplcative-convolution}
Let $(\mathcal{A}, \varphi)$ be a non-commutative probability space.
Let $z'=( (z'_i)_{i \in I}, (z'_j)_{j \in J})$ and $z''=( (z''_i)_{i \in I}, (z''_j)_{j \in J})$ be bi-free two-faced families in $\mathcal{A}$, and set $z_i = z'_iz''_i$, $z_j = z''_jz'_j$ for $i \in I$ and $j \in J$.
Then for $\alpha : \{1,\ldots, n\} \to I \sqcup J$, we have
\[
\kappa_\chi(z_{\alpha(1)}, \ldots, z_{\alpha(n)}) = \sum_{\pi \in BNC(\chi_\alpha)} \kappa_\pi(z'_{\alpha(1)}, \ldots, z'_{\alpha(n)}) \cdot \kappa_{K_{BNC}(\pi)}(z''_{\alpha(1)}, \ldots, z''_{\alpha(n)}).
\]
\end{prop}
\begin{proof}
Define $\widehat{\alpha} : \{1,\ldots, 2n\} \to I \sqcup J$ by $\widehat{\alpha}(2k-1) = \widehat{\alpha}(2k) = \alpha(k)$ for $k \in \{1,\ldots, n\}$, and define $\epsilon : \{1,\ldots, 2n\} \to \{', ''\}$ by
\[
\epsilon(2k-1)  = \left\{
\begin{array}{ll}
' & \mbox{if } \alpha(k) \in I  \\
'' & \mbox{if } \alpha(k) \in J
\end{array} \right.
 \qquad \mbox{ and }\qquad \epsilon(2k) = 
  \left\{
\begin{array}{ll}
'' & \mbox{if } \alpha(k) \in I  \\
' & \mbox{if } \alpha(k) \in J
\end{array} \right..
\]
Using Theorem \ref{thmadvcanedcumulantreduction}, we easily obtain
\[
\kappa_\chi(z_{\alpha(1)}, \ldots, z_{\alpha(n)}) = \sum_{\substack{\pi \in BNC(\chi_{\widehat{\alpha}}) \\ \pi \vee \sigma = 1_{\chi_{\widehat{\alpha}}}}} \kappa_\pi\left(z^{\epsilon(1)}_{\alpha(1)}, z^{\epsilon(2)}_{\alpha(1)}, \ldots, z^{\epsilon(2n-1)}_{\alpha(n)}, z^{\epsilon(2n)}_{\alpha(n)}\right)
\]
where $\sigma = \{(1,2), (3,4), \ldots, (2n-1, 2_n)\}$.  Since $z'$ and $z''$ are bi-free, Theorem \ref{thmequivalenceofbifreeandcombintoriallybifree} (or simply \cite{CNS2014}*{Theorem 4.3.1}) implies mixed bi-free cumulants vanish and thus only $\pi$ of the form $\pi = \pi' \cup \pi''$ with $\pi' \in BNC\left( \chi_{\widehat{\alpha}}|_{\{k \, \mid \, \epsilon(k) = '\}}\right)$ and $\pi'' \in BNC\left( \chi_{\widehat{\alpha}}|_{\{k \, \mid \, \epsilon(k) = ''\}}\right)$ will provide a non-zero contribution.  However, for an arbitrary $\pi' \in BNC\left( \chi_{\widehat{\alpha}}|_{\{k \, \mid \, \epsilon(k) = '\}}\right)$, it is elementary to see (for example, by the relations between the Kreweras complements for bi-non-crossing and non-crossing partitions) that there exists a unique $\pi'' \in BNC\left( \chi_{\widehat{\alpha}}|_{\{k \, \mid \, \epsilon(k) = ''\}}\right)$ such that $(\pi' \cup \pi'') \vee \sigma = 1_{\chi_{\widehat{\alpha}}}$; namely, $\pi'' = K_{BNC}(\pi')$.  Therefore, since we are in the scalar setting and $\kappa_\tau(T_1,\ldots, T_n) = \kappa_{\tau|_{V}}((T_1,\ldots, T_n)|_V) \kappa_{\tau|_{V^c}}((T_1,\ldots, T_n)|_{V^c})$ whenever $\tau \in BNC(\chi')$ and $V$ is a block of $\pi$, we obtain the desired equation.
\end{proof}

\section{Additional Examples of Bi-Free Families with Amalgamation}
\label{sec:howtogetexamplesofbifreefamilies}

In this section, we will demonstrate some useful techniques for constructing bi-free pairs of $B$-faces.

\subsection{Conjugation by Haar Bi-Unitaries}

\begin{defn}
\label{defnconcretebihaar}
Consider the $B$-$B$-bimodule $\ell^2(\mathbb{Z}, B) = (B \delta_0) \oplus (B\delta_0)^\bot$.  A \emph{concrete $B$-valued Haar bi-unitary} is the invertible element
\[
U_{\mathrm{Haar}} \in \mathcal{L}(\ell^2(\mathbb{Z}, B))
\]
defined by
\[
U_{\mathrm{Haar}}((b_k)_{k \in K}) = (b_{k+1})_{k \in K}
\]
for all $(b_k)_{k \in K} \in \ell^2(\mathbb{Z}, B)$.  The moment series of
\[
z_{\mathrm{Haar}} = (\{U_{\mathrm{Haar}}, U_{\mathrm{Haar}}^{-1}\}, \{U_{\mathrm{Haar}}, U_{\mathrm{Haar}}^{-1}\})
\]
will be called the \emph{$B$-valued Haar bi-unitary moment series}.
\end{defn}
\begin{defn}
Let $(\mathcal{A}, E, \varepsilon)$ be a $B$-$B$-non-commutative probability space.  An \emph{abstract $B$-valued Haar bi-unitary} in $\mathcal{A}$ is a pair of invertible elements $U_\ell \in \mathcal{A}_\ell$ and $U_r \in \mathcal{A}_r$ such that the moment series of $(\{(U_\ell, U^{-1}_\ell\}, \{U_r, U^{-1}_r\})$ is equal to the $B$-valued Haar bi-unitary moment series.  We will call 
\[
(\mathrm{alg}(\varepsilon(B \otimes 1_B), U_\ell, U_\ell^{-1}), \mathrm{alg}(\varepsilon(1_B \otimes B^{\text{op}}), U_r, U_r^{-1}))
\]
the \emph{pair of $B$-faces generated by $(U_\ell, U_r)$}.
\end{defn}
\begin{thm}
\label{thmconjugatingbyHaarunitary}
Let $(\mathcal{A}, E, \varepsilon)$ be a $B$-$B$-non-commutative probability space, let $(C, D)$ be a pair of $B$-faces in $(\mathcal{A},E)$, and let $(U_\ell, U_r)$ be a $B$-valued Haar bi-unitary in $(\mathcal{A},E, \varepsilon)$ such that $(C, D)$ is bi-free with amalgamation over $B$ from the pair of $B$-faces generated by $(U_\ell, U_r)$.
Then the pairs of $B$-faces $(C, D)$ and $(U^{-1}_\ell C U_\ell, U^{-1}_r DU_r)$ are bi-free over $B$ and the joint distribution of $(U^{-1}_\ell CU_\ell, U^{-1}_r DU_r)$ is equal to the joint distribution of $(C, D)$.
\end{thm}

\begin{proof}
Let $(\X, \mathring{\X}, \xi)$ be a $B$-bimodule that realizes $(C, D)$ with expectation $E$ (see Theorem \ref{thm:representingbbncps}).
Since $(C, D)$ and the pair of $B$-faces generated by $(U_\ell, U_r)$ are bi-free over $B$, their joint distribution may be computed via the vacuum state with respect to their canonical actions on 
\[
(\X, \mathring{\X}, \xi) \ast_B (\ell^2(\mathbb{Z}, B), (B\delta_0)^\bot, \delta_0),
\]
where both $U_\ell$ and $U_r$ act on $\ell^2(\mathbb{Z}, B)$ as $U_{\mathrm{Haar}}$.
\par 
Let $\mathcal{Y}$ denote the $B$-submodule of $(\X, \mathring{\X}, \xi) \ast_B (\ell^2(\mathbb{Z}, B), (B\delta_0)^\bot, \delta_0)$ generated by all vectors of the form
\[
x_0 \otimes (\delta_{-1} \otimes x'_1 \otimes \delta_{1}) \otimes x_1 \otimes (\delta_{-1} \otimes x'_2 \otimes \delta_{1}) \otimes x_2 \otimes \cdots \otimes (\delta_{-1} \otimes x'_n \otimes \delta_{1}) \otimes x_n
\]
where $n \geq 0$, $x_1, \ldots, x_{n-1}, x'_1, \ldots, x'_n \in \mathring{\X}$, and $x_0, x_n \in \X$.  It is straightforward to verify that $\mathcal{Y}$ is invariant under the actions of both $(C, D)$ and $(U^{-1}_\ell CU_\ell, U^{-1}_r DU_r)$.
As the vacuum vector is itself an element of $\mathcal Y$, the joint distribution of $(C, D)$ and $(U^{-1}_\ell CU_\ell, U^{-1}_r DU_r)$ may be computed purely by examining their actions on $\mathcal Y$.

To show that $(C, D)$ and $(U^{-1}_\ell CU_\ell, U^{-1}_r DU_r)$ are bi-free with amalgamation over $B$ and have the same joint distribution, we show that the joint distribution of $(C, D)$ and $(U^{-1}_\ell CU_\ell, U^{-1}_r DU_r)$ is the same as that of two bi-free copies of $(C, D)$ acting on a reduced free product space.

Let $(C_1, D_1)$ and $(C_2, D_2)$ be copies of $(C, D)$ acting on copies of $(\X, \mathring{\X}, \xi)$ given by $(\X_1, \mathring{\X}_1, \xi_1)$ and $(\X_2, \mathring{\X}_2, \xi_2)$, respectively.  Thus $(C_1, D_1) \ast \ast_B (C_2, D_2)$ has an induced action on $(\X_1, \mathring{\X}_1, \xi_1) \ast_B (\X_2, \mathring{\X}_2, \xi_2)$.

Consider the map 
\[
\Phi : (\X_1,\mathring{\X}_1, \xi_1) \ast_B (\X_2, \mathring{\X}_2, \xi_2) \to \mathcal{Y}
\]
defined as follows: for $n \geq 0$, $x_1, \ldots, x_{n-1} \in \mathring{\X}_1$, $x'_1, \ldots, x'_n \in \mathring{\X}_2$, and $x_0, x_n \in \X_1$, we define
\[
\Phi(x_0 \otimes x'_1 \otimes x_1 \otimes \cdots \otimes x'_n \otimes x_n) = x_0 \otimes (\delta_{-1} \otimes x'_1 \otimes \delta_{1}) \otimes x_1 \otimes \cdots \otimes (\delta_{-1} \otimes x'_n \otimes \delta_{1}) \otimes x_n
\]
It is elementary to verify that $\Phi$ is a $B$-bimodule isomorphism that sends the vacuum vector to the vacuum vector in $\mathcal{Y}$, that if $T \in C \cup D$ then viewing $T \in C_1 \cup D_1$,
\[
\Phi(T \eta) = T \eta
\]
for all $\eta$, and that viewing $T \in C_2 \cup D_2$,
\[
\Phi(T \eta) = U^{-1}_n TU_n \eta
\]
where $n = \ell$ if $T \in B_2$ and $n = r$ if $T \in C_2$.
The result follows.
\end{proof}
\begin{rem}
Given a two-faced family $(C, D)$ in a non-commutative probability space $(\A, \varphi)$, we may always enlarge $\A$ via the operator model from \cite{CNS2014}*{Theorem 6.4.1} to include a $\mathbb{C}$-valued Haar bi-unitary $(U_\ell, U_r)$ bi-free from $(C, D)$.
\end{rem}

\subsection{Bi-Free Families where Left and Right Faces Commute}

The following result demonstrates that if all of the left faces commute with all of the right faces, then, in certain circumstances, bi-freeness can be deduced from freeness of either the left faces or the right faces.
\begin{thm}
\label{thmcommutingandfreeimpliesbifree}
Let $(\mathcal{A},E, \varepsilon)$ be a $B$-$B$-non-commutative probability space acting on the $B$-bimodule $\X = B \oplus \mathring{\X}$ (as in Theorem \ref{thm:representingbbncps} for example) and let $\xi = 1_B \oplus 0 \in \X$.  Suppose $\{(C_k, D_k)\}_{k \in K}$ is a family of pairs of $B$-faces such that
\begin{enumerate}
\item $C_n$ and $D_m$ commute for all $n, m \in K$, and
\item for each $T \in D_k$ there exists an $S \in C_k$ such that $T \xi = S \xi$.
\end{enumerate}
Then $\{(C_k, D_k)\}_{k \in K}$ are bi-free with amalgamation over $B$ if and only if $\{C_k\}_{k \in K} $ are free with amalgamation over $B$.  Therefore, if $\{C_k\}_{k \in K} $ are free with amalgamation over $B$ then $\{D_k\}_{k \in K} $ are free with amalgamation over $B$.
\end{thm}

\begin{proof}
If $\{(C_k, D_k)\}_{k \in K}$ are bi-free over $B$ then it is clear that $\{C_k\}_{k \in K} $ are free over $B$ and $\{D_k\}_{k \in K} $ are free over $B$.
\par 
Suppose $\{C_k\}_{k \in K}$ are free over $B$.  To show that $\{(C_k, D_k)\}_{k \in K}$ are bi-free over $B$ we need to verify the operator polynomials
\[
E\left(T_1 \cdots T_n\right)
= \sum_{\pi\in BNC(\chi)} \left[\sum_{\substack{\sigma\in BNC(\chi)\\ \pi\leq\sigma\leq \epsilon}} \mu_{BNC}(\pi, \sigma)\right] E_\pi(T_1, \ldots, T_n)
\]
whenever $\chi : \set{1, \ldots, n} \to \set{\ell, r}$, $\epsilon : \set{1, \ldots, n} \to K$, and $T_k \in C_{\epsilon(k)}$ if $\chi(k) = \ell$, and $T_k \in D_{\epsilon(k)}$ if $\chi(k) = r$.
Note if $\chi^{-1}(\{\ell\}) = \{1,\ldots, n\}$, the freeness of $\{C_k\}_{k \in K}$ implies that these polynomials hold.
Thus we will proceed by induction on $|\chi^{-1}(\{r\})|$ with the base case of $|\chi^{-1}(\{r\})|=0$ complete.

Let $s$ be the permutation such that 
\[
\chi^{-1}\paren{\set\ell} = \set{s(1) < \ldots < s(k)} \qquad \mbox{and} \qquad \chi^{-1}\paren{\set r} = \set{s(k+1) < \ldots < s(n)}.
\]
Let $\hat\e = \e\circ s$, and $\hat\chi = \chi\circ s$. 
Note that replacing $\chi$ by $\hat\chi$ and $\e$ by $\hat\e$ corresponds to moving the right nodes beneath the left ones, without changing their relative order.
In particular, due to commutation, $T_{s(1)} \cdots T_{s(n)} = T_1 \cdots T_n$ and bi-multiplicativity implies
\[
E\left(T_1 \cdots T_n\right)
= \sum_{\pi\in BNC(\chi)} \left[\sum_{\substack{\sigma\in BNC(\chi)\\ \pi\leq\sigma\leq \epsilon}} \mu_{BNC}(\pi, \sigma)\right] E_\pi(T_1, \ldots, T_n)
\]
if and only if
\[
E\left(T_{s(1)} \cdots T_{s(n)}\right)
= \sum_{\pi\in BNC(\hat\chi)} \left[\sum_{\substack{\sigma\in BNC(\hat\chi)\\ \pi\leq\sigma\leq \hat\e}} \mu_{BNC}(\pi, \sigma)\right] E_\pi(T_{s(1)}, \ldots, T_{s(n)}).
\]

To reduce the number of right operators, we note that $ \hat{\chi}(n) = r$ and select $S \in C_{\hat{\e}(n)}$ such that $S\xi = T_{s(n)} \xi$.
Define $\overline{\chi} : \{1,\ldots, n\} \to \{\ell, r\}$ by 
\[
\overline{\chi}(p) = \left\{
\begin{array}{ll}
\hat{\chi}(p) & \mbox{if } p < n  \\
\ell & \mbox{if } p = n
\end{array} \right. .
\]
Clearly there is a canonical map from $BNC(\hat{\chi})$ to $BNC(\overline{\chi})$ that takes $\pi \in BNC(\hat{\chi})$ and produces $\overline{\pi} \in BNC(\overline{\chi})$ by moving the bottom node of $\pi$ from a right node to a left node while keeping all strings attached.
Note that
\[
\sum_{\substack{\sigma\in BNC(\hat{\chi})\\ \pi\leq\sigma\leq \hat{\epsilon}}} \mu_{BNC}(\pi, \sigma) = \sum_{\substack{\sigma\in BNC(\overline{\chi})\\ \overline{\pi}\leq\overline{\sigma}\leq \hat{\epsilon}}} \mu_{BNC}(\overline{\pi}, \overline{\sigma}).
\]
Moreover, $S\xi = T_{s(n)} \xi$ implies that for all $\pi \in BNC(\hat\chi)$,
\[
E_\pi(T_{s(1)}, \ldots, T_{s(n)}) = E_{\overline{\pi}} (T_{s(1)}, \ldots, T_{s(n-1)}, S).
\]
Then we see that
\[
E\left(T_{s(1)} \cdots T_{s(n)}\right)
= \sum_{\pi\in BNC(\hat\chi)} \left[\sum_{\substack{\sigma\in BNC(\hat\chi)\\ \pi\leq\sigma\leq \paren{\hat\e}}} \mu_{BNC}(\pi, \sigma)\right] E_\pi(T_{s(1)}, \ldots, T_{s(n)}).
\]
if and only if
\[
E\left(T_{s(1)} \cdots T_{s(n-1)} S\right) = \sum_{\overline{\pi}\in BNC(\overline{\chi})} 
\left[\sum_{\substack{\sigma\in BNC(\overline{\chi})\\ \overline{\pi}\leq\overline{\sigma}\leq \hat{\e}}} \mu_{BNC}(\overline{\pi}, \overline{\sigma})\right] 
E_{\overline{\pi}} (T_{s(1)}, \ldots, T_{s(n-1)}, S).
\]
Since the last equation holds by our inductive hypothesis, the proof is complete.
\end{proof}

\begin{cor}
Let $(\mathcal{A},E, \varepsilon)$ be a $B$-$B$-non-commutative probability space acting on the $B$-bimodule $\X = B \oplus \mathring{\X}$ and let $\xi = 1_B \oplus 0 \in \X$.  Suppose $\{(C_k, D_k)\}_{k \in K}$ is a family of pairs of $B$-faces such that
\begin{enumerate}
\item $C_n$ and $D_m$ commute for all $n, m \in K$, and
\item $\set{T\xi : T \in C_k} = \set{S\xi : S \in D_k}$ for all $k \in K$.
\end{enumerate}
Then $\{(C_k, D_k)\}_{k \in K}$ are bi-free over $B$ if and only if $\{C_k\}_{k \in K} $ are free over $B$ if and only if $\{D_k\}_{k \in K} $ are free over $B$.
\end{cor}
\begin{cor}
\label{corcommutingandfreeimpliesbifree}
Let $(\mathcal{A},\varphi)$ be a non-commutative probability space acting on a pointed vector space $(X, \xi)$ with $\xi$ realizing $\varphi$.  Suppose $\{(C_k, D_k)\}_{k \in K}$ is a family of pairs of faces such that
\begin{enumerate}
\item $C_n$ and $D_m$ commute for all $n, m \in K$, and
\item for each $T \in D_k$ there exists an $S \in C_k$ such that $T \xi = S \xi$.
\end{enumerate}
Then $\{(C_k, D_k)\}_{k \in K}$ are bi-free if and only if $\{C_k\}_{k \in K} $ are free.  Therefore, if $\{C_k\}_{k \in K} $ are free then $\{D_k\}_{k \in K} $ are free.
\end{cor}
The above is particularly interesting as it enables the transference of freeness from one algebra to another.


\begin{thebibliography}{13}
\bib{CNS2014}{article}{
    author={Charlesworth, I.},
    author={Nelson, B.},
    author={Skoufranis, P.},
    title={On Two-Faced Families of Non-Commutative Random Variables},
    eprint={arXiv:1403.4907},
    year={2014},
    pages={26}
}




\bib{MN2013}{article}{
    author={Mastank, M.},
    author={Nica, A.},
    title={Double-ended queues and joint moments of left-right canonical operators on full Fock space},
    eprint={arXiv:1312.0269},
    year={2013},
    pages={28}
}



\bib{NSS2002}{article}{
    author={Nica, A.},
    author={Shylakhtenko, D.},
    author={Speicher, R.},
    title={Operator-valued distributions: I. Characterizations of freeness},
    journal={Int. Math. Res. Notices},
    volume={29},
    year={2002},
    pages={1509-1538}
}

\bib{NSBook}{book}{
	author={Nica, A.},
	author={Speicher, R.},
	title={Lectures on the Combinatorics of Free Probability}, 
	series={London Mathematics Society Lecture Notes Series},
	volume={335},
	publisher={Cambridge University Press},
	date={2006}
}



\bib{S1998}{book}{
	author={Speicher, R.},
	title={Combinatorial theory of the free product with amalgamation and operator-valued free probability theory}, 
	series={Memoirs of the American Mathematical Society},
	volume={627},
	publisher={American Mathematical Soc.},
	date={1998}
}



\bib{V1995}{article}{
    author={Voiculescu, D. V.},
    title={Operations on certain non-commutative operator-valued random variables},
    journal={Ast\'{e}risque},
    volume={232},
    year={1995},
    pages={243-375}
}



\bib{V2013-1}{article}{
    author={Voiculescu, D. V.},
    title={Free Probability for Pairs of Faces I},
    eprint={arXiv:1306.6082},
    year={2013},
    pages={33}
}




\end{thebibliography}
\end{document}